\documentclass[11pt,reqno]{amsart}

\usepackage{amssymb}
\usepackage{amscd}
\usepackage{amsfonts}
\usepackage{mathrsfs}
\usepackage{setspace}
\usepackage{version}
\usepackage{mathtools}

\numberwithin{equation}{section} 
\theoremstyle{plain}
\newtheorem{theorem}[subsection]{Theorem}
\newtheorem{proposition}[subsection]{Proposition}
\newtheorem{lemma}[subsection]{Lemma}
\newtheorem{corollary}[subsection]{Corollary}
\newtheorem{definition}[subsection]{Definition}

\newtheorem*{mainthm3-repeat}{Theorem \ref{mainthm3}}
\newtheorem*{nontriv-source-rpt}{Proposition \ref{nontriv-source}}
\newtheorem*{sum-product-rpt}{Theorem \ref{sum-product-fp}}
\newtheorem*{mainthm-repeat}{Theorem \ref{mainthm2}}
\newtheorem*{virtually-solvable-repeat}{Lemma \ref{virtually-solvable}}

\renewcommand{\leq}{\leqslant}
\renewcommand{\geq}{\geqslant}

\newcommand{\md}[1]{\ensuremath{(\operatorname{mod}\, #1)}}

\newcommand\E{\mathbb{E}}
\newcommand\Var{\operatorname{Var}}
\newcommand\Z{\mathbb{Z}}
\newcommand\R{\mathbb{R}}

\newcommand\C{\mathbb{C}}
\newcommand\N{\mathbb{N}}

\newcommand\X{\mathbf{X}}

\newcommand\A{\mathbb{A}}

\newcommand\TV{{\operatorname{TV}}}

\newcommand\vol{\operatorname{vol}}

\newcommand\Lip{{\operatorname{Lip}}}

\newcommand\waste{\operatorname{waste}}
\newcommand\Err{\operatorname{Err}}
\newcommand\Energy{\operatorname{Energy}}

\newcommand\BQ{\operatorname{BQ}}
\newcommand\VBQ{\operatorname{VBQ}}

\renewcommand\P{\mathbb{P}}

\newcommand\F{\mathbb{F}}

\renewcommand\b{{\bf b}}
\newcommand\p{\mathfrak{p}}

\newcommand\n{\mathbf{n}}
\newcommand\m{\mathbf{m}}
\newcommand\Nb{\mathbf{N}}
\newcommand\h{\mathbf{h}}
\newcommand\I{\mathbf{I}}

\renewcommand\k{\mathbf{k}}
\renewcommand\l{\mathbf{l}}
\renewcommand\a{\mathbf{a}}
\renewcommand\A{\mathbf{A}}
\renewcommand\r{\mathbf{r}}
\renewcommand\c{\mathbf{c}}
\newcommand\f{\mathbf{f}}
\newcommand\x{\mathbf{x}}
\renewcommand\F{\mathbf{F}}
\newcommand\lambdab{{\boldsymbol\lambda}}

\newcommand\Xib{{\boldsymbol\Xi}}

\begin{document}

\title[A new bound for $r_4(N)$]{New bounds for Szemer\'edi's theorem, III: A polylogarithmic bound for $r_4(N)$}

\author{Ben Green}
\address{Mathematical Institute, Andrew Wiles Building, Radcliffe Observatory Quarter, Woodstock Rd, Oxford OX2 6GG.}
\email{ben.green@maths.ox.ac.uk}

\author{Terence Tao}
\address{Department of Mathematics, UCLA, Los Angeles CA 90095-1555, USA.
}
\email{tao@math.ucla.edu}

\begin{abstract} 
Define $r_4(N)$ to be the largest cardinality of a set $A \subset \{1,\dots,N\}$ which does not contain four elements in arithmetic progression. In 1998 Gowers proved that
\[ r_4(N) \ll N(\log \log N)^{-c}\] for some absolute constant $c>0$. In 2005, the authors improved this to
\[ r_4(N) \ll N e^{-c\sqrt{\log\log N}}.\]
In this paper we further improve this to \[ r_4(N) \ll N(\log N)^{-c},\] which appears to be the limit of our methods.  
\end{abstract}

\maketitle

\setcounter{tocdepth}{1}
\tableofcontents

\section{Introduction}

Let $N \geq 100$ be a natural number (so that $\log \log N$ is positive).  If $k \geq 3$ is a natural number we define $r_k(N)$ to be the largest cardinality of a set $A \subset [N] \coloneqq  \{1,\dots,N\}$ which does not contain an arithmetic progression of $k$ distinct elements.  

Klaus Roth proved in 1953 \cite{roth} that $r_3(N) \ll N(\log \log N)^{-1}$, and so in particular\footnote{See Section \ref{notation-sec} for the asymptotic notation used in this paper.} $r_3(N) = o(N)$ as $N \to \infty$.
Since Szemer\'edi's 1969 proof  \cite{szemeredi-4} that $r_4(N) = o(N)$, and his later proof \cite{szemeredi} that $r_k(N) = o_k(N)$ for $k \geq 5$ (answering a question from \cite{erdos-turan}), it has been natural to ask for similarly effective bounds for these quantities.  It is worth noting that the famous conjecture of Erd\H{o}s \cite{erdos} asserting that every set of natural numbers whose sum of reciprocals is divergent is equivalent to the claim that $\sum_{n=1}^\infty \frac{r_k(2^n)}{2^n} < \infty$ for all $k \geq 3$ (see \cite[Exercise 10.0.6]{taovu-book}).  

A first attempt towards quantitative bounds for higher $k$ was made by Roth in \cite{roth-4}, who provided a new proof that $r_4(N) = o(N)$. A major breakthrough was made in 1998 by Gowers \cite{gowers-4-aps,gowers-long-aps}, who obtained the bound $r_k(N) \ll_k N(\log \log N)^{-\epsilon_k}$ for each $k \geq 4$, where $\epsilon_k \coloneqq 1/2^{2^{k+9}}$.   In the other direction, a classical result of Behrend \cite{behrend} shows that $r_3(N) \gg N \exp( - c \sqrt{\log N} )$ for some absolute constant $c>0$ (see \cite{elkin,gw} for a slight refinement of this bound), and in \cite{rankin} (see also \cite{laba}) the argument was generalised to give the bound $r_{1+2^k}(N) \gg_k N \exp( - c \log^{1/(k+1)} N)$ for any $k \geq 1$.

In the meantime, there has been progress on $r_3(N)$. Szemer\'edi (unpublished) obtained the bound $r_3(N) \ll Ne^{-c\sqrt{\log \log N}}$, and shortly thereafter Heath-Brown \cite{heath} and Szemer\'edi \cite{szem-3ap} independently obtained the bound $r_3(N) \ll N(\log N)^{-c}$ for some absolute constant $c > 0$. The best known value of $c$ has been improved in a series of papers \cite{bloom,bourgain, Bou-2, Sanders,Sanders-2}. Sanders \cite{Sanders-2} was the first to show that any $c < 1$ is admissible, and Bloom \cite{bloom} improved the factor of $\log \log N$ in Sanders's bound.

The only other direct progress on upper bounds for $r_k(N)$ is our previous paper \cite{gtz2}, obtaining the bound $r_4(N) \ll Ne^{-c\sqrt{\log \log N}}$. The main objective of this paper is to obtain a bound for $r_4(N)$ of the same quality as the Heath-Brown and Szemer\'edi bound for $r_3(N)$. 

\begin{theorem}\label{main}  We have $r_4(N) \ll  N(\log N)^{-c}$ for some absolute constant $c > 0$.
\end{theorem}

An analogous result in finite fields was claimed (and published \cite{green-tao-szem}) by us around twelve years ago, although an error in this paper came to light some years later.  This was corrected around 5 years ago in \cite{green-tao-szem-fix}. These papers (like almost all of the previously cited quantitative results on $r_k(N)$) are based on the density increment argument of Roth \cite{roth}.  However we will use a slightly different ``energy decrement'' and ``regularity'' approach here, inspired by the Khintchine-type recurrence theorems for length four progressions established by Bergelson-Host-Kra \cite{bergelson-host-kra} in the ergodic setting, and by the authors \cite{green-reg} in the combinatorial setting.\vspace{11pt}

\emph{Acknowledgments.} The first author is supported by a Simons Investigator grant. The second author is supported by a Simons Investigator grant, the James and Carol Collins Chair, the Mathematical Analysis \& Application Research Fund Endowment, and by NSF grant DMS-1266164.  Part of this paper was written while the authors were in residence at MSRI in Spring 2017, which is supported by NSF grant DMS-1440140.  

We are indebted to the anonymous referee for helpful corrections and suggestions.  Finally, we would like to thank any readers interested in the result of this paper for their patience. Most of the argument was worked out by us in 2005, and the result was claimed in \cite{gtz2}, dedicated to Roth's 80th birthday. Whilst a complete, though not very readable, version has been available on request since around 2012, it has taken us until now to create a potentially publishable manuscript. 

\section{Notation}\label{notation-sec}

We use the asymptotic notation $X \ll Y$ or $X = O(Y)$ to denote $|X| \leq C Y$ for some constant $C$.  Given an asymptotic parameter $N$ going to infinity, we use $X = o(Y)$ to denote the bound $|X| \leq c(N) Y$ for some function $c(N)$ of $N$ that goes to zero as $N$ goes to infinity.  We also write $X \asymp Y$ for $X \ll Y \ll X$.  If we need the implied constant $C$ or decay function $c()$ to depend on an additional parameter, we indicate this by subscripts, e.g. $X = o_k(Y)$ denotes the bound $|X| \leq c_k(N) Y$ for a function $c_k(N)$ that goes to zero as $N \to \infty$ for any fixed choice of $k$.

We will frequently use probabilistic notation, and adopt the convention that boldface variables such as $\a$ or $\r$ represent random variables, whereas non-boldface variables such as $a$ and $r$ represent deterministic variables (or constants).
We write $\P(E)$ for the probability of a random event $E$, and $\E \X$ and $\Var \X$ for the expectation and variance of a real or complex random variable $\X$; we also use $\E(\X|E) = \frac{\E \X 1_E}{\P(E)}$ for the conditional expectation of $\X$ relative to an event $E$ of non-zero probability, where of course $1_E$ denotes the indicator variable of $E$.  In this paper, the random variables $\X$ of which we will compute expectations of will be discrete, in the sense that they take only finitely many values, so there will be no issues of measurability.  The \emph{essential range} of a discrete random variable $\X$ is the set of all values $X$ for which $\P(\X=X)$ is non-zero.

By a slight abuse of notation, we also retain the traditional (in additive combinatorics) use for $\E$ as an average, thus $\E_{a \in A} f(a) \coloneqq \frac{1}{|A|} \sum_{a \in A} f(a)$ for any finite non-empty set $A$ and function $f: A \to \C$, where we use $|A|$ to denote the cardinality of $A$.  Thus for instance $\E_{a \in A} f(a) = \E f(\a)$ if $\a$ is drawn uniformly at random from $A$.

A function $f: A \to \C$ is said to be \emph{$1$-bounded} if one has $|f(a)| \leq 1$ for all $a \in A$.  We will frequently rely on the following probabilistic form of the Cauchy-Schwarz inequality, the proof of which is an exercise.

\begin{lemma}[Cauchy-Schwarz]\label{cauchy-schwarz}  Let $A, B$ be sets, let $f: A \to \C$ be a $1$-bounded function, and let $g: A \times B \to \C$ be another function.  Let $\a, \b, \b'$ be discrete random variables in $A,B,B'$ respectively, such that $\b'$ is a conditionally independent copy of $\b$ relative to $\a$, that is to say that
$$ \P( \b = b, \b' = b' | \a = a ) = \P( \b=b | \a = a) \P( \b = b' | \a = a )$$
for all $a$ in the essential range of $\a$ and all $b,b' \in B$.  Then we have
\begin{equation}\label{fg}
 |\E f(\a) g(\a,\b)|^2 \leq \E g(\a,\b) \overline{g(\a,\b')}.
\end{equation}
\end{lemma}

We will think of this lemma as allowing one to eliminate a factor $f(\a)$ from a lower bound of the form $|\E f(\a) g(\a,\b)| \geq \eta$, at the cost of duplicating the factor $g$, and worsening the lower bound from $\eta$ to $\eta^2$.


We also have the following variant of Lemma \ref{cauchy-schwarz}:

\begin{lemma}[Popularity principle]\label{popular}  Let $\a$ be a random variable taking values in a set $A$, and let $f: A \to [-C,C]$ be a function for some $C>0$.  If we have $\E f(\a) \geq \eta$ for some $\eta>0$ then, with probability at least $\frac{\eta}{2C}$, the random variable $\a$ attains a value $a \in A$ for which $f(a) \geq \frac{\eta}{2}$.
\end{lemma}

\begin{proof}  If we set $\Omega \coloneqq \{ a \in A: f(a) \geq \eta/2 \}$, then
$$ f(\a) \leq \frac{\eta}{2} + C 1_{\a \in \Omega}$$
and hence on taking expectations
$$ \E f(\a) \leq \frac{\eta}{2} + C \P( \a \in \Omega ).$$
This implies that
$$ \P( \a \in \Omega) \geq \eta/2C$$
giving the claim.
\end{proof}

If $\theta \in \R$, we write $\| \theta\|_{\R/\Z}$ for the distance from $\theta$ to the nearest integer, and $e(\theta)= e^{2\pi i \theta}$.  Observe from elementary trigonometry that
\begin{equation}\label{bilipschitz}
|e(\theta)-1| = 2 |\sin(\pi \theta)| \asymp \| \theta \|_{\R/\Z}
\end{equation}
and hence also
\begin{equation}\label{cos-bound}
1 - \cos(2\pi \theta) = 2 |\sin(\pi \theta)|^2 \asymp \| \theta \|_{\R/\Z}^2.
\end{equation}
We will also use the triangle inequalities
\begin{equation}\label{rz-tri}
\| \theta_1 + \theta_2 \|_{\R/\Z} \leq \| \theta_1 \|_{\R/\Z} + \| \theta_2 \|_{\R/\Z}; \quad \| k \theta \|_{\R/\Z} \leq |k| \|\theta \|_{\R/\Z}
\end{equation}
for $\theta_1,\theta_2 \in \R/\Z$ and $k \in \Z$ frequently in the sequel, often without further comment.

For any prime $p$, we (by slight abuse of notation) let $a \mapsto \frac{a}{p}$ be the obvious homomorphism from $\Z/p\Z$ to $\R/\Z$ that maps $a \md{p}$ to $\frac{a}{p}\md{1}$ for any integer $a$.  We then define $e_p: \Z/p\Z \to \C$ to be the character
$$ e_p(a) \coloneqq e\left( \frac{a}{p} \right) = e^{2\pi i a/p}$$
of $\Z/p\Z$.

\section{High-level overview of argument}\label{overview-sec}

We will establish Theorem \ref{main} by establishing the following result, related to the Khintchine-type recurrence theorems mentioned earlier.  It will be convenient to introduce the notation
$$ \Lambda_{\a,\r}(\f) \coloneqq \E \f(\a) \f(\a+\r) \f(\a+2\r) \f(\a+3\r)$$
whenever $\a,\r$ are random variables on $\Z/p\Z$ and $\f: \Z/p\Z \to [-1,1]$ is a random function; of course, the notation can also be applied to deterministic functions $f: \Z/p\Z \to [-1,1]$.  Later on we will also need the conditional variant
\begin{equation}\label{Lambda-def}
 \Lambda_{\a,\r}(\f|E) \coloneqq \E(\f(\a) \f(\a+\r) \f(\a+2\r) \f(\a+3\r)|E)
\end{equation}
for some events $E$ of non-zero probability.  Informally, this quantity counts the density of arithmetic progressions $\a, \a+\r, \a+2\r, \a+3\r$ on the event $E$ weighted by $\f$, where $\a,\r$ need not be drawn uniformly or independently (and $\f$ may also be coupled to $\a,\r$).

\begin{theorem}\label{khint}
Let $p$ be a prime, let $\eta$ be a real number with $0 < \eta \leq \frac{1}{10}$, and let $f: \Z/p\Z \to [-1,1]$ be a function.  Then there exist random variables $\a,\r \in \Z/p\Z$, not necessarily independent, obeying the near-uniform distribution bound
\begin{equation}\label{eta-1}
\E f(\a) = {\mathbb E}_{x \in \Z/p\Z} f(x) + O(\eta), 
\end{equation}
the recurrence property
\begin{equation}\label{etat}
\Lambda_{\a,\r}(f) \geq (\E f(\a))^4 - O(\eta),
\end{equation}
and the ``thickness'' bound
\begin{equation}\label{r-thick}
\P(\r = 0) \ll \exp( - \eta^{-O(1)} ) / p.
\end{equation}
\end{theorem}

We note that a variant of Theorem \ref{khint} was established by us in \cite{green-reg} (answering a question in \cite{bergelson-host-kra}), in which the random variable $\a$ was uniformly distributed in $\Z/p\Z$, the random variable $\r$ was uniformly distributed in a subset of $\Z/p\Z$ of size $\gg_\eta p$ and was independent of $\a$, and the condition \eqref{r-thick} (which is crucial to the quantitative bound in Theorem \ref{main}) was not present.  Compared to that result, Theorem \ref{khint} obtains the much more quantitative bound \eqref{r-thick}, but at the expense of no longer enforcing independence between $\a$ and $\r$.  The use of non-independent random variables $\a, \r$ is an innovation of this current paper; it is similar to the technique in previous papers of using ``factors'' (finite partitions) to break up the domain $\Z/p\Z$ into smaller ``atoms'' such as Bohr sets and analysing each atom separately. However there will be technical advantages from the more general framework of pairs of independent random variables $\a,\r$.  In particular we will be able to avoid some of the boundary issues arising from irregularity of Bohr sets, by using the smoother device of ``regular probability distributions'' associated to such sets.  Although $f$ is allowed to attain negative values in Theorem \ref{khint}, in our applications we shall only be concerned with the case when $f$ is non-negative.

Let us now see how Theorem \ref{main} follows from Theorem \ref{khint}.  Clearly we may assume that $N \geq 100$. Suppose that $A$ is a subset of $\{1,\dots,N\}$ without any non-trivial four-term arithmetic progressions.  By Bertrand's postulate, we may find a prime $p$ between (say) $2N$ and $4N$.  If we define $f: \Z/p\Z \to [-1,1]$ to be the indicator function $1_A$ of $A$ (viewed as a subset of $\Z/p\Z$), then we have
\begin{equation}\label{epfa}
{\mathbb E}_{x \in \Z/p\Z} f(x) = \frac{|A|}{p}
\end{equation}
and also
\begin{equation}\label{faa}
f(a) f(a+r) f(a+2r) f(a+3r) = 0
\end{equation}
whenever $a,r \in \Z/p\Z$ with $r$ non-zero.  Now let $\a, \r$ be as in Theorem \ref{khint}, with $\eta$ to be chosen later.  From \eqref{eta-1}, \eqref{etat}, \eqref{epfa} we have
$$ \Lambda_{\a,\r}(f) \geq \left(\frac{|A|}{p}\right)^4 - O(\eta).$$
But by \eqref{faa}, \eqref{r-thick}, the left-hand side is $O( \exp( -\eta^{-O(1)} ) / p )$.  Setting $\eta \coloneqq c \log^{-c} p$ for a sufficiently small absolute constant $c>0$, we conclude that
$$ \left(\frac{|A|}{p}\right)^4 \ll \log^{-c} p$$
and hence $A \ll N \log^{-c/4} N$, giving Theorem \ref{main}.\vspace{11pt}

\emph{Remark.} As mentioned previously, the arguments in \cite{green-reg} established a bound of the form \eqref{etat} with $\a$ and $\r$ independent, and also one could ensure that $\a$ was uniformly distributed over $\Z/p\Z$.  As a consequence, one could establish a variant of Theorem \ref{main}, namely that for any $N \geq 1$, $\eta > 0$, and $A \subset [N]$, one had
$$ \frac{|A \cap (A-r) \cap (A-2r) \cap (A-3r)|}{N} \geq \left(\frac{|A|}{N}\right)^4 - \eta$$
for $\gg_\eta N$ choices of $0 \leq r \leq N$.  Unfortunately our methods do not seem to provide a good bound of this form due to our coupling together of $\a$ and $\r$.\vspace{11pt}

It remains to establish Theorem \ref{khint}.  As in \cite{bergelson-host-kra,green-reg}, the lower bound \eqref{etat} will ultimately come from the following consequence of the Cauchy-Schwarz inequality which counts solutions to the equation $x-3y+3z-w=0$ for $x,y,z,w$ in some subset of a compact abelian group; this inequality is a specific feature of the theory of length four progressions which is not available for longer progressions\footnote{For longer progressions, the relevant constraints coming from nilpotent algebra are significantly more complicated than a single linear equation; see \cite{ziegler}.  In any event, the counterexamples in \cite{bergelson-host-kra} indicate that no comparable positivity property with polynomial lower bounds will hold for higher length progressions.}.

\begin{lemma}[Application of Cauchy-Schwarz]\label{cauchy}  Let $G = (G,+)$ be a compact abelian group, let $\mu$ be the probability Haar measure on $G$, and let $F: G \to \R$ be a bounded measurable function.  Then
\[ \int_G \int_G \int_G F(x) F(y) F(z) F(x-3y+3z)\  d\mu(x) d\mu(y) d\mu(z) \geq \left(\int_G F d\mu\right)^4.\]
\end{lemma}

\begin{proof}  Making the change of variables $w = x - 3y$ and using Fubini's theorem, the left-hand side may be rewritten as
$$ \int_G \left(\int_G F(w+3y) F(y)\ d\mu(y)\right)^2\ d\mu(w),$$
which by the Cauchy-Schwarz inequality is at least
$$ \left(\int_G \int_G F(w+3y) F(y)\ d\mu(y)d\mu(w)\right)^2.$$
But by a further application of Fubini's theorem, the expression inside the square is $(\int_G F(x)\ d\mu(x))^2$.  The claim follows.
\end{proof}

To see the relevance of this lemma to Theorem \ref{khint}, and to motivate the strategy of proof of that theorem, let us first test that theorem on some key examples.  To simplify the exposition, our discussion will be somewhat non-rigorous in nature; for instance, we will make liberal use of the non-rigorous symbol $\approx$ without quantifying the nature of the approximation.\vspace{11pt}

\emph{Example 1: a well-distributed pure quadratic factor.}  Let $G$ be the $d$-torus $G = (\R/\Z)^d$ for some bounded $d=O(1)$, and let $F: G \to [-1,1]$ be a smooth function (independent of $p$); for instance, $F$ could be a finite linear combination of characters $\chi: G \to S^1$ of $G$.  Let $\alpha_1,\dots,\alpha_d \in \Z/p\Z$ be ``generic'' frequencies, in the sense that there are no non-trivial linear relations of the form 
\begin{equation}\label{kad}
k_1 \alpha_1 + \dots + k_d \alpha_d = 0
\end{equation}
with $k_1,\dots,k_d=O(1)$ not all equal to zero.  We also introduce some additional frequencies $\beta_1,\dots,\beta_d \in \Z/p\Z$, for which we impose no genericity restrictions.  Let $f: \Z/p\Z \to [-1,1]$ be the function
$$ f(a) \coloneqq F\left( Q(a) \right)$$
where $Q: \Z/p\Z \to G$ is the quadratic polynomial
$$ Q(a) \coloneqq \left( \frac{\alpha_1 a^2 + \beta_1 a}{p}, \dots, \frac{\alpha_d a^2 + \beta_d a}{p} \right),$$
and where we use the obvious division by zero map $a \mapsto \frac{a}{p}$ from $\Z/p\Z$ to $\R/\Z$.  
For any tuples $k = (k_1,\dots,k_d) \in \Z^d \equiv \hat G$ and $\xi = (\xi_1,\dots,\xi_d) \in G$, we define the dot product
$$ k \cdot \xi \coloneqq k_1 \xi_1 + \dots + k_d \xi_d.$$
Because of our genericity hypothesis on the $\alpha_i$, we see from Gauss sum estimates that
$$ {\mathbb E}_{a \in \Z/p\Z} e(k \cdot Q(a)) \approx 0$$
for any bounded tuple $k \in \Z^d$ when $p$ is large.  By the Weyl equidistribution criterion, we thus see that when $p$ is large, the quantity $\frac{\alpha a^2 + \beta a}{p}$ becomes equidistributed in $G$ as $a$ ranges over $\Z/p\Z$.  In particular, as $F$ was assumed to be smooth, we expect to have
$$ \E f(\a) = {\mathbb E}_{a \in \Z/p\Z} f(a) \approx \int_G F(x)\ d\mu(x)$$
if $\a$ is drawn uniformly in $\Z/p\Z$.
Now suppose that $\r$ is also drawn uniformly in $\Z/p\Z$, independently of $\a$.  The tuple 
\begin{equation}\label{tupo}
(Q(\a), Q(\a+\r), Q(\a+2\r), Q(\a+3\r))
\end{equation}
will \emph{not} become equidistributed in $G^4$, because of the elementary algebraic identity
\begin{equation}\label{cons}
 Q(\a) - 3Q(\a+\r) + 3Q(\a+2\r) - Q(\a+3\r) = 0,
\end{equation}
which is a discrete version of the fact that the third derivative of any quadratic polynomial vanishes.  However, this turns out to be the \emph{only} constraint on this tuple in the limit $p \to \infty$.  Indeed, from the genericity hypothesis on the $\alpha_i$, one can verify that the quadratic form
$$ (a,r) \mapsto k_0 \cdot Q_0(a) + k_1 \cdot Q_0(a+r) + k_2 \cdot Q_0(a+2r)  + k_3 \cdot Q_0(a+3r)$$
on $(\Z/p\Z)^2$ for bounded tuples $k_0,k_1,k_2,k_3 \in \Z^d$ vanishes if and only if $(k_0,k_1,k_2,k_3)$ is of the form $(k,-3k, 3k, -k)$ for some tuple $k$, where
$$ Q_0(a) \coloneqq \left( \frac{\alpha_1 a^2}{p}, \dots, \frac{\alpha_d a^2}{p} \right)$$
denotes the purely quadratic component of $Q(a)$.  Using this and a variant of the Weyl equidistribution criterion, one can eventually compute that
$$\Lambda_{\a,\r}(f) \approx \int_G \int_G \int_G F(x) F(y) F(z) F(x-3y+3z)\ d\mu(x) d\mu(y) d\mu(z).$$
Applying Lemma \ref{cauchy}, we conclude (a heuristic version of) Theorem \ref{khint} in this case, taking $\a,\r$ to be independent uniformly distributed variables on $\Z/p\Z$.\vspace{11pt}

\emph{Example 2. A well-distributed impure quadratic factor.} Now we give a ``local'' version of the first example, in which the function $f$ exhibits ``locally quadratic'' behaviour rather than ``globally quadratic'' behaviour.  Let $\eta > 0$ be a small parameter, and suppose that $p$ is very large compared to $\eta$.  We suppose that the cyclic group $\Z/p\Z$ is somehow partitioned into a number $P_1,\dots,P_m$ of arithmetic progressions; the number $m$ of such progressions should be thought of as being moderately large (e.g. $m \sim \exp(1/\eta^{O(1)})$ for some parameter $\eta > 0$).  Consider one such progression, say $P_c = \{ b_c + n s_c: 1 \leq n \leq N_c \}$ for some $b_c,s_c \in \Z/p\Z$ and some $N_c > 0$; one should think of $N_c$ as being reasonably large, e.g. $N_c \gg \exp(-1/\eta^{O(1)}) p$.  To each such progression $P_c$, we associate a torus $G_c = (\R/\Z)^{d_c}$ for some bounded $d_c$ with probability Haar measure $\mu_c$, a smooth function $F_c: G_c \to [-1,1]$, and a collection $\xi_{c,1}, \dots,\xi_{c,d_c} \in \R/\Z$ of frequencies which are generic in the sense that there does not exist any non-trivial relations of the form
\begin{equation}\label{kdj}
 k_1 \xi_{c,1} + \dots + k_{d_c} \xi_{c,d_c} = O\left( \frac{1}{N_c} \right) \md{1}
\end{equation}
for bounded $k_1,\dots,k_{d_c} \in \Z$.  We then define the function $f: \Z/p\Z \to [-1,1]$ by setting
$$ f( b_c + n s_c ) \coloneqq F_c( \xi_{c,d_1} n^2, \dots, \xi_{c,d_c} n^2 ) $$
for $1 \leq c\leq m$ and $1 \leq n \leq N_c$.  One could also add a lower order linear term to the phases $\xi_{c,i} n^2$, as in the preceding example, if desired, but we will not do so here to simplify the exposition slightly.

Within each progression $P_c$, a Weyl equidistribution analysis (using the genericity hypothesis) reveals that the tuple $( \xi_{c,d_1} n^2, \dots, \xi_{c,d_c} n^2 )$ becomes equidistributed in $G_c$ as $p$ becomes large, so that
\begin{equation}\label{epa}
 {\mathbb E}_{a \in P_c} f(a) \approx \int_{G_c} F_c(x)\ d\mu_c(x).
\end{equation}
Now we define the random variables $\a, \r \in \Z/p\Z$ as follows.  We first select a random element $\c$ from $\{1,\dots,m\}$ with $\P( \c = c) = |P_j| / p$ for $c=1,\dots,m$.  Conditioning on the event that $\c$ is equal to $c$, we then select $\a$ uniformly at random from $P_c$, and also select $\r$ uniformly at random from an arithmetic progression of the form 
\begin{equation}\label{rs}
\{ n s_c: |n| \leq \exp(-1/\eta^{-C})N_c \},
\end{equation}
with $\a$ and $\r$ independent after conditioning on $\c = c$.  Note that $\a$ and $\r$ are only \emph{conditionally} independent, relative to the auxiliary variable $\c$; if one does not perform this conditioning, then $\a$ and $\r$ become coupled to each other through their mutual dependence on $\c$.

Without conditioning on $\c$, the random variable $\a$ becomes uniformly distributed on $\Z/p\Z$, thus
$$ \E f(\a) = {\mathbb E}_{a \in \Z/p\Z} f(a).$$
Also, from \eqref{epa} we have the conditional expectation
$$ \E (f(\a) | \c = c ) \approx \int_{G_c} F_c(x)\ d\mu_c(x).$$
A modification of the equidistribution analysis from the first example also gives
\begin{align*}
& \Lambda_{\a,\r}(f  | \c = c ) \\
&\quad \gtrapprox 
\int_{G_c} \int_{G_c} \int_{G_c} F_c(x) F_c(y) F_c(z) F(x-3y+3z)\ d\mu_c(x) d\mu_c(y) d\mu_c(z),
\end{align*}
where the conditional quartic form $\Lambda_{\a,\r}(f  | \c = c )$ was defined in \eqref{Lambda-def},
and hence by Lemma \ref{cauchy} we have
$$ \Lambda_{\a,\r}(f  | \c = c ) \gtrapprox \left(\E(f(\a) | \c = c )\right)^4.$$
Averaging in $c$ (weighted by $\P( \c = c)$) to remove the conditional expectation on the left-hand side, and then applying H\"older's inequality, we obtain a heuristic version of Theorem \ref{khint} in this case.\vspace{11pt}

\emph{Example 3: A poorly distributed pure quadratic factor.} We now return to the situation of the first example, except that we no longer impose the genericity hypothesis, that is to say we allow for a non-trivial relation of the form \eqref{kad}.  Without loss of generality we can take the coefficient $k_d$ of this relation to be non-zero.  Because of this relation, the quantity $Q(\a)$ studied in the first example and the tuple \eqref{tupo} may not necessarily be as equidistributed as before.  However, we can use this irregularity of distribution to modify the representation of $f$ (up to a small error) in such a manner as to reduce the number $d$ of quadratic phases involved.  Namely, we can write
$$ f(a) \coloneqq \tilde F\left( \tilde Q(a), \frac{\gamma a}{p} \right)$$
where
\begin{align*}
\tilde Q(a) &\coloneqq \left( \frac{k_d^{-1} \alpha_1 a^2 + k_d^{-1} \beta_1 a }{p}, \dots, \frac{k_d^{-1} \alpha_{d-1} a^2 + k_d^{-1} \beta_{d-1} a}{p} \right) \\
\gamma &\coloneqq \beta_d + k_1 k_d^{-1} \beta_1 + \dots + k_{d-1} k_d^{-1} \beta_{d-1} \\
\tilde F( x_1,\dots,x_{d-1}, y) & \coloneqq F( k_d x_1, \dots, k_d x_{d-1}, - k_1 x_1 - \dots - k_{d-1} x_{d-1} + y )
\end{align*}
and where we take advantage of the field structure of $\Z/p\Z$ to locate an inverse $k_d^{-1}$ of $k_d$ in this field.  For our quantitative analysis we will run into a technical difficulty with this representation, in that the Lipschitz constant of $\tilde F$ will increase by an undesirable amount compared to that of $F$ when one performs this change of variable, at least if one uses the standard metric on the torus.  To fix this, we will eventually have to work with more general tori $\prod_{i=1}^d \R/\lambda_i \Z$ than the standard torus $(\R/\Z)^d$, but we ignore this issue for now to continue with the heuristic discussion.

To remove the dependence on the linear phase $\frac{\gamma a}{p}$, we partition $\Z/p\Z$ into ``(shifted) Bohr sets'' $B_1,\dots,B_m$ for some moderately large $m$ (e.g. $m \sim \exp(1/\eta^{-C})$ for some constant $C>0$), defined by
$$ B_c \coloneqq \left\{ a \in \Z/p\Z: \frac{\gamma a}{p} \in \left[\frac{c-1}{m}, \frac{c}{m}\right) \md{1} \right\}$$
for $c=1,\dots,m$.  On each Bohr set $B_c$, we have the approximation
$$ f(a) \coloneqq \tilde F_c( \tilde Q(a) )$$
where $\tilde F_c( x, y) \coloneqq \tilde F( x, \frac{c}{m} )$.  Using the heuristic that Bohr sets behave like arithmetic progressions, the situation is now similar to that in the second example, with the number of quadratic phases involved reduced from $d$ to $d-1$, except that there may still be some non-trivial relations amongst the surviving quadratic phases (and one also now has some lower order linear terms in the quadratic phases).  To deal with this difficulty, we turn now to the consideration of yet another example.\vspace{11pt}

\emph{Example 4: A poorly distributed impure quadratic factor.} We now consider an example which is in some sense a combination of the second and third examples.  Namely, we suppose we are in the same situation as in the second example, except that we allow some of the indices $c$ to have ``poor quadratic distribution'' in the sense that
they admit non-trivial relations of the form \eqref{kdj}.  Again we may assume without loss of generality that $k_{d_c}$ is non-zero in such relations.  Because of such relations, we no longer expect to have the equidistribution properties that were used in the second example.  However, by modifying the calculations in the third example, we can obtain a new representation of $f$ (again allowing for a small error) on each of the progressions $P_c$ with poor quadratic distribution in order to reduce the number $d_c$ of quadratic polynomials used in that progression by one.  Iterating this process a finite number of times, one eventually returns to the situation in the second example in which no non-trivial relations occur, at which point one can (heuristically, at least) verify Theorem \ref{khint} in this case.

The situation becomes slightly more complicated if one adds a lower order linear term $\zeta_{c,i} n$ to the purely quadratic phases $\xi_{c,i} n^2$ appearing in the second example; this basically is the type of situation one encounters for instance at the conclusion of the third example.  In this case, every time one converts a non-trivial relation of the form \eqref{kdj} on one of the cells $P_c$ of the partition into a new representation of $f$ on that cell, one must subdivide that cell $P_j$ into smaller pieces, by intersecting $P_j$ with various Bohr sets.  However, the resulting sets still behave somewhat like arithmetic progressions, and it turns out that we can still iterate the construction a bounded number of times until no further non-trivial relations between surviving quadratic phases remain on any of the cells of the partition, at which point one can (heuristically, at least) verify Theorem \ref{khint} in this case (as well as in the case considered in the third example). \vspace{11pt} 

\emph{Example 5: A pseudorandom perturbation of a pure quadratic factor.} In all the preceding examples, the function $f: \Z/p\Z \to [-1,1]$ under consideration was ``locally quadratically structured'', in the sense that on local regions such as $P_c$, the function $f$ could be accurately represented in terms of quadratic phase functions $a \mapsto Q(a)$.  This is however not the typical behaviour expected for a general function $f: \Z/p\Z \to [-1,1]$.  A more representative example would be a function of the form
$$ f(a) \coloneqq f_1(a) + f_2(a),$$
where $f_1: \Z/p\Z \to \R$ is a function of the type considered in the first example, thus
$$ f_1(a) = F( Q(a) )$$
for some quadratic function $Q: \Z/p\Z \to G$ into a torus $G = (\R/\Z)^d$ and some smooth $F: G \to [-1,1]$, and $f_2: \Z/p\Z \to [-1,1]$ is a function which is \emph{globally Gowers uniform} in the sense that
\begin{equation}\label{f2unif}
 \E \prod_{(\omega_1,\omega_2,\omega_3) \in \{0,1\}^3} f_2(\a + \omega_1 \h_1 + \omega_2 \h_2 + \omega_3 \h_3) \approx 0,
\end{equation}
where $\a,\h_1,\h_2,\h_3$ are drawn independently and uniformly at random from $\Z/p\Z$.  A typical example to keep in mind is when $F$ (and hence $f_1$) takes values in $[0,1]$, and $f=\f$ is a random function with $f(a)$ equal to $1$ with probability $f_1(a)$ and $0$ with probability $1-f_1(a)$, independently as $a \in \Z/p\Z$ varies; then the $f_2(a)$ for $a \in \Z/p\Z$ become independent random variables of mean zero, and the global Gowers uniformity can be established with high probability using tools such as the Chernoff inequality.

From the standard theory of the Gowers norms (see e.g. 
\cite[Chapter 11]{taovu-book}), one can use the global Gowers uniformity of $f_2$, combined with a number of applications of the Cauchy-Schwarz inequality, to establish a ``generalised von Neumann theorem'' which, in our current context, implies that $f$ and $f_1$ globally count about the same number of length four progressions in the sense that
\begin{equation}\label{eff}
\Lambda_{\a,\r}(f) \approx \Lambda_{\a,\r}(f_1);
\end{equation}
similarly one also has
\begin{equation}\label{eff-2}
 \E f(\a) \approx \E f_1(\a).
\end{equation}
As a consequence, Theorem \ref{khint} for such functions follows (heuristically, at least) from the analysis of the first example, at least if one assumes the genericity of the frequencies $\xi_1,\dots,\xi_d$.\vspace{11pt}

\emph{Example 6: A pseudorandom perturbation of an impure quadratic factor.} We now consider a situation which is to the second example as the fifth example was to the first.  Namely, we consider a function of the form
$$ f(a) \coloneqq f_1(a) + f_2(a),$$
where $f_1: \Z/p\Z \to [-1,1]$ is a function of the type considered in the second example, thus
$$ f_1( b_c + n s_c ) \coloneqq F_c( \xi_{c,d_1} n^2, \dots, \xi_{c,d_c} n^2 ) $$
for $c=1,\dots,m$ and $n=1,\dots,N_c$.  As for the function $f_2: \Z/p\Z \to [-1,1]$, global Gowers uniformity of $f_2$ will be too weak of a hypothesis for our purposes, because the random variable $\r$ appearing in the second example is now localised to a significantly smaller region than $\Z/p\Z$.  Instead, we will require the \emph{local Gowers uniformity} hypothesis
\begin{equation}\label{lgu}
 \E \prod_{(\omega_1,\omega_2,\omega_3) \in \{0,1\}^3} f_2(\a + \omega_1 \h_1 + \omega_2 \h_2 + \omega_3 \h_3) \approx 0,
\end{equation}
where $\a$ is now the random variable from the second example (in particular, $\a$ depends on the auxiliary random variable $\c$), and once one conditions on an event $\c = c$ for $c=1,\dots,m$, one draws $\h_1,\h_2,\h_3$ independently of each other and from $\a$, and each $\h_i$ drawn uniformly from an arithmetic progression of the form
\begin{equation}\label{rsc}
\{ n s_c: |n| \leq \exp(-1/\eta^{-C_i})N_c \},
\end{equation}
for some constant $C_i>0$ (for technical reasons, it is convenient to allow these constants $C_1,C_2,C_3$ to be different from each other, and also to be larger than the constant $C$ appearing in \eqref{rs}, so that $\h_1,\h_2,\h_3$ range over a narrower scale than $\r$).  As with $\a$ and $\r$, the random variables $\a,\h_1,\h_2,\h_3$ are now only \emph{conditionally} independent relative to the auxiliary variable $\c$, but are not independent of each other without this conditioning, as they are coupled to each other through $\c$.

As it turns out, once one assumes this local Gowers uniformity of $f_2$, one can modify the Cauchy-Schwarz arguments used to establish the global generalised von Neumann theorem to obtain the approximations \eqref{eff}, \eqref{eff-2} for the random variables $\a,\r$ considered in the second example, at which point Theorem \ref{khint} for this choice of $f$ follows (heuristically, at least) from the analysis of that example, at least if one assumes that there are no non-trivial relations of the form \eqref{kdj}.\vspace{11pt}

\emph{Example 7: Non-pseudorandom perturbation of a pure quadratic factor.} We now modify the fifth example by replacing the hypothesis \eqref{f2unif} by its negation
\begin{equation}\label{roq}
 \E \prod_{(\omega_1,\omega_2,\omega_3) \in \{0,1\}^3} f_2(\a + \omega_1 \h_1 + \omega_2 \h_2 + \omega_3 \h_3) \gg 1
\end{equation}
(it is not difficult to show that the left-hand side is non-negative).  In this case, the generalised von Neumann theorem used in that example does not give a good estimate.  However, in this situation one can apply the inverse theorem for the Gowers norm established by us in \cite{gt-inverseu3}.  In order to obtain good quantitative bounds, we will use the version of that theorem that involves local correlation with quadratic objects (as opposed to a somewhat weak global correlation with a single ``locally quadratic'' object).  Namely, if \eqref{roq} holds, then one can partition $\Z/p\Z$ into a moderately large (e.g. $O(\exp(1/\eta^{-O(1)}))$) number of pieces $P_1,\dots,P_m$, such that on each piece $P_c$, the function $f_2$ correlates with a ``quadratically structured'' object.  The precise statement is somewhat technical to state, but one simple special case of this conclusion is that the pieces $P_1,\dots,P_m$ are arithmetic progressions as in the second example, and for a ``significant number'' of the progressions
$$P_c = \{ b_c + n s_c: 1 \leq n \leq N_c \}$$
there exists a frequency $\xi_c \in \R/\Z$ such that
$$ |{\mathbb E}_{1 \leq n \leq N_c} f_2(b_c + ns_c) e(-\xi_c n^2)| \gg 1.$$
(In general, one would take $P_c$ to be Bohr sets of moderately high rank, rather than arithmetic progressions, and the phase $a \mapsto \xi_c a^2/p$ would have to be replaced by a more general locally quadratic phase function on such a Bohr set, but we ignore these technicalities for the current informal discussion.)  From this and the cosine rule, it is possible to find a function $g: \Z/p\Z \to [-1,1]$ that is equal to (the real part of) a scalar multiple of the quadratic phases $b_c + ns_c \mapsto e(\xi_c n^2)$ on each progression $P_c$, such that $f_2+g$ has an \emph{energy decrement} compared to $f_2$ in the sense that
\begin{equation}\label{f2g}
 {\mathbb E}_{a \in \Z/p\Z} (f_2(a) + g(a))^2 \leq {\mathbb E}_{a \in \Z/p\Z} f_2(a)^2 - \eta^C
\end{equation}
for some constant $C>0$.  In this situation, we can modify the decomposition $f = f_1 + f_2$ by adding $g$ to $f_2$ and subtracting it from $f_1$.  (Strictly speaking, this may make $f_1$ and $f_2$ range slightly outside of $[-1,1]$, but because $f$ itself ranges in $[-1,1]$, it turns out to be relatively easy to modify $f_1,f_2$ further to rectify this problem.)  The new function $f_1$ has a similar ``quadratic structure'' to the previous function $f_1$, except that the quadratic structure is now localised to the cells $P_1,\dots,P_m$ of the partition of $\Z/p\Z$, and the number of quadratic functions has been increased by one.  If the new function $f_2$ is now locally Gowers uniform in the sense of \eqref{lgu}, then we are now essentially in the situation of the sixth example (at least if there are no non-trivial relations of the form \eqref{kdj}), and we can (heuristically at least) conclude Theorem \ref{khint} in this case by the previous analysis.  If $f_2$ is locally Gowers uniform but there are additionally some relations of the form \eqref{kdj}, then one can hope to adapt the analysis of the fourth example to reduce the quadratic complexity of $f_1$ on all the poorly distributed cells, at which point one restarts the analysis.  If however $f_2$ remains non-uniform, then we need to argue using the analysis of the next and final example.\vspace{11pt}

\emph{Example 8: Non-pseudorandom perturbation of an impure quadratic factor.} Our final and most difficult example will be as to the sixth example as the seventh example was to the fifth.  Namely, we modify the sixth example by assuming that the negation of \eqref{lgu} holds.  Equivalently, one has the lower bound
\begin{equation}\label{fah}
 \E \left(\prod_{(\omega_1,\omega_2,\omega_3) \in \{0,1\}^3} f_2(\a + \omega_1 \h_1 + \omega_2 \h_2 + \omega_3 \h_3) | \c = c\right) \gg 1
\end{equation}
on the local Gowers norm for a ``significant fraction'' of the $c=1,\dots,m$.

At the qualitative level, the inverse theorem in \cite{gt-inverseu3} for the global Gowers norm allows one to also deduce a similar conclusion starting from the hypothesis \eqref{fah}.  However, the quantitative bounds obtained by this approach turn out to be too poor for the purposes of establishing Theorem \ref{khint} or Theorem \ref{main}.  Instead, one must obtain a quantitative local inverse theorem for the Gowers norm that has reasonably good bounds (of polynomial type) on the amount of correlation that is (locally) attained.  Establishing such a theorem is by far the most complicated and lengthy component of this paper, although broadly speaking it follows the same strategy as previous theorems of this type in \cite{gowers-4-aps,gt-inverseu3}.  If one takes this local inverse theorem for granted, then roughly speaking what we can then conclude from the hypothesis \eqref{fah} is that for a significant number of $c=1,\dots,m$, one can partition the cell $P_c$ into subcells $P_{c,1},\dots,P_{c,m_c}$, and locate a ``locally quadratic phase function'' $\phi_{c,i}: P_{c,i} \to \R/\Z$ on each such subcell (generalising the functions $b_c + ns_c \mapsto e(\xi_c n^2)$ from the previous example), such that
$$ |{\mathbb E}_{a \in P_{c,i}} f_2(b_{c,i}) e(-\phi_{c,i}(a))| \gg 1$$
for a significant fraction of the $c,i$.  Using this, one can again obtain an energy decrement of the form \eqref{f2g}, where now $g$ is (the real part of) a scalar multiple of the functions $a \mapsto e(\phi_{c,i}(a))$ on each $P_{c,i}$.  By arguing as in the sixth example, one can then modify $f_1$ and $f_2$ in such a way that the ``energy'' $\E f_2(\a)^2$ decreases significantly, while $f_1$ is now locally quadratically structured on a somewhat finer partition of $\Z/p\Z$ than the original partition $P_1,\dots,P_m$, with the number of quadratic phases needed to describe $f_1$ on each partition having increased by one.  If the function $f_2$ is now locally Gowers uniform (with respect to a new set of random variables $\a,\r$ adapted to this finer partition), and there are no non-trivial relations of the form we can now (heuristically) conclude Theorem \ref{khint} from the analysis of the sixth example, assuming the addition of the new quadratic phase has not introduced relations of the form \eqref{kdj}.  If such relations occur, though, one can hope to adapt the analysis of the fourth example to reduce the quadratic complexity of the poorly distributed cells, perhaps at the cost of further subdivision of the cells.  Finally, if the new version of $f_2$ remains non-uniform with respect to the finer partition, then one iterates the analysis of this example to reduce the energy of $f_2$ further.  This process cannot continue indefinitely due to the non-negativity of the energy (and also because none of the other steps in the iteration will cause a significant increase in energy).  Because of this, one can hope to cover all cases of Theorem \ref{khint} by some complicated iteration of the eight arguments described above.\vspace{11pt}

Having informally discussed the eight key examples for Theorem \ref{khint}, we return now to the task of proving this theorem rigorously.

It will be convenient to work throughout the rest of the paper with a fixed choice
$$ 1 < C_1 < C_2 < \dots < C_5$$
of absolute constants, with each $C_i$ assumed to be sufficiently large depending on the previous $C_1,\ldots,C_{i-1}$.  For instance, for sake of concreteness one could choose $C_i \coloneqq  2^{2^{100i}}$; of course, other choices are possible.  The implied constants in the $O()$ notation will not depend on the $C_i$ unless otherwise specified.  These constants will serve as exponents for various scales $\eta^{-C_i}$ that will appear in our analysis, with the point being that any scale of the form $\eta^{-C_i}$ for $i=2,\dots,5$ is extremely tiny with respect to any polynomial combination of the previous scales $\eta^{-C_1},\dots,\eta^{-C_{i-1}}$.

In all of the eight examples considered above, the function $f$ was approximated by some ``quadratically structured'' function, usually denoted $f_1$, with the approximation being accurate in various senses with respect to some pair $(\a,\r)$ of random variables.  The rigorous argument will similarly approximate $f$ by a quadratically structured object; it will be convenient to make this object a \emph{random} function $\f$ rather than a deterministic one (though as it turns out, this function will become deterministic again once an auxiliary random variable $\c$ is fixed).  The precise definition of ``quadratically structured'' will be rather technical, and will eventually be given in Definition \ref{sla-def}.  For now, we shall abstract the properties of ``quadratic structure'' we will need, in the following proposition involving an abstract directed graph $G = (V,E)$ (encoding the ``structured local approximants'') which we will construct more explicitly later.  We will shortly iterate this proposition to establish Theorem \ref{khint} and hence Theorem \ref{main}.

\begin{proposition}[Main proposition, abstract form]\label{main-abstract}  
Let $\eta$ be a real number with $0 < \eta \leq \frac{1}{10}$, and let $p$ be a prime with
\begin{equation}\label{p-large}
p \geq \exp(\eta^{-3C_5}).
\end{equation}
Let $f: \Z/p\Z \to [0,1]$ be a function.
Then there exist the following:
\begin{itemize}
\item[(a)] A \textup{(}possibly infinite\textup{)} directed graph $G = (V,E)$, with elements $v \in V$ referred to as \emph{structured local approximants}, and the notation $v \to v'$ used to denote the existence of a directed edge from one structured local approximant $v$ to another $v'$;
\item[(b)] A triple $(\a_v, \r_v, \f_v)$ associated to $f$ and to each structured local approximant $v \in V$, where $\a_v, \r_v$ are random variables in $\Z/p\Z$, and $\f_v: \Z/p\Z \to [-1,1]$ is a random function \textup{(}with $\a_v, \r_v, \f_v$ \emph{not} assumed to be independent\textup{)};
\item[(c)] A \emph{quadratic dimension} $d_2(v) \in \N$ assigned to each vertex $v \in V$;
\item[(d)] A \emph{poorly distributed quadratic dimension} $d_2^{\mathrm{poor}}(v) \in \N$ assigned to each vertex $v \in V$, with $0 \leq d_2^{\mathrm{poor}}(v) \leq d_2(v)$; and
\item[(e)] An initial approximant $v_0 \in V$, with $d_2(v_0)=0$ \textup{(}and hence $d_2^{\mathrm{poor}}(v_0)$ $= 0$\textup{)}.
\end{itemize}
Furthermore, whenever a structured local approximant $v_k \in V$ can be reached from $v_0$ by a path $v_0 \to v_1 \to \dots \to v_k$ with $0 \leq k \leq 8\eta^{-2C_2}$, then the following properties are obeyed:
\begin{itemize}
\item[(i)] One has the ``thickness'' condition
\begin{equation}\label{tp-thick}
\P( \r_{v_k} = 0 ) \ll \exp( 3 \eta^{-C_5} ) / p;
\end{equation}
\item[(ii)]  We have the almost uniformity condition
\begin{equation}\label{tp-good-2}
|\E f(\a_{v_k}) - {\mathbb E}_{a \in \Z/p\Z} f(a)| \leq \eta;
\end{equation}
\item[(iii)] Bad approximation implies energy decrement: if 
\begin{equation}\label{tp-bad-1}
|\E \f_{v_k}(\a_{v_k}) - f(\a_{v_k}) | > \eta
\end{equation}
or
\begin{equation}\label{tp-bad}
\begin{split}
&|\Lambda_{\a_{v_k}, \r_{v_k}}(\f_{v_k}) - \Lambda_{\a_{v_k}, \r_{v_k}}(f)|> \eta
\end{split}
\end{equation}
then there exists a structured local approximant $v_{k+1} \in V$ with $v_k \to v_{k+1}$ such that 
$$ \E |f(\a_{v_{k+1}}) - \f_{v_{k+1}}(\a_{v_{k+1}})|^2 \leq \E |f(\a_{v_k}) - \f_{v_k}k(\a_{v_k})|^2 - \eta^{C_2} $$
and
$$
d_2(v_{k+1}) \leq d_2(v_k)+1.
$$
\item[(iv)] Failure of ``Khintchine-type recurrence'' implies dimension decrement: if 
\begin{equation}\label{lower}
\Lambda_{\a_{v_k},\r_{v_k}}(\f_{v_k}) \leq (\E \f_{v_k}(\a_{v_k}))^4 - \eta,
\end{equation}
then there exists a structured local approximant $v_{k+1} \in V$ with $v_k \to v_{k+1}$ obeying the bounds
\begin{align*}
\E |f(\a_{v_{k+1}}) - \f_{v_{k+1}}(\a_{v_{k+1}})|^2 &\leq \E |f(\a_{v_k}) - \f_{v_k}(\a_{v_k})|^2 + \eta^{3C_2}, \\
d_2(v_{k+1}) &\leq d_2(v_k), \\
d_2^{\mathrm{poor}}(v_{k+1}) &\leq d_2^{\mathrm{poor}}(v_k)-1.
\end{align*}
\end{itemize}
\end{proposition}

The proof of this proposition will occupy the remainder of the paper.  For now, let us see how this proposition implies Theorem \ref{khint}.
Let $p, \eta, f$ be as in that theorem, and let $C_1,\dots,C_5$ be as above.  If the largeness criterion \eqref{p-large} fails, then we may set $\r \coloneqq 0$, $\f \coloneqq f$, and draw $\a$ uniformly at random from $\Z/p\Z$, and it is easy to see that the conclusions of Theorem \ref{khint} are obeyed (with \eqref{etat} following from H\"older's inequality).  Thus we may assume without loss of generality that \eqref{p-large} holds.

Let $G = (V,E)$, $v_0$, $d_2()$, $d_2^{\mathrm{poor}}()$, and $(\a_v, \r_v, \f_v)$ be as in Proposition \ref{main-abstract}.  Suppose first that there exists a structured local approximant $v_k \in V$ that can be reached from $v_0$ by a path of length at most $8\eta^{-2C_2}$, and for which none of the inequalities \eqref{tp-bad-1}, \eqref{tp-bad}, \eqref{lower} hold, that is to say one has the bounds
\begin{align}
|\E \f_{v_k}(\a_{v_k}) - f_{v_k}(\a_{v_k}) | &\leq \eta,\label{khint-2-new1} \\
|\Lambda_{\a_{v_k}, \r_{v_k}}(\f_{v_k})-\Lambda_{\a_{v_k}, \r_{v_k}}(f_{v_k})| &\leq \eta \label{khint-2-new} \\
\Lambda_{\a_{v_k}, \r_{v_k}}(\f_{v_k}) & > (\E \f_{v_k}(\a_{v_k}))^4 - \eta.\label{khint-1-new}
\end{align}
From \eqref{khint-1-new}, \eqref{khint-2-new}, \eqref{khint-2-new1} and the triangle inequality (and the boundedness of $\f_{v_k}, f$) we conclude that
$$
\Lambda_{\a_{v_k}, \r_{v_k}}(f_{v_k}) > (\E f(\a_{v_k}))^4 - O(\eta);$$
combining this with \eqref{tp-thick} and \eqref{tp-good-2} we see that the random variables $\a_{v_k}, \r_{v_k}$ obey the properties required of Theorem \ref{khint}.  Thus we may assume for sake of contradiction that this situation never occurs, which by Proposition \ref{main-abstract} implies that whenever $v_k \in V$ is a structured local approximant that can be reached from $v_0$ by a path of length at most $8\eta^{-2C_2}$, then the conclusions of at least one of (iii) and (iv) hold.  Iterating this we may therefore construct a path
$$ v_0 \to v_1 \to \dots \to v_{k_0+1}$$
with 
\begin{equation}\label{kdef}
k_0 \coloneqq  \lfloor 8 \eta^{-2C_2} \rfloor,
\end{equation}
such that for every $0 \leq k \leq k_0$, one either has the energy decrement bounds
\begin{align*}
\E |f(\a_{v_{k+1}}) - \f_{k+1}(\a_{v_{k+1}})|^2 &\leq \E |f(\a_{v_k}) - \f_k(\a_{v_k})|^2 - \eta^{C_2} \\
d_2(v_{k+1}) &\leq d_2(v_k)+1
\end{align*}
or the dimension decrement bounds
\begin{align*}
\E |f(\a_{v_{k+1}}) - \f_{k+1}(\a_{v_{k+1}})|^2 &\leq \E |f(\a_{v_k}) - \f_k(\a_{v_k})|^2 + \eta^{3C_2} \\
 d_2(v_{k+1}) &\leq d_2(v_k), \\
 d_2^{\mathrm{poor}}(v_{k+1}) &\leq d_2^{\mathrm{poor}}(v_k)-1.
\end{align*}

Since $v_0$ already has the minimum quadratic dimension $d_2^{\mathrm{poor}}(v_0)=0$, we see that we must experience an energy decrement at the $k=0$ stage.  Also, if $k$ is the $j^{\operatorname{th}}$ index to experience an energy decrement, we see that $d_2^{\mathrm{poor}}(v_{k+1}) \leq d_2(v_{k+1}) \leq j$, and so one can have at most $j$ consecutive dimension decrements after the $k^{\operatorname{th}}$ stage; in other words, we must experience another energy decrement within $j+1$ steps.  By definition of $k_0$, we have $\sum_{0 \leq j \leq 2 \eta^{-C_2}} (j+1) < k_0$ if $C_2$ is large enough.  We conclude that at least $2 \eta^{-C_2}$ energy decrements occur within the path $v_0 \to \dots \to v_{k_0+1}$.  This implies that
$$  \E |f(\a_{v_{k_0+1}}) - \f_{k_0+1}(\a_{v_{k_0+1}})|^2  \leq  \E |f(\a_{v_{0}}) - \f_{k+1}(\a_{v_{0}})|^2  - (2 \eta^{-C_2}) \eta^{C_2} + k_0 \eta^{3C_2}.$$
But if $C_2$ is sufficiently large, this implies from \eqref{kdef} that
$$  \E |f(\a_{v_{k_0+1}}) - \f_{k_0+1}(\a_{v_{k_0+1}})|^2  <  \E |f(\a_{v_0}) - \f_0(\a_{v_0})|^2 - 4$$
(say), which leads to a contradiction because the left-hand side is clearly non-negative, and the right-hand side non-positive.
This gives the desired contradiction that establishes Theorem \ref{khint} and hence Theorem \ref{main}.

It remains to establish Proposition \ref{main-abstract}.  This will occupy the remaining portions of the paper.

\section{Bohr sets}

In order to define and manipulate the ``structured local approximants'' that appear in Proposition \ref{main-abstract}, we will need to develop the theory of two mathematical objects.  The first is that of a \emph{Bohr set}, which will be covered in this section; the second is that of a \emph{dilated torus}, which we will discuss in the next section.

\begin{definition}[Bohr set] A subset $S$ of $\Z/p\Z$ is said to be \emph{non-degenerate} if it contains at least one non-zero element. In this case we define the dual $S$-norm
$$ \|a\|_{S^\perp} \coloneqq \sup_{\xi \in S} \left\| \frac{a \xi}{p} \right\|_{\R/\Z} $$
for any $a \in \Z/p\Z$, and then define the \emph{Bohr set} $B(S,\rho) \subset \Z/p\Z$ for any $\rho>0$ by the formula
$$ B(S,\rho) \coloneqq \left\{ a \in \Z/p\Z: \|a\|_{S^\perp} < \rho \right\}$$
where $\|\theta\|_{\R/\Z}$ denotes the distance from $\theta$ to the nearest integer.  We refer to $S$ as the \emph{set of frequencies} of the Bohr set, $\rho$ as the \emph{radius}, and $|S|$ as the \emph{rank} of the Bohr set.  We also define the shifted Bohr sets
$$ n + B(S,\rho) \coloneqq \{ a+n: a \in B(S,\rho) \}$$
for any $n \in \Z/p\Z$.
\end{definition}

From \eqref{rz-tri} we have the triangle inequalities
\begin{equation}\label{sperp-tri}
\|a+b\|_{S^\perp} \leq \|a\|_{S^\perp} + \|b\|_{S^\perp}; \quad \|ka\|_{S^\perp} \leq |k| \|a\|_{S^\perp}
\end{equation}
for $a,b \in \Z/p\Z$ and $k \in \Z$; also we trivially have
$$ \|a\|_{S^\perp} \leq \|a\|_{(S')^\perp}$$
if $S \subset S'$ and $a \in \Z/p\Z$, or equivalently that $B(S',\rho) \subset B(S,\rho)$ for $\rho>0$.  We will frequently use these inequalities in the sequel, usually without further comment.  In Lemma \ref{edual} below, we will show that $\| \|_{S^\perp}$ is ``dual'' to a certain word norm $\| \|_S$ on $\Z/p\Z$.  One could also define Bohr sets in the case when $S$ is degenerate, but this creates some minor complications in our arguments, so we remove this case from our definition of a Bohr set.

We have the following standard size bounds for Bohr sets, whose proof may be found in \cite[Lemma 4.20]{taovu-book}.

\begin{lemma}\label{size-bohr}  If $B(S,\rho)$ is a Bohr set, then $|B(S,\rho)| \geq \rho^{|S|} p$
and $|B(S,2\rho)| \leq 4^{|S|} |B(S,\rho)|$.
\end{lemma}

In previous work on Roth-type theorems, one sometimes restricts attention to \emph{regular Bohr sets}, as first introduced in \cite{bourgain}; see \cite[\S 4.4]{taovu-book} for some discussion of this concept.  Due to our use of the probabilistic method, we will be able to work with a technically simpler and ``smoothed out'' version of a regular Bohr set, which we call the \emph{regular probability distribution} on a Bohr set.

\begin{definition}  Let $B(S,\rho)$ be a Bohr set.  The \emph{regular probability distribution} ${\mathfrak p}_{B(S,\rho)}: \Z/p\Z \to \R$ associated to $B(S,\rho)$ is the function defined by the formula
\begin{equation}\label{bsa}
 {\mathfrak p}_{B(S,\rho)}(a)  \coloneqq 2 \int_{1/2}^1 \frac{1_{B(S,t\rho)}(a)}{|B(S,t\rho)|}\ dt;
\end{equation}
it is easy to see \textup{(}from Fubini's theorem\textup{)} that this is indeed a probability distribution on $\Z/p\Z$.  A random variable $\a \in \Z/p\Z$ is said to be \emph{drawn regularly} from $B(S,\rho)$ if it has probability density function ${\mathfrak p}_{B(S,\rho)}$, thus $\P( \a = a ) = {\mathfrak p}_{B(S,\rho)}(a)$ for all $a \in \Z/p\Z$.

More generally, for any shifted Bohr set $n + B(S,\rho)$, we define the regular probability distribution ${\mathfrak p}_{n+B(S,\rho)}: \Z/p\Z \to \R$ by the formula
$$  {\mathfrak p}_{n+B(S,\rho)}(a) \coloneqq {\mathfrak p}_{B(S,\rho)}(a-n),$$
and say that $\a$ is \emph{drawn regularly} from $n + B(S,\rho)$ if it has probability distribution ${\mathfrak p}_{n+B(S,\rho)}$.
\end{definition}

Informally, to draw a random variable $\a$ regularly from $n+B(S,\rho)$, one should draw it uniformly from $n+B(S,{\mathbf t} \rho)$, where ${\mathbf t}$ is itself selected uniformly at random from the interval $[1/2,1]$.  Note that if $\a$ is drawn regularly from $n + B(S,\rho)$, then $m+\a$ will be drawn regularly from $m+n+B(S,\rho)$ for any $m \in \Z/p\Z$, and similarly $k\a$ will be drawn from $kn + B(k^{-1} \cdot S, \rho)$ for any non-zero $k \in \Z/p\Z$, where $k^{-1} \cdot S \coloneqq \{ k^{-1} \xi: \xi \in S \}$ is the dilate of the frequency set $S$ by $k^{-1}$.

From Lemma \ref{size-bohr} we see that if $\a$ is drawn regularly from a shifted Bohr set $n+B(S,\rho)$, then
\begin{equation}\label{theta-crude}
\P(\a = a) \leq \frac{1}{(\rho/2)^{|S|} p}
\end{equation}
for all $a \in \Z/p\Z$.  In practice, this will mean that the influence of any given value of $\a$ will be negligible.

The presence of the averaging parameter $t$ in \eqref{bsa} allows for the following very convenient approximate translation invariance property.  Given two random variables $\a, \a'$ taking values in a finite set $A$, we define the \emph{total variation distance} between the two to be the quantity
$$ d_\TV(\a,\a') \coloneqq \sum_{a \in A} |\P(\a = a) - \P(\a' = a)|,$$
or equivalently 
$$ d_\TV(\a,\a') = \sup_f |\E f(\a) - \E f(\a')|$$
where $f: A \to \C$ ranges over $1$-bounded functions.  

The next lemma gives some approximate translation-invariance properties of Bohr sets. Its proof is a thinly disguised version of the arguments of Bourgain \cite{bourgain}. 

\begin{lemma}\label{ati}  Let $n+B(S,\rho)$ be a shifted Bohr set, and let $\a$ be drawn regularly from $B(S,\rho)$.  Let $B(S',\rho')$ be another Bohr set with $S' \supset S$.
\begin{itemize}
\item[(i)]  If $h \in B(S', \rho')$, then $\a$ and $\a+h$ differ in total variation by at most $O( |S| \frac{\rho'}{\rho} )$.
\item[(ii)]  More generally, if $\h$ is a random variable independent of $\a$ that takes values in $B(S',\rho')$, then $\a$ and $\a+\h$ differ in total variation by at most $O( |S| \frac{\rho'}{\rho} )$.
\end{itemize}
\end{lemma}

\begin{proof}  To prove (i), it suffices to show that
$$ \E f(\a+h) = \E f(\a) + O\left( |S| \frac{\rho'}{\rho} \right)$$
for any $1$-bounded function $f: \Z/p\Z \to \C$; the claim (ii) then also follows by conditioning $\h$ to a fixed value $h \in B(S',\rho')$, then multiplying by $\P(\h=h)$ and summing over $h$.  

By translating $f$ by $n$, we may assume that $n=0$.  We may assume that $\rho' \leq \frac{\rho}{10|S|}$, as the claim is trivial otherwise.

From \eqref{bsa} we have 
$$ \E f(\a) = 2 \int_{1/2}^1 \sum_{a \in \Z/p\Z} f(a) \frac{1_{B(S,t\rho)}(a)}{|B(S,t\rho)|}\ dt $$
and
$$ \E f(\a+h) = 2 \int_{1/2}^1 \sum_{a \in \Z/p\Z} f(a) \frac{1_{B(S,t\rho)-h}(a)}{|B(S,t\rho)|}\ dt $$
so by the triangle inequality it suffices to show that
\begin{equation}\label{avg}
 \int_{1/2}^1 \frac{\sum_{a \in \Z/p\Z} |1_{B(S,t\rho)}(a) - 1_{B(S,t\rho)-h}(a)|}{|B(S,t\rho)|}\ dt \ll |S| \frac{\rho'}{\rho}.
\end{equation}
By the triangle inequality, the integrand here is bounded above by $2$.  Also, from \eqref{sperp-tri}, we see that any $a$ for which $1_{B(S,t\rho)-h}(a) \neq 1_{B(S,t\rho)}(a)$ lies in the ``annulus'' $B(S,t\rho + \rho') \backslash B(S,t\rho - \rho')$.  We conclude that the left-hand side of \eqref{avg} is bounded by
$$ \int_{1/2}^1 O\left( \min\left( \frac{|B(S,t\rho+\rho')| - |B(S,t\rho-\rho')|}{|B(S,t\rho-\rho')|}, 1 \right) \right)\ dt$$
which, using the elementary bound $\min(x-1,1) \ll \log x$ for $x \geq 1$, can be bounded in turn by
$$ O\left( \int_{1/2}^1 \log \frac{|B(S,t\rho+\rho')|}{|B(S,t\rho-\rho')|} \ dt \right).$$
The integral telescopes to
$$ O\left( \int_1^{1+\rho'/\rho} \log |B(S,t\rho)|\ dt - \int_{1/2-\rho'/\rho}^{1/2} \log |B(S,t\rho)|\ dt \right)$$
which can be bounded in turn by
$$ O\left( \frac{\rho'}{\rho} \log \frac{|B(S,2\rho)|}{|B(S,\rho/4)|} \right).$$
The claim now follows from Lemma \ref{size-bohr}.
\end{proof}

We will be interested in the Fourier coefficients $\E e_p(\lambda \n) = \E e( \frac{\lambda \n}{p} )$ of random variables $\n$ drawn regularly from Bohr sets $B(S,\rho)$.  As was noted by Bourgain \cite{bourgain}, these coefficients are controlled by a ``word norm'' $\| \|_S$, defined as follows:

\begin{definition}[Word norm]\label{word-def}
If $S \subset \Z/p\Z$ is non-degenerate, and $a$ is an element of $\Z/p\Z$, we define the \emph{word norm} $\|a\|_S$ of $a$ to be the minimum value of $\sum_{s \in S} |n_s|$, where $(n_s)_{s \in S} \in \Z^S$ ranges over tuples of integers such that one has a representation $a = \sum_{s \in S} n_s s$; note that such a representation always exists because $S$ is non-degenerate.  
\end{definition}

Similarly to \eqref{sperp-tri}, we observe the triangle inequalities
\begin{equation}\label{a-tri}
\|a+b\|_{S} \leq \|a\|_{S} + \|b\|_{S}; \quad \|ka\|_{S} \leq |k| \|a\|_{S}
\end{equation}
for $a,b \in \Z/p\Z$ and $k \in \Z$, which we will use frequently in the sequel, often without further comment.

We now give a duality relationship between the word norm $\| \|_S$ and the dual $S$-norm $\| \|_{S^\perp}$:

\begin{lemma}[Duality]\label{edual}  Let $S$ be a non-degenerate subset of $\Z/p\Z$, and let $\lambda \in \Z/p\Z$.
\begin{itemize}
\item[(i)]  For every $n \in \Z/p\Z$, one has $\| \frac{n \lambda}{p} \|_{\R/\Z} \leq \| n \|_{S^\perp} \| \lambda \|_S$.
\item[(ii)]  Conversely, if one has the estimate $\| \frac{n \lambda}{p} \|_{\R/\Z} \leq A \| n \|_{S^\perp}$ for some $A \geq 1$ and all $n \in \Z/p\Z$, then $\| \lambda \|_S \ll |S|^{3/2} A$.
\end{itemize}
\end{lemma}

\begin{proof}  To prove (i), we simply observe (using \eqref{rz-tri}) that for any $n \in \Z/p\Z$, one has $\| n \lambda/p\|_{\R/\Z} =$
\[
 = \left\| \sum_{\xi \in S} a_\xi \frac{n \xi}{p} \right\|_{\R/\Z} \leq \sum_{\xi \in S} |a_\xi| \left\| \frac{n\xi}{p} \right\|_{\R/\Z} \leq \sum_{\xi \in S} |a_\xi| \left\|n\right\|_{S^\perp} \leq \| \lambda \|_S \|n\|_{S^\perp}
\]
as desired, where $\lambda = \sum_{\xi \in S} a_\xi \xi$ is a representation of $\lambda$ that minimises $\sum_{\xi \in S} |\xi|$.

Estimates such as (ii) go back to the work of Bourgain \cite{bourgain}.  We will prove this claim by a Fourier-analytic argument.  We may assume that $\|\lambda \|_S \geq |S|^{3/2}$, as the claim is trivial otherwise.  Let $\psi: \R \to \R$ be a non-negative smooth even function (not depending on $p$ or $\lambda$) supported on $[-1,1]$ and non-zero on $[-1/2,1/2]$, whose Fourier transform $\hat \psi(\xi) \coloneqq \int_\R \psi(x) e(-\xi x)\ dx$ is also non-negative.  Set $N \coloneqq |S|^{-1} \| \lambda \|_S$, so in particular $N \geq 1$.  We consider the kernel $K_N: \Z/p\Z \to \C$ defined by
$$ K_N(n) \coloneqq \sum_{k \in \Z} e_p(kn) \psi( \frac{k}{N} );$$
by the Poisson summation formula we have
$$ K_N(n \md{p}) = N \sum_{m \in \Z} \hat \psi\left( \frac{Nn}{p} - N m\right)$$
for any integer $n$, so in particular $K_N$ is non-negative.

By definition of $N$, the frequency $\lambda$ has no representations of the form $\lambda = \sum_{\xi \in S} a_\xi \xi$ with $\sup_{\xi \in S} |a_\xi| < N$.  Hence the Riesz-type product $\prod_{\xi \in S} K_N(\xi n)$, when expanded, contains no terms of the form $e_p( \lambda n )$ or $e_p( -\lambda n )$, and is therefore orthogonal to $\cos(\frac{2\pi \lambda n}{p})$.  In particular we have the identity
$$ {\mathbb E}_{n \in \Z/p\Z} \prod_{\xi \in S} K_N(\xi n) = {\mathbb E}_{n \in \Z/p\Z} \left(1 - \cos\left( \frac{2\pi \lambda n}{p} \right) \right) \prod_{\xi \in S} K_N(\xi n) .$$
On the other hand, from two applications of \eqref{cos-bound} we have
\begin{align*}
1 - \cos( \frac{2\pi \lambda n}{p} ) &\ll \left\| \frac{\lambda n}{p} \right\|_{\R/\Z}^2 
\leq A^2 \| n \|_{S^\perp}^2 \\
& \leq A^2 \sum_{\xi_0 \in S} \left\| \frac{\xi_0 n}{p} \right\|_{\R/\Z}^2 
\leq A^2 \sum_{\xi_0 \in S} \left(1 - \cos\left( \frac{2\pi \xi_0 n}{p}\right)\right).
\end{align*}
As $K_N$ is non-negative, we conclude that
\begin{align}\nonumber
 & {\mathbb E}_{n \in \Z/p\Z}  \prod_{\xi \in S} K_N(\xi n) \\ & \ll A^2 \sum_{\xi_0 \in S} {\mathbb E}_{n \in \Z/p\Z} 
\left(\bigg(\prod_{\xi \in S \backslash \xi_0} K_N(\xi n)\bigg) K_N(\xi_0 n) \left(1 - \cos( \frac{2\pi \xi_0 n}{p} \right) \right).\label{ens}
\end{align}
We can expand $K_N(\xi_0 n) \left(1 - \cos\left( \frac{2\pi \xi_0 n}{p} \right) \right)$ as a Fourier series
$$ \sum_{k \in \Z} e_p(kn) \left(\psi\left(\frac{k}{N}\right) - \frac{\psi\left(\frac{k-1}{N}\right) + \psi\left(\frac{k+1}{N}\right)}{2}\right).$$
The expression inside parentheses is only non-vanishing for $|k| \leq N+1$, and has magnitude $O(1/N^2)$.  As $\psi$ is non-negative everywhere and non-zero on $[-1/2,1/2]$, we thus have a pointwise estimate of the form
$$
\psi\left(\frac{k}{N}\right) - \frac{\psi\left(\frac{k-1}{N}\right) + \psi\left(\frac{k+1}{N}\right)}{2} \ll \frac{1}{N^2} \sum_{j=-8}^{8} \psi\left( \frac{k}{N} - \frac{j}{4} \right)$$
(say).  By using the non-negativity of the Fourier coefficients of $K_N$, this gives the estimate

\begin{align*} {\mathbb E}_{n \in \Z/p\Z} 
\left(\prod_{\xi \in S \backslash \xi_0}  K_N(\xi n)\right) & K_N(\xi_0 n) \left(1 - \cos\left( \frac{2\pi \xi_0 n}{p} \right) \right)
\\ & \ll \frac{1}{N^2} {\mathbb E}_{n \in \Z/p\Z} \prod_{\xi \in S} K_N(\xi n).\end{align*}
Comparing this with \eqref{ens}, we conclude that $1 \ll  A^2 |S|/N^2$, and the claim follows from the definition of $N$.
\end{proof}

Next, we estimate the Fourier coefficients of a regular distribution on a Bohr set in terms of the word norm.

\begin{lemma}\label{fde}  Let $S$ be a non-degenerate subset of $\Z/p\Z$. Suppose that $\n$ is drawn regularly from $B(S,\rho)$. Then we have
$$
\E e_p(\lambda \n) \ll \frac{|S|^{5/2}}{\rho \|\lambda \|_{S}} 
$$
for all $\lambda \in \Z/p\Z$, where we adopt the convention that the above estimate is vacuously true if $\|\lambda\|_{S} = 0$.
\end{lemma}

\begin{proof}  For any $h \in \Z/p\Z$, one has from Lemma \ref{ati} that
$$ \E e_p( \lambda \n ) = \E e_p( \lambda (\n+h) ) + O\left(\frac{ |S| \|h\|_{S^\perp}}{\rho} \right) $$
which we may rearrange as
$$ \left(1 - e_p(\lambda h)\right) \E e_p( \lambda \n ) \ll\frac{ |S| \|h\|_{S^\perp}}{\rho}.$$
Since $|1 - e_p(\lambda h)| \gg \|\frac{\lambda h}{p} \|_{\R/\Z}$, we conclude that
$$ \|\frac{\lambda h}{p} \|_{\R/\Z} \E e_p( \lambda \n ) \ll\frac{ |S| \|h\|_{S^\perp}}{\rho}.$$ Taking $h$ so as to minimise the ratio $\Vert h \Vert_{S^{\perp}}/\Vert \lambda h/p \Vert_{\R/\Z}$, the claim follows from Lemma \ref{edual}.
\end{proof}

We will take advantage of the fact that Bohr sets can be approximately described as generalised arithmetic progressions.  A key lemma in this regard is the following.

\begin{lemma}\label{bohr-basis}  Let $\Gamma$ be a lattice in $\R^d$.  Then there exist linearly independent generators $v_1,\dots,v_d$ of $\Gamma$ and real numbers $N_1,\dots,N_d > 0$ such that
\begin{equation}\label{bvta}
 B_{\R^d}(0, O(d)^{-3d/2} t ) \cap \Gamma \subset \{ \sum_{i=1}^d n_i v_i: |n_i| < tN_i \} \subset B_{\R^d}(0,t) \cap \Gamma
\end{equation}
for all $t>0$, where $B_{\R^d}(0,r)$ is the open Euclidean ball of radius $r$ in $\R^d$, and the $n_i$ are understood to be integers.  Furthermore, the determinant/covolume $\det(\Gamma)$ obeys the bounds
\begin{equation}\label{covol}
 \det(\Gamma) = (2d)^{O(d)} \prod_{i=1}^d N_i^{-1}.
\end{equation}
\end{lemma}

\begin{proof}  Applying \cite[Theorem 1.6]{taovu-john}, we can find elements $v_1,\dots,v_r$ of $\Gamma$ for some $r \leq d$, linearly independent over the rationals, and real numbers $N_1,\dots,N_d > 0$ such that
\begin{equation}\label{bvt}
 B_{\R^d}(0, O(d)^{-3d/2} t ) \cap \Gamma \subset \{ \sum_{i=1}^r n_i v_i: |n_i| < tN_i \} \subset B_{\R^d}(0,t) \cap \Gamma
\end{equation}
for all $t>0$, and such that
$$ O(d)^{-7d/2} |B_{\R^d}(0,t) \cap \Gamma| \leq |\{ \sum_{i=1}^r n_i v_i: |n_i| < tN_i\}| \leq |B_{\R^d}(0,t) \cap \Gamma|.$$
(Strictly speaking, the statement of \cite[Theorem 1.6]{taovu-john} only claims the latter bound for $t=1$, but the same argument gives the bound for all $t>0$.)  Sending $t$ to infinity, we conclude that the $v_1,\dots,v_r$ generate $\Gamma$; since, by virtue of being a lattice, $\Gamma$ is cocompact, this forces $d=r$.  Also, volume packing arguments show that as $t \to \infty$, the cardinality $|B_{\R^d}(0,t) \cap \Gamma|$ is asymptotic to the measure of $B_{\R^d}(0,t)$ divided by $\det(\Gamma)$, while the cardinality of $|\{ n_1 v_1 + \dots + n_d v_d: |n_i| \leq tN_i \}|$ is asymptotic to $\prod_{i=1}^{d} (2 t N_{i})$.  We conclude \eqref{covol} as desired.
\end{proof}

The following corollary describes how we may pick a ``basis'' for a Bohr set.

\begin{corollary}\label{bohr-basis-cor}  Let $S$ be a non-degenerate subset of $\Z/p\Z$, and set $d \coloneqq |S|$.  Then there exist elements $a_1,\dots,a_d$ of $\Z/p\Z$ and real numbers $N_1,\dots,N_d > 0$ such that
\begin{equation}\label{nip}
 \prod_{i=1}^d N_i^{-1} = (2d)^{O(d)} p
\end{equation}
and
\begin{equation}\label{ain}
 \|a_i\|_{S^\perp} \leq N_i^{-1}
\end{equation}
for all $i=1,\dots,d$.  Furthermore, for any $a \in \Z/p\Z$, there exists a representation
\begin{equation}\label{nag}
 a = n_1 a_1 + \dots + n_d a_d
\end{equation}
with $n_1,\dots,n_d$ integers of size
\begin{equation}\label{nis}
 n_i = (2d)^{O(d)} N_i \| a\|_{S^\perp}
\end{equation}
for $i=1,\dots,d$.  Finally, if one imposes the additional condition $|n_i| < N_i/2$ for all $i=1,\dots,d$, then there is at most one such representation of this form \eqref{nag} for a given $a$.
\end{corollary}

\begin{proof} For each $s \in S$, the fraction $\frac{s}{p}$ can be viewed as an element of $\R/\Z$ of order at most $p$; as $S$ is non-degenerate, we see that the tuple $(\frac{s}{p})_{s \in S}$ is an element of the torus $(\R/\Z)^S$ of order $p$.  Let $\Gamma$ be the preimage in $\R^S$ of the group generated by this element, thus $\Gamma$ is a lattice of $\R^S$ that contains $\Z^S$ as a sublattice of index $p$; in particular, $\Gamma$ has determinant $p$.  Applying Lemma \ref{bohr-basis}, one can find generators $v_1,\dots,v_d$ of $\Gamma$ and real numbers $N_1,\dots,N_d$ obeying \eqref{nip} such that
\begin{equation}\label{bvta-0}
 B_{\R^S}(0, O(d)^{-3d/2} t ) \cap \Gamma \subset \{ \sum_{i=1}^d n_i v_i: |n_i| < tN_i \} \subset B_{\R^S}(0,t) \cap \Gamma
\end{equation}
for all $t>0$.

By construction of $\Gamma$, we can find elements $a_1,\dots,a_d$ of $\Z/p\Z$ such that 
\begin{equation}\label{vis}
v_i = \left(\frac{a_i s}{p}\right)_{s \in S} \md{\Z^S}
\end{equation}
for $i=1,\dots,d$.  Applying \eqref{bvta-0} with $t$ slightly larger than $N_i^{-1}$ for some $i=1,\dots,d$, we see that $v_i \in B_{\R^d}( N_i^{-1})$, and hence by \eqref{vis} we have \eqref{ain}.

Finally, if $a \in \Z/p\Z$, then by definition of $\Gamma$ we can find an element $x$ of $\Gamma$ in the preimage of $(\frac{as}{p})_{s \in S}$
such that each component of $x$ has magnitude less than $\|a\|_{S^\perp}$; in particular, $x \in B_{\R^S}(0, \sqrt{d} \|a\|_{S^\perp})$.  Applying \eqref{bvta-0}, we conclude that $x = \sum_{i=1}^d n_i v_i$ for some integers $n_1,\dots,n_d$ obeying \eqref{nis}, giving the desired representation \eqref{nag}.

Finally, we show uniqueness.  If there were two representations of the form \eqref{nag} with $|n_i| < N_i/2$ for all $i=1,\dots,d$, then there exists a tuple $(n'_1,\dots,n'_d) \in \Z^d$, not identically zero, with $|n'_i| < N_i$ for all $i=1,\dots,d$ and
$\sum_{i=1}^d n_i a_i = 0$, which implies that the vector $\sum_{i=1}^d n_i v_i$ lies in $\Z^S$.  As the $v_1,\dots,v_d$ are linearly independent, this vector must have magnitude at least $1$; but this contradicts \eqref{bvta} (with $t=1$).
\end{proof}

\emph{Linear and quadratic functions on Bohr sets.} We will frequently need to deal with locally linear or quadratic functions on Bohr sets.  We review the definitions of these now.

\begin{definition}  Let $B$ be a subset of $\Z/p\Z$, and let $G = (G,+)$ be an abelian group.  A function $\phi: B \to G$ is said to be \emph{locally linear} on $B$ if one has
$$ \phi(n+h_1+h_2) - \phi(n+h_1) - \phi(n+h_2) + \phi(n) = 0 $$
whenever $n,h_1,h_2 \in \Z/p\Z$ are such that $n,n+h_1,n+h_2,n+h_1+h_2 \in B$.  Similarly, $\phi$ is said to be \emph{locally quadratic} on $B$ if one has
\begin{equation}\label{omh}
 \sum_{(\omega_1,\omega_2,\omega_3) \in \{0,1\}^3} (-1)^{\omega_1+\omega_2+\omega_3} \phi(n+\omega_1 h_1+\omega_2 h_2 + \omega_3 h_3) = 0
\end{equation}
whenever $n,h_1,h_2,h_3 \in \Z/p\Z$ are such that $n + \omega_1 h_1 + \omega_2 h_2 + \omega_3 h_3 \in B$ for all $(\omega_1,\omega_2,\omega_3) \in \{0,1\}^3$.

A function $\psi: B \times B \to G$ is said to be \emph{locally bilinear} on $B$ if one has
$$ \psi(h_1+h'_1, h_2) = \psi(h_1,h_2) + \psi(h'_1,h_2)$$
whenever $h_1,h'_1,h_2 \in B$ are such that $h_1+h'_1 \in B$, and similarly one has
$$ \psi(h_1,h_2+h'_2) = \psi(h_1,h_2) + \psi(h_1,h'_2)$$
whenever $h_1,h_2,h'_2 \in B$ are such that $h_2+h'_2 \in B$.
\end{definition}

Specialising \eqref{omh} to the case $h_1=h_2=h_3=h$, we conclude that
\begin{equation}\label{omh-ap}
\phi(n) - 3 \phi(n+h) + 3 \phi(n+2h) - \phi(n+3h) = 0
\end{equation}
whenever $\phi: B \to G$ is locally quadratic on $B$ and $n,n+h,n+2h,n+3h \in B$.

It is well known (from the Weyl exponential sum estimates) that quadratic exponential sums such as ${\mathbb E}_{1 \leq n \leq N} e(\alpha n^2 + \beta n)$ can only be large when the quadratic phase $\alpha n^2$ is of ``major arc'' type in the sense that $k \alpha n^2$ is close to constant on the range $\{1,\dots,N\}$ of the summation variable $n$, for some bounded positive integer $k$.  The following proposition is an analogue of this phenomenon on Bohr sets.

\begin{proposition}[Large local quadratic exponential sums]\label{large-quadratic}  Let $B(S,\rho)$ be a Bohr set, let $0 < \delta \leq 1/2$, let $\lambda, \mu: B(S,10 \rho) \to \R/\Z$ be locally linear maps, and let $\phi: B(S,10\rho) \times B(S,10\rho) \to \R/\Z$ be a locally bilinear phase such that
\begin{equation}\label{ephin}
|\E e( \phi( \n, \m ) + \lambda(\n) + \mu(\m) )| \geq \delta
\end{equation}
if $\n, \m$ are drawn independently and regularly from $B(S,\rho)$.  Then there exists a natural number
$$ 1 \leq k \leq \delta^{-O(C_1|S|^2)}$$
such that
\begin{equation}\label{ko}
 \| k \phi(n,m) \|_{\R/\Z} \ll \delta^{-O(C_1|S|^2)} \frac{\|n\|_S \|m\|_S}{\rho^2}
\end{equation}
whenever $n,m \in B\left(S, \frac{\delta^{C_1} \rho}{(C_1|S|)^{3 |S|}} \right)$.
\end{proposition}

\begin{proof} 
Let $d \coloneqq |S|$, thus $d \geq 1$.  By Corollary \ref{bohr-basis-cor}, we can find elements $a_1,\dots,a_d$ of $\Z/p\Z$ and real numbers $N_1,\dots,N_d$ obeying the conclusions of that corollary.

Suppose that $1 \leq i,j \leq d$ are such that $N_i, N_j \geq \frac{d}{\delta^{C_1/2} \rho}$ (we allow $i$ and $j$ to be equal).  Then by \eqref{ain} we have
$$ \|a_i\|_{S^\perp}, \|a_j\|_{S^\perp} \leq d^{-1} \delta^{C_1/2} \rho.$$
We can control the coefficient $\phi(a_i,a_j)$ by the following argument. If we draw $\b_i$ and $\b_j$ uniformly from $\{ b_i \in \Z: 1 \leq b_i \leq \delta^{C_1/4} N_i \rho/d \}$ and $\{ b_j \in \Z: 1 \leq b_j \leq \delta^{C_1/4} N_j \rho/d \}$ respectively and independently of each other and of $\n,\m$, then from two applications of Lemma \ref{ati} (comparing $\n$ with $\n+\b_ia_i$, and $\m$ with $\m+\b_j a_j$) we have
\begin{align*} \E e( \phi( \n + \b_i a_i, \m + \b_j a_j ) + \lambda(\n & + \b_i a_i) + \mu(\m + \b_j a_j) ) \\ &
= \E e( \phi( \n, \m) + \lambda(\n) + \mu(\m ) ) + O( \delta^{C_1/4})\end{align*}
and hence from \eqref{ephin} (assuming $C_1$ large enough) we have
$$ |\E e( \phi( \n + \b_i a_i, \m + \b_j a_j ) + \lambda(\n + \b_i a_i) + \mu(\m + \b_j a_j) )| \gg \delta.$$
By the pigeonhole principle, we can therefore find $n,m \in B(S,\rho)$ such that
$$ |\E e( \phi( n + \b_i a_i, m + \b_j a_j ) + \lambda(n + \b_i a_i) + \mu(m + \b_j a_j) )| \gg \delta.$$
Using the local bilinearity of $\phi$, the left-hand side may be written as
$$ |\E e( \b_i \b_j \phi(a_i,a_j) + \alpha \b_i + \beta \b_j + \gamma )|$$
for some $\alpha,\beta,\gamma \in\R/\Z$ depending on $i,j,n,m$ whose exact values are not of importance to us.  Evaluating the expectations and using the triangle inequality, we conclude that
$$ {\mathbb E}_{1 \leq b_i \leq \delta^{C_1/4} N_i \rho/d} | {\mathbb E}_{1 \leq b_j \leq \delta^{C_1/4} N_j \rho/d} e( b_j (b_i \phi(a_i,a_j) + \beta) )| \gg \delta $$
and hence (by Lemma \ref{popular})
$$| {\mathbb E}_{1 \leq b_j \leq \delta^{C_1/4} N_j \rho/d} e( b_j (b_i \phi(a_i,a_j) + \beta) )| \gg \delta $$
for $\gg \delta^{C_1/4 + 1} N_i \rho/d$ values of $b_i$ in the range $1 \leq b_i \leq \delta^{C_1/4} N_i \rho/d$.  This average is a geometric series that can be explicitly computed, leading to the bound
$$ \| b_i \phi(a_i,a_j) + \beta \|_{\R/\Z} \ll \frac{d}{\delta^{\frac{C_1}{4}+1} N_j \rho}$$
for $\gg \delta^{C_1/4 + 1} N_i \rho/d$ values of $b_i$ in the range $1 \leq b_i \leq \delta^{C_1/4} N_i \rho/d$.  Applying \cite[Lemma A.4]{gt-mobius} (which is really an observation of Vinogradov, used often in the theory of Weyl sums), we conclude that 
$$ \| k_{i,j} \phi(a_i,a_j) \|_{\R/\Z} \ll \frac{d^2}{\delta^{O(C_1)} N_i N_j \rho^2}$$
for some natural number $k_{i,j}$ with $1 \leq k_{i,j} \ll \delta^{-O(C_1)}$.  If we then ``clear denominators'' by defining 
$$k \coloneqq \prod_{1 \leq i,j \leq d: N_i, N_j \geq \frac{d}{\delta^{C_1/2} \rho}} k_{i,j},$$
then $1 \leq k \ll \delta^{-O(C_1 d^2)}$ and
\begin{equation}\label{nast}
 \| k \phi(a_i,a_j) \|_{\R/\Z} \ll \frac{1}{\delta^{O(C_1 d^2)} N_i N_j \rho^2}
\end{equation}
for all $1 \leq i,j \leq d$ with $N_i, N_j \geq \frac{d}{\delta^{C_1/2} \rho}$.

For any $n \in \Z/p\Z$, we see from Corollary \ref{bohr-basis-cor} that we can find integers $n_1,\dots,n_d$ with
$$ n_i \ll (2d)^{O(d)} N_i \|n\|_{S^\perp}$$
such that
$$ n = n_1 a_1 + \dots + n_d a_d.$$
In particular, if $n \in B(S, \frac{\delta^{C_1} \rho}{(C_1 d)^{3 d}})$, then $n_i$ is only non-zero when $N_i \geq \frac{d}{\delta^{C_1/2} \rho}$.  From these bounds, \eqref{nast}, and the local bilinearity of $\phi$, we conclude \eqref{ko} as desired.
\end{proof}

\emph{Local $U^2$-inverse theorem.} The global inverse $U^2$ theorem, which is a simple and well-known exercise in discrete Fourier analysis, asserts that if a $1$-bounded function $f: \Z/p\Z \to \C$ obeys the bound
\begin{equation}\label{fhhu}
 |\E f(\h_0+\h_1) \overline{f}(\h_0+\h'_1) \overline{f}(\h'_0+\h_1) f(\h'_0+\h'_1)| \geq \eta
\end{equation}
where $\h_0, \h_1, \h'_0, \h'_1$ are drawn uniformly at random from $\Z/p\Z$, then there exists $\xi \in \Z/p\Z$ such that
\begin{equation}\label{chc}
 |\E f(\h) e_p(-\xi \h)| \geq \eta^{1/2}
\end{equation}
where $\h$ is also drawn uniformly at random from $\Z/p\Z$.

In this section we give a local version of the above claim, in which the random variables $\h, \h_0,\h_1,\h'_0,\h'_1$ are localised to a small Bohr set.  If the rank of the Bohr set is bounded, one can modify the above arguments to obtain a reasonable inverse theorem of this nature, but in our application the rank of the Bohr set will be rather large, and it will be important that this rank does not affect the lower bound in correlations of the form \eqref{chc}.  Fortunately, such a result is available, and will be crucial in the proofs of the two remaining claims 
(Corollary \ref{loc-to-glob} and Theorem \ref{locu3}) needed to prove Theorem \ref{main}.

Here is a precise version of the claim.

\begin{theorem}\label{locu2}  Let $S \subset \Z/p\Z$ be non-degenerate for some prime $p$, and let $0 < \eta < 1/2$.  Let $\rho_0, \rho_1$ be real parameters with $0 < \rho_1 < \rho_0 < 1/2$ and such that
\begin{equation}\label{pron}
\rho_0 > \frac{C |S|}{\eta^2} \rho_1
\end{equation}
 for a sufficiently large absolute constant $C$.  Let $f: \Z/p\Z \to \C$ be a $1$-bounded function such that
\begin{equation}\label{soap}
 |\E f(\h_0+\h_1) \overline{f}(\h_0+\h'_1) \overline{f}(\h'_0+\h_1) f(\h'_0+\h'_1)| \geq \eta
\end{equation}
where $\h_0,\h'_0,\h_1,\h'_1$ are drawn independently and regularly from $B(S,\rho_0)$, $B(S,\rho_0)$, $B(S,\rho_1)$, $B(S,\rho_1)$ respectively.  Then there exists $\xi \in \Z/p\Z$ such that
$$ \sum_{n_0 \in \Z/p\Z} \P( \n_0 = n_0) \left| \E f(n_0+\n_1) e_p( -\xi \n_1 )\right|^2 \geq \eta/2$$
where $\n_0, \n_1$ are drawn independently and regularly from $B(S,\rho_0), B(S,\rho_1)$ respectively.
\end{theorem}

\begin{proof}  
We thank Fernando Shao for supplying a proof of this result, which was considerably simpler than our original argument.  

For this proof, which is Fourier-analytic in nature, it will be convenient to work explicitly with probability densities rather than probabilistic notation. (However, in the lengthier proof of the local inverse $U^3$ theorem given in the next section, the probabilistic notation will be significantly cleaner to use.)  In this argument, all sums will be over $\Z/p\Z$.  We abbreviate 
$$\p_i(h) \coloneqq \p_{B(S,\rho_i)}(h) = \P( \h_i = h )$$
for $i=0,1$ and $h \in \Z/p\Z$; clearly we have $\p_i(h) \geq 0$ and 
\begin{equation}\label{sph}
\sum_h \p_i(h) = 1.
\end{equation}
The hypothesis \eqref{soap} may be written as
\begin{align}\nonumber
\bigg |\sum_{h_0, h'_0, h_1, h'_1} \p_0(h_0) \p_0(h'_0) & \p_1(h_1) \p_1(h'_1) f(h_0+h_1) \overline{f}(h_0+h'_1)\times \\ & \times \overline{f}(h'_0+h_1) f(h'_0+h'_1)\bigg| \geq \eta  \label{soap-2}
\end{align} 
and our goal is to locate $\xi \in \Z/p\Z$ such that
$$ \sum_{n_0} \p_0(n_0) \left|\sum_{n_1} \p_1(n_1) f(n_0+n_1) e_p(-\xi n_1)\right|^2 \geq \eta/2.$$

The first step is to replace the factor $\p_0(h_0)$ by the slightly different factor $\p_0^{1/2}(h_0+h_1) \p_0^{1/2}(h_0+h'_1)$.  If we use the elementary inequality $|x^{1/2}-y^{1/2}| \leq |x-y|^{1/2}$ for $x,y \geq 0$ and then apply Cauchy-Schwarz, Lemma \ref{ati}, and \eqref{pron}, we see that
\begin{align*}
\sum_{h_0} \big|\p_0^{1/2}(h_0+h_1) - & \p_0^{1/2}(h_0)\big| \p_0^{1/2}(h_0)\\  &\leq
\sum_{h_0} |\p_0(h_0+h_1) - \p_0(h_0)|^{1/2} \p_0^{1/2}(h_0) \\
&\leq
\bigg(\sum_{h_0 \in \Z/p\Z} |\p_0(h_0+h_1) - \p_0(h_0)|\bigg)^{1/2} \\
&= (\sum_{h_0 \in \Z/p\Z} b_{h_1}(h_0) \p_0(h_0+h_1) - b_{h_1}(h_0) \p_0(h_0))^{1/2} \\
&\ll \left( \frac{|S| \rho_1}{\rho_0}\right)^{1/2} \ll \frac{\eta}{C^{1/2}}
\end{align*}
for any $h_1$ in the support of $\p_1$, where the $1$-bounded function $b_{h_1}$ is given by $b_{h_1}(h_0) \coloneqq \mathrm{sgn}(\p_0(h_0+h_1) - \p_0(h_0))$.  Similarly we have
$$ \sum_{h_0} |\p_0^{1/2}(h_0+h'_1) - \p_0^{1/2}(h_0)| \p_0^{1/2}(h_0+h_1) \ll \frac{\eta}{C^{1/2}}
$$
whenever $h'_1$ is also in the support of $\p_1$; by the triangle inequality, we conclude that
$$ \sum_{h_0} |\p_0^{1/2}(h_0+h_1) \p_0(h_0+h'_1)^{1/2} - \p_0(h_0)| \ll \frac{\eta}{C^{1/2}}
$$
for all $h_1,h'_1$ in the support of $\p_1$.  From the $1$-boundedness of $f$ and \eqref{sph}, we conclude that
\begin{align*}
&\biggl|\sum_{h_0, h'_0, h_1, h'_1}\ |\p_0^{1/2}(h_0+h_1) \p_0^{1/2}(h_0+h'_1) - \p_0(h_0)| \p_0(h'_0) \p_1(h_1) \p_1(h'_1) \\
&\quad f(h_0+h_1) \overline{f}(h_0+h'_1) \overline{f}(h'_0+h_1) f(h'_0+h'_1)\biggr| \ll \frac{\eta}{C^{1/2}}.
\end{align*}
If $C$ is large enough,  the left-hand side is thus bounded by $0.1 \eta$ (say), so by \eqref{soap-2} and the triangle inequality we conclude that
\begin{align*}
&\bigl|\sum_{h_0, h'_0, h_1, h'_1} \p_0^{1/2}(h_0+h_1) \p_0^{1/2}(h_0+h'_1) \p_0(h'_0) \p_1(h_1) \p_1(h'_1) \\
&\quad f(h_0+h_1) \overline{f}(h_0+h'_1) \overline{f}(h'_0+h_1) f(h'_0+h'_1)| \geq 0.9 \eta
\end{align*}
If we write 
\begin{equation}\label{fon}
f_0(n) \coloneqq f(n) \p_0^{1/2}(n),
\end{equation}
we may rewrite the above estimate as
\begin{align*}
&\bigl|\sum_{h_0, h'_0, h_1, h'_1} \p_0(h'_0) \p_1(h_1) \p_1(h'_1) \\
&\quad f_0(h_0+h_1) \overline{f_0}(h_0+h'_1) \overline{f}(h'_0+h_1) f(h'_0+h'_1)| \geq 0.9 \eta.
\end{align*}
A similar argument then lets us replace $\p_0(h'_0)$ with $\p_0^{1/2}(h'_0+h_1) \p_0^{1/2}(h'_0+h'_1)$, leaving us with
\begin{align*}
\bigl|\sum_{h_0, h'_0, h_1, h'_1} & \p_0(h'_0+h_1)^{1/2} \p_0(h'_0+h'_1)^{1/2}\p_1(h_1) \p_1(h'_1) \times  \\ &
\times  f_0(h_0+h_1) \overline{f_0}(h_0+h'_1) \overline{f}(h'_0+h_1) f(h'_0+h'_1)| \geq 0.8 \eta.
\end{align*}
which we can simplify using \eqref{fon} to
$$
 \left|\sum_{h_0, h'_0, h_1, h'_1} \!\!\!\! \p_1(h_1) \p_1(h'_1) f_0(h_0+h_1) \overline{f_0}(h_0+h'_1) \overline{f_0}(h'_0+h_1) f_0(h'_0+h'_1)\right| \geq 0.8 \eta.$$
Making the change of variables $n \coloneqq  h_1-h'_1$, we may rewrite the left-hand side as
$$ \sum_{n} (\p_1 * \tilde \p_1)(n) |(f_0 * \tilde f_0)(n)|^2$$
where $\tilde f_0(n) \coloneqq  \overline{f_0}(-n)$, and similarly for $p_1$, and $f*g$ denotes the discrete convolution
$$ f*g(n) \coloneqq \sum_m f(m) g(n-m)$$
Using the Fourier transform, we may then rewrite the previous bound as
\begin{equation}\label{prepig}
p^4 \sum_{\xi,\xi'} |\hat \p_1(\xi')|^2 |\hat f_0(\xi)|^2 |\hat f_0(\xi+\xi')|^2 \geq 0.8 \eta
\end{equation}
where 
$$ \hat f(\xi) \coloneqq \frac{1}{p} \sum_n f(n) e_p(-\xi n).$$
From \eqref{sph}, the $1$-boundedness of $f$, and the Plancherel identity we have
$$ \sum_{\xi} |\hat f_0(\xi)|^2  = \frac{1}{p} \sum_{n} |f_0(n)|^2 \leq \frac{1}{p}.$$
By this, \eqref{prepig}, and the pigeonhole principle, we may therefore find $\xi \in \Z/p\Z$ such that
$$ p^3 \sum_{\xi' \in \Z/p\Z} |\hat p_1(\xi')|^2 |\hat f_0(\xi+\xi')|^2 \geq 0.8 \eta.$$
By the Plancherel identity again, the left-hand side may be rewritten as
$$ \sum_{n_0} \left| \sum_{n_1}\ f_0(n_0-n_1) \p_1(n_1) e_p( \xi n_1 ) \right|^2$$
and hence (by replacing $n_1$ with $-n_1$ and using \eqref{fon})
$$ \sum_{n_0} \left| \sum_{n_1}\ f(n_0+n_1) \p_0^{1/2}(n_0+n_1) \p_1(n_1) e_p( -\xi n_1 ) \right|^2 \geq 0.8 \eta.$$
By argument similar to those at the beginning of the proof, we may replace $\p_0^{1/2}(n_0+n_1)$ by $\p_0^{1/2}(n_0)$ and conclude that
$$ \sum_{n_0} \left| \sum_{n_1}\ f(n_0+n_1) \p_0^{1/2}(n_0) \p_1(n_1) e( -\xi n_1 ) \right |^2 \geq 0.7 \eta,$$
and the claim follows.
\end{proof}

As a corollary of this inverse theorem, we can establish that locally almost linear phases on Bohr sets can be approximated by globally linear phases; this will be needed in Section \ref{baddim-sec} to deal with poorly distributed quadratic factors.  

Here is a precise statement.

\begin{corollary}\label{loc-to-glob}  Let $\phi: n_0 + B(S,\rho) \to \R/\Z$ be a function on a shifted Bohr set $n_0+B(S,\rho)$ which is ``locally almost linear'' in the sense that one has the bound
\begin{equation}\label{ppp}
 \| \phi(n_0+h+k) - \phi(n_0+h) - \phi(n_0+k) + \phi(n_0) \|_{\R/\Z} \leq A \frac{ \| h \|_{S^\perp} \| k \|_{S^\perp} }{\rho^2}
\end{equation}
for all $h,k \in B(S,\rho/2)$ and some $A \geq 1$.  Then there exists $\xi \in \Z/p\Z$ such that
\begin{equation}\label{bomb}
\left\| \phi(n_0+h) - \phi(n_0) - \frac{\xi h}{p} \right\|_{\R/\Z} \ll A^{1/2} |S|^4 \frac{\| h \|_{{S^\perp}}}{\rho}
\end{equation}
for all $h \in B(S,\rho)$.
\end{corollary}

\begin{proof}  By translating in space, we may normalise so that $n_0=0$; by shifting $\phi$ by a phase, we may also suppose that $\phi(0)=0$.  
By replacing $\rho$ with the smaller quantity $\rho/A^{1/2}$ if necessary, we may normalise $A$ to be $1$ (note that \eqref{bomb} is trivial for $\|h\|_{S^\perp} \geq \rho/A^{1/2}$).  Thus, we now have a function $\phi: B(S,\rho)  \to \R/\Z$ with $\phi(0)=0$ such that the quantity
\begin{equation}\label{phph}
 \partial^2 \phi(h,k) \coloneqq \phi(h+k) - \phi(h) - \phi(k)
\end{equation}
obeys the bound
\begin{equation}\label{ppp-2}
 \| \partial^2 \phi(h,k)  \|_{\R/\Z} \leq \frac{ \| h \|_{S^\perp} \| k \|_{S^\perp} }{\rho^2}
\end{equation}
for all $h,k \in B(S,\rho/2)$, and our task is to locate $\xi \in \Z/p\Z$ such that
\begin{equation}\label{bomb-2}
\left\| \phi(h) - \frac{\xi h}{p} \right\|_{\R/\Z} \ll |S|^4 \frac{\| h \|_{S^\perp}}{\rho}
\end{equation}
for all $h \in B(S,\rho)$.

Let $\rho_0 \coloneqq  \rho/100$, and set $\rho_1 \coloneqq \frac{\rho}{C |S|^3}$ for some sufficiently large absolute constant $C$.  If we let $f: \Z/p\Z \to \C$ be the $1$-bounded function
\begin{equation}\label{fdef}
f(x) \coloneqq  1_{B(S,\rho)} e( \phi(x) )
\end{equation}
and draw $\h_0, \h'_0, \h_1, \h'_1$ independently and regularly from $B(S,\rho_0)$, $B(S,\rho_0)$, $B(S,\rho_1)$, $B(S,\rho_1)$ respectively, then from \eqref{phph} we have
\begin{align*}
& f(\h_0+\h_1) \overline{f}(\h_0+\h'_1) \overline{f}(\h'_0+\h_1) f(\h'_0+\h'_1) \\
&\quad = e\left( \partial^2 \phi( \h_0, \h_1 ) - \partial^2 \phi( \h'_0, \h_1 ) - \partial^2 \phi( \h_0, \h'_1 ) + \partial^2 \phi( \h'_0, \h'_1 ) \right).
\end{align*}
Applying \eqref{ppp-2} and taking expectations, we conclude that
$$ |\E f(\h_0+\h_1) \overline{f}(\h_0+\h'_1) \overline{f}(\h'_0+\h_1) f(\h'_0+\h'_1)| \geq 1/2$$
(say).  Applying Theorem \ref{locu2} (which is applicable for $C$ large enough), we may thus find $\xi \in \Z/p\Z$ such that 
$$ \sum_{n_0 \in \Z/p\Z} \P(\n_0 = n_0) \left| \E f(n_0+\n_1) e_p( -\xi \n_1) \right|^2 \geq 1/4$$
if $\n_0, \n_1$ are drawn independently and regularly from $B(S,\rho_0), B(S,\rho_1)$ respectively.  In particular, there exists $n \in B(S,\rho_0)$ such that
$$ | \E f(n+\n_1) e_p( -\xi \n_1 ) | \geq 1/4.$$
By \eqref{fdef}, \eqref{phph} we have
$$ f(n+\n_1) = e\left( \phi( \n_1 ) + \phi( n ) + \partial^2 \phi(n, \n_1) \right)$$
so by \eqref{ppp-2} we conclude that
\begin{equation}\label{qqq}
 \big| \E e\big( \phi(\n_1) - \frac{\xi \n_1}{p}\big) \big| \gg 1.
\end{equation}

For any $h \in B(S,\rho_1)$, we have from Lemma \ref{ati} that
$$\big| \E e\big( \phi(\n_1+h) - \frac{\xi (\n_1+h)}{p} \big)  - \E e\big( \phi(\n_1) - \frac{\xi \n_1}{p} \big) | \ll |S| \frac{ \|h\|_{S^\perp} }{\rho_1};$$
on the other hand, from \eqref{phph} we have the identity
\begin{align*}
 \E e\big( \phi(\n_1+h) & - \frac{\xi (\n_1+h)}{p} \big) \\ & = e\big( \phi(h) - \frac{\xi h}{p} \big)\E e\big( \phi(\n_1) - \frac{\xi \n_1}{p} + \partial^2 \phi(\n_1, h)\big).\end{align*}
Combining this with \eqref{ppp-2}, \eqref{qqq}, and \eqref{bilipschitz}, we conclude that
$$ \left\| \phi(h) - \frac{\xi h}{p} \right\|_{\R/\Z} \asymp |e(\phi(h)-\frac{\xi h}{p}) - 1| \ll |S| \frac{ \|h\|_{S^\perp} }{\rho_1}$$
for all $h \in B(S,\rho_1)$.  As the claim \eqref{bomb-2} is trivial for $h \in B(S,\rho) \backslash B(S,\rho_1)$, the claim follows.
\end{proof}

\section{Dilated tori}

As mentioned in Example 3 of Section \ref{overview-sec}, in order to maintain good quantitative control (and specifically, Lipschitz norm control) on the functions $F: G \to [-1,1]$ used to build quadratic approximants, one needs to generalise the underlying domain $G$ to more general tori than the standard tori $(\R/\Z)^d$ with the usual norm structure.  It turns out that it will suffice to work with \emph{dilated tori} of the form
$$ G = \prod_{i=1}^d (\R / \lambda_i \Z)$$
where $\lambda_1,\dots,\lambda_d \geq 1$ are real numbers.  One can view this dilated torus as the quotient of $\R^d$ by a dilated lattice $\Gamma \coloneqq \prod_{i=1}^d \lambda_i \Z$.  We can place a ``norm'' on $G$ by declaring $\|x\|_G$ for $x \in G$ to be the Euclidean distance in $\R^d$ from $x$ to $\Gamma$; this generalises the norm $\| \|_{\R/\Z}$ from Section \ref{notation-sec}.  This in turn defines a metric $d_G$ on $G$ by the formula
$$ d_G(x, y) \coloneqq \|x - y \|_G.$$

The \emph{volume} $\vol(G)$ of a dilated torus is defined to be the product
$$ \vol(G) \coloneqq \prod_{i=1}^d \lambda_i = \det(\Gamma).$$
It will be important to keep this quantity under control during the iteration process.  In particular, when transforming from one dilated torus to another, the volume of the new torus should behave like a linear function of the existing torus; anything worse than this (e.g. quadratic behaviour) will lead to undesirable bounds upon iteration.

We define the \emph{Pontryagin dual} $\hat G$ of a dilated torus $G$ to be the lattice
$$ \hat G \coloneqq \prod_{i=1}^d \frac{1}{\lambda_i} \Z.$$
Elements $k$ of this dual will be called \emph{dual frequencies} of the torus.  If $k = (k_1,\dots,k_d)$ is a dual frequency and $x = (x_1,\dots,x_d)$ is an element of $G$, we define the dot product $k \cdot x \in \R/\Z$ in the usual fashion as
$$ k \cdot x = k_1 x_1 + \dots + k_d x_d$$
noting that this gives a well-defined element of $\R/\Z$.

A dual frequency $k$ is said to be \emph{irreducible} if it is non-zero, and not of the form $k = nk'$ for some other dual frequency $k'$ and some natural number $n>1$.  If a dual frequency $k$ is irreducible, then its orthogonal complement
$$ k^\perp \coloneqq \{ x \in G: k \cdot x = 0 \}$$
is a $(d-1)$-dimensional subtorus of $G$; it inherits a metric $d_{k^\perp}$ from the torus $G$ it lies in.  We will need to pass to such a complement when dealing with poorly distributed quadratic factors (as in the third or fourth examples in Section \ref{overview-sec}), however we encounter the technical issue that these complements $k^\perp$ will not quite be of the form of a dilated torus.  However, we will be able to transform $k^\perp$ into a dilated torus using a bilipschitz transformation, as the following result shows.

\begin{theorem}\label{nfoc}  Let $G = \prod_{i=1}^d (\R/\lambda_i \Z)$ be a dilated torus, and let $k \in \hat G$ be an irreducible dual frequency of $G$.  Then there exists a dilated torus $G' = \prod_{i=1}^{d-1} (\R/\lambda'_i \Z)$ and a Lie group isomorphism $\psi: k^\perp \to G'$ obeying the bilipschitz bounds
\begin{equation}\label{bilip}
 \| \psi \|_\Lip, \| \psi^{-1} \|_\Lip \ll d^{O(d)} 
\end{equation}
and such that one has the volume bound
\begin{equation}\label{volb}
 \vol(G') = d^{O(d)} |k| \vol(G)
\end{equation}
where $|k|$ denotes the Euclidean magnitude of $k$ in $\R^d$.
\end{theorem}

\begin{proof}  The case $d=0$ is vacuous and the case $d=1$ is trivial, so we may assume $d > 1$. One can identify $k^\perp$ with the quotient $V / \Gamma$, where $V \coloneqq \{ x \in \R^d: k \cdot x = 0 \}$
is the hyperplane in $\R^d$ orthogonal to $k$ (now viewed as an element of $\R^d$), and 
$\Gamma \coloneqq V \cap \prod_{i=1}^d (\lambda_i \Z)$
is the restriction of the lattice $\prod_{i=1}^d (\lambda_i \Z)$ to $V$.

As $k$ is irreducible, there exists a vector $e$ in the lattice $\prod_{i=1}^d (\lambda_i \Z)$ with $k \cdot e = 1$; thus $e$ has distance $1/|k|$ to $V$.  One can form a fundamental domain of $\R^d / \prod_{i=1}^d (\lambda_i \Z)$ by taking any fundamental domain for $V/\Gamma$ and performing the Minkowski sum of that domain with the interval $\{ te: 0 \leq t \leq 1\}$.  By Fubini's theorem, the $d$-dimensional Lebesgue measure of such a sum will equal the $(d-1)$-dimensional Lebesgue measure of the fundamental domain of $V/\Gamma$ and $1/|k|$; thus the covolume of $\prod_{i=1}^d (\lambda_i \Z)$ in $\R^d$ equals $1/|k|$ times the covolume of $\Gamma$ in $V$.  As the former covolume (determinant) is $\prod_{i=1}^d \lambda_i = \vol(G)$, we conclude that $\Gamma$ has covolume $|k| \vol(G)$ in $V$.

Applying Lemma \ref{bohr-basis}, we can find linearly independent elements $v_1,\dots,$ $v_{d-1}$ generating $\Gamma$ such that
\begin{equation}\label{bvt-1}
 B_V(0, O(d)^{-3d/2} t ) \cap \Gamma \subset \{ \sum_{i=1}^r n_i v_i: |n_i| \leq tN_i \} \subset B_V(0,t) \cap \Gamma
\end{equation}
for all $t>0$, where $B_V(0,r)$ is the Euclidean ball of radius $r$ in $V$, and the $n_i$ are understood to be integers, with the bound
\begin{equation}\label{sting}
 \prod_{i=1}^{d-1} N_i^{-1} = (2d)^{O(d)} |k| \vol(G).
\end{equation}
From \eqref{bvt-1} we conclude in particular that
\begin{equation}\label{dmark}
 O(d)^{-3d/2} N_i^{-1} \leq |v_i| \leq N_i^{-1}
\end{equation}
for all $1 \leq i \leq d$.

We now define the $(d-1)$-dimensional dilated torus
$$ G' \coloneqq \prod_{i=1}^{d-1} (\R/N_i^{-1} \Z)$$
and the isomorphism $\phi: V/\Gamma \to G'$ by the formula
$$ \phi( \sum_{i=1}^{d-1} t_i v_i \md{\Gamma} ) \coloneqq (t_1 N_1^{-1},\dots,t_{d-1} N_{d-1}^{-1}) \md{ \prod_{i=1}^{d-1} N_i^{-1} \Z }$$
for real numbers $t_1,\dots,t_{d-1}$.  It is easy to see that this is a Lie group isomorphism, and the bound \eqref{volb} follows from \eqref{sting}.  It remains to establish the bilipschitz bounds \eqref{bilip}.  It suffices to show that the linear isomorphism
$$ \sum_{i=1}^{d-1} t_i v_i  \mapsto (t_1 N_1^{-1},\dots,t_{d-1} N_{d-1}^{-1}) $$
from $V$ to $\R^{d-1}$, together with its inverse, have an operator norm of $O( d^{O(d)} )$.  For the inverse map, this is clear from \eqref{dmark}.  For the forward map, it suffices from Cram\'er's rule to show that
$$ \frac{|v_1 \wedge \dots \wedge v_{i-1} \wedge x \wedge v_{i+1} \wedge \dots \wedge v_{d-1}|}{|v_1 \wedge \dots \wedge v_{d-1}|} \ll \frac{d^{O(d)}}{\lambda'_i} $$
for all $i=1,\dots,d-1$ and all unit vectors $x$ in $V$.  But from \eqref{dmark} the numerator is at most $\prod_{1 \leq i' \leq d-1: i' \neq i} N_{i'}^{-1}$, while the denominator is the volume of a fundamental domain in $V$ and is thus equal to $d^{O(d)} N_1^{-1} \dots N_{d-1}^{-1}$ thanks to \eqref{sting}.  The claim follows.
\end{proof}

\section{Constructing the approximants}

In this section we construct the abstract directed graph $G = (V,E)$ that appears in Proposition \ref{main-abstract}.  For the rest of the paper, the prime $p$, the function $f: \Z/p\Z \to [-1,1]$, and the parameter $\eta$ with $0 < \eta \leq \frac{1}{10}$ are fixed, and we assume that \eqref{p-large} holds.

We begin with a description of the structured approximants $v \in V$.

\begin{definition}[Structured local approximant]\label{sla-def}  A \emph{structured local approximant} is a tuple
$$ v = \left( C, \c, (n_c + B(S_c,\rho_c))_{c \in C}, (G_c)_{c \in C}, (F_c)_{c \in C}, (\Xi_c)_{c \in C} \right)$$
consisting of the following objects:
\begin{itemize}
\item A finite non-empty set $C$;
\item A random variable $\c$, which we call the \emph{label variable}, taking values in $C$;
\item A shifted Bohr set $n_c + B(S_c,\rho_c)$ associated to each label $c \in C$;
\item A dilated torus $G_c$ associated to each label $c \in C$; 
\item A $1$-Lipschitz function $F_c: G_c \to [-1,1]$ associated to each label $c \in \C$; and
\item A locally quadratic function $\Xi_c: n_c + B(S_c,\rho_c) \to G_c$ associated to each label $c \in C$.
\end{itemize}
We denote the collection of all structured local approximants \textup{(}up to isomorphism\footnote{This caveat is needed for the technical reason that $V$ should be a set and not a proper class.}\textup{)} as $V$.  Given any structured local approximant $v \in V$, we define the random variables $(\a_v,\r_v,\f_v)$ associated to $v$ by the following construction.
\begin{itemize}
\item[1.]  First, let $\c$ be the random label variable appearing above.
\item[2.]  For each $c \in C$ in the essential range of $\c$, if we condition on the event $\c = c$, we draw $\a_v, \r_v$ independently and regularly from $n_c + B(S_c, \rho_c / 2)$ and $B( S_c, \exp(-\eta^{-C_4}) \rho_c)$ respectively, and then we let $\f_v$ be the function
$$ \f_v(a) \coloneqq F_c( \Xi_c(a) ).$$
\end{itemize}
Thus $\f_v$ is deterministic when $\c$ is conditioned to be fixed, but random when $\c$ is allowed to vary.
\end{definition}

We also define the following additional statistics of the structured local approximant $v$:
\begin{itemize}
\item The \emph{waste} $\waste(v)$ is the quantity $|\E f(\a) - {\mathbb E}_{a \in \Z/p\Z} f(a)|$;
\item The \emph{$1$-error} $\Err_1(v)$ is $|\E \f(\a) - \E f(\a)|$;
\item The \emph{$4$-error} $\Err_4(v)$ is $|\Lambda_{\a,\r}(\f) - \Lambda_{\a,\r}(f)|$;
\item The \emph{energy} $\Energy(v)$ is $\E |f(\a) - \f(\a)|^2$;
\item The \emph{linear rank} $d_1(v)$ is $\max_{c \in C} |S_c|$;
\item The \emph{quadratic dimension} $d_2(v)$ is $\max_{c \in C} \dim(G_c)$;
\item The \emph{linear scale} $\rho(v)$ is $\min_{c \in C} \rho_c$;
\item The \emph{quadratic volume} $\vol(v)$ is the quantity $\max_{c \in C} \vol(G_c)$;
\item The \emph{poorly distributed quadratic dimension} $d_2^{\mathrm{poor}}(v)$ is the maximum value of $\dim(G_c)$ over all poorly distributed $c$ in the essential range of $\c$, or zero if no such $c$ exists. Here, an element $c$ in the essential range of $\c$ is said to be \emph{poorly distributed} if one has
\begin{equation}\label{coo}
\Lambda_{\a,\r}(f | \c=c) < \E(\f(\a)|\c=c)^4 - \frac{\eta}{2}.
\end{equation}

\end{itemize}

This gives the set $V$ of structured local approximants for Proposition \ref{main-abstract}; we clearly have $0 \leq d_2^{\mathrm{poor}}(v) \leq d_2(v)$ for all $v \in V$.

We now also define the initial approximant.

\begin{definition} The initial approximant $v_0 \in V$ is defined to be the tuple
$$ v_0 = \left( C, \c, (n_c + B(S_c,\rho_c))_{c \in C}, (G_c)_{c \in C}, (F_c)_{c \in C}, (\Xi_c)_{c \in C} \right)$$
defined as follows:
\begin{itemize}
\item $C \coloneqq \Z/p\Z$, and $\c$ is drawn uniformly from $C$.
\item For each $c \in C$, we have $n_c\coloneqq 0$, $S_c \coloneqq \{1\}$, and $\rho_c \coloneqq 1$.
\item For each $c \in C$, the group $G_c$ is the standard $0$-torus $(\R/\Z)^0$ \textup{(}that is to say, a point\textup{)}.
\item For each $c \in C$, the function $F_c: G_c \to [-1,1]$ is the zero function $F_c(x) \coloneqq 0$.
\item For each $c \in C$, the function $\Xi_c: \Z/p\Z \to G_c$ is the unique \textup{(}constant\textup{)} map from $\Z/p\Z$ to the point $G_c$.
\end{itemize}
\end{definition}

By chasing the definitions, we see that $\a_{v_0}$ is uniformly distributed in $\Z/p\Z$, and we can compute several of the statistics of the initial approximant $v_0$:
\begin{equation}\label{d-stat}
\waste(v_0) = d_2^{\mathrm{poor}}(v_0) = d_2(v_0) = 0; d_1(v_0) = \rho(v) = \vol(v) = 1.
\end{equation}

Now we define the edges of the graph $G(V,E)$.

\begin{definition} \label{edge-def} 
We let $E$ be the set of all directed edges $v \to v'$, where $v, v' \in V$ are structured local approximants such that
\begin{align*}
d_1(v') &\leq d_1(v) + \eta^{-C_2} \\
d_2(v') &\leq d_2(v) + 1 \\
\rho(v') &\geq \exp(-\eta^{-C_5}) \rho(v) \\
\vol(v') &\leq \exp( \eta^{-C_3} ) \vol(v)\\
|\waste(v) - \waste(v')| &\leq \eta^{C_3}.
\end{align*}
\end{definition}

From this definition and \eqref{d-stat} we have the following bounds on the various statistics of vertices of $V$ that are not too far from the initial vertex $v_0$, assuming that each constant $C_i$ is chosen sufficiently large depending on the preceding constants $C_1,\dots,C_{i-1}$.

\begin{lemma}\label{bounds}
Suppose a vertex $v = v_k \in V$ can be reached from $v_0$ by a path $v_0 \to v_1 \to \dots \to v_k$ with $0 \leq k \leq 8\eta^{-2C_2}$.  Then we have
\begin{align}
d_1(v) &\leq 8\eta^{-3C_2} \label{bound-1}\\
d_2(v) &\leq 8\eta^{-2C_2} \label{bound-2}\\
\rho(v) &\geq \exp( -\eta^{-2C_5} ) \label{bound-3}\\
\vol(v) &\leq \exp( \eta^{-2C_3} ) \label{bound-4}\\
\waste(v) &\leq \eta^{C_3/2}\label{bound-5}
\end{align}
\end{lemma}

From \eqref{bound-5} we see in particular that the almost uniformity axiom in Proposition \ref{main-abstract}(ii) is obeyed.  The thickness axiom in Proposition \ref{main-abstract}(i) is also easy, as the following corollary shows.

\begin{corollary}
Suppose a quadratic approximant $v = v_k \in V$ can be reached from $v_0$ by a path $v_0 \to v_1 \to \dots \to v_k$ of length $k$ at most $8\eta^{-2C_2}$.  Then we have $\P( \r_{v} = 0 ) \ll \exp( \eta^{-C_5^2} ) / p$.
\end{corollary}

\begin{proof}  
Write
$$ v = \left( C, \c, (n_c + B(S_c,\rho_c))_{c \in C}, (G_c)_{c \in C}, (F_c)_{c \in C}, (\Xi_c)_{c \in C} \right).$$
It suffices to show that
$$ \P( \r_{v} = 0 | \c = c ) \ll \exp( \eta^{-C_5^2} ) / p$$
for each $c$ in the essential range of $\c$.  But once $\c$ is fixed to equal $\c$, then $\r_v$ is drawn regularly from
$n_c + B( S_c, \exp(-\eta^{-C_4}) \rho_c)$.  By Lemma \ref{bounds}, $S_c$ has cardinality at most $8\eta^{-3C_2}$ and $\rho_c$ is at least $\exp(-\eta^{-2C_5})$.  The claim now follows from Lemma \ref{size-bohr}.
\end{proof}

It remains to verify the last two axioms (iii), (iv) of Proposition \ref{main-abstract}.  We isolate these statements formally, using Lemma \ref{bounds} and Definition \ref{edge-def}. 

The first of these results, Theorem \ref{bad-ed}, states that ``a bad approximation implies an energy decrement''. The second, Theorem \ref{bad-dim}, states that ``a bad lower bound implies a dimension increment''.

\begin{theorem}\label{bad-ed} Let the notation and hypotheses be as above.  Suppose that
$v \in V$ is a structured local approximant obeying \eqref{bound-1}-\eqref{bound-4}.  If we have
\begin{equation}\label{tp-bad-again-0}
\Err_1(v) > \eta
\end{equation}
or
\begin{equation}\label{tp-bad-again}
\Err_4(v) > \eta
\end{equation}
then there exists a structured local approximant $v'$ obeying the bounds
\begin{align}
d(v') &\leq d(v) + \eta^{-C_2}\label{be-1} \\
d_2(v') &\leq d_2(v)+1  \label{be-2}\\
\rho(v') &\geq \exp(-\eta^{-C_5}) \rho(v) \label{be-3} \\
\vol(v') &\leq \exp( \eta^{-C_3} ) \vol(v) \label{be-4} \\
|\waste(v') - \waste(v)| &\leq \eta^{C_3} \label{be-5} \\
\Energy(v') &\leq \Energy(v) - \eta^{C_2}.\label{be-6}
\end{align}
\end{theorem}

\begin{theorem}\label{bad-dim} Let the notation and hypotheses be as above.  Suppose that
$v \in V$ is a structured local approximant obeying \eqref{bound-1}-\eqref{bound-4}, and let $\a_v,\r_v,\f_v$ be the random variables associated to $v$. If we have
\begin{equation}\label{lower-again}
\Lambda_{\a_v, \r_v}(\f_v) \leq (\E \f_v(\a_v))^4 - \eta,
\end{equation}
then there exists a quadratic approximant $v' \in V$ with 
\begin{align}
d(v') &\leq d(v) + \eta^{-C_2}\label{bd-1} \\
d_2(v') &\leq d_2(v)  \label{bd-2}\\
d_2^{\mathrm{poor}}(v') &\leq d_2^{\mathrm{poor}}(v)-1 \label{bd-3}  \\
\rho(v') &\geq \exp(-\eta^{-C_5}) \rho(v) \label{bd-4} \\
\vol(v') &\leq \exp( \eta^{-C_3} ) \vol(v) \label{bd-5} \\
|\waste(v') - \waste(v)| &\leq \eta^{C_3} \label{bd-6} \\
\Energy(v') &\leq \Energy(v) + \eta^{3C_2}.\label{bd-7}
\end{align}
\end{theorem}

It remains to prove Theorem \ref{bad-ed} and Theorem \ref{bad-dim}.  Theorem \ref{bad-ed} will be proven in Section \ref{baded-sec} using a difficult local inverse Gowers theorem, Theorem \ref{locu3}, that will be proven in later sections.  Theorem \ref{bad-dim}, on the other hand, will not rely on the local inverse Gowers theorem; it is proven in Section \ref{baddim-sec}.

\section{Bad lower bound implies dimension decrement}\label{baddim-sec}

In this section we prove Theorem \ref{bad-dim}.  Let the notation and hypotheses be as in Theorem \ref{bad-dim}.  We abbreviate $\a_v,\r_v,\f_v$ as $\a,\r,\f$ respectively.  We can write the left-hand side of \eqref{lower-again} as $\E A(\c)$, where for any $c \in C$, the quantity $A(c)$ is defined as the conditional expectation
$$ A(c) \coloneqq \Lambda_{\a,\r}(\f | \c = c ).$$
Similarly, we can write $\E \f(\a) = \E B(\c)$, where $B(\c) \coloneqq \E(\f(\a)|\c = c)$.  By \eqref{lower-again} and H\"older's inequality, we thus have
$$ \E B(\c)^4 - A(\c) \geq \eta.$$
Applying Lemma \ref{popular}, we must therefore have
$$ \P( B(\c)^4 - A(\c) > \eta/2 ) \gg \eta.$$
By \eqref{coo}, we conclude that $\c$ is poorly distributed with probability $\gg \eta$.  In particular, there is at least one poorly distributed value of $c$.

Most of this section will be devoted to the proof of the following proposition, which roughly speaking asserts that when $\c$ is poorly distributed, there is a linear constraint between the quadratic frequencies which will ultimately allow us to decrease the poorly distributed quadratic dimension $d_2^{\mathrm{poor}}$.

\begin{proposition}\label{cprop}  Let $c$ be a poorly distributed element of the essential range of $\c$.  Then there exists a natural number $m_c$, a frequency $\xi_c \in \Z/p\Z$ and an irreducible dual frequency $k'_c \in \hat G_c$ with
\begin{equation}\label{loo}
1 \leq m_c \ll \exp( \eta^{-4C_3} )
\end{equation}
and
\begin{equation}\label{loo-2}
\exp( -\eta^{-4C_3} ) \ll |k'_c| \ll \exp( \eta^{-3C_2} )
\end{equation}
such that
\begin{equation}\label{kpc-lin}
\| k'_c \cdot \Xi_c(a + 2m_c h) - k'_c \cdot \Xi_c(a) \|_{\R/\Z} \ll \exp( -\eta^{-3C_4} )
\end{equation}
for all $a \in B(S_c, \rho_c/2)$ and $h \in B(S_c \cup \{ \xi_c \}, \exp(-\eta^{-5C_4}) \rho )$.
\end{proposition}

A key technical point here is that the upper bound on $|k'_c|$ involves only $C_2$ and not $C_3$ or $C_4$; this is necessary in order to keep the bounds under control during the iteration process.  However, we will be able to tolerate the presence of the $C_3$ and $C_4$ constants in the other components of Proposition \ref{cprop}.

\begin{proof}  We condition on the event $\c = c$.  By Definition \ref{sla-def}, the random variables $\a, \r$ are now independent and regularly drawn from $n_c + B(S_c, \rho_c / 2)$ and $B( S_c, \exp(-\eta^{-C_4}) \rho_c)$ respectively, while $\f(n) = F_c( \Xi_c(a))$.  We conclude that
\begin{align*} \E( F_c(\Xi_c(\a)) F_c(\Xi_c(\a+\r)) F_c(\Xi_c(\a+2\r)) & F_c(\Xi_c(\a+3\r)) | \c = c ) \\ & < \E(F_c(\Xi_c(\a))|\c=c)^4 - \eta/2.
\end{align*}
Since $\Xi_c: \Z/p\Z \to G_c$ is locally quadratic on $n_c + B(S_c, \rho_c)$, which contains the progression $\a,\a+\r,\a+2\r, \a+3\r$, we see from \eqref{omh-ap} that
$$ \Xi_c(\a) - 3 \Xi_c(\a+\r) + 3 \Xi_c(\a+2\r) - \Xi_c(\a+3\r) = 0$$
and so the left-hand side can be written as
$$ \E( F^{(3)}_c( \Xi_c(\a), \Xi_c(\a+\r), \Xi_c(\a+2\r)) | \c = c ),$$
where $F^{(3)}_c: G_c^3 \to [-1,1]$ is the function
$$ F^{(3)}_c(x_0,x_1,x_2) \coloneqq F_c(x_0) F_c(x_1) F_c( x_2) F_c( x_0 - 3x_1 + 3x_2 ).$$
Applying Lemma \ref{cauchy}, we have 
$$ \int_{G_c^3} F^{(3)}_c(x_0,x_1,x_2)\ d\mu_c(x_0) d\mu_c(x_1) d\mu_c(x_2) \geq \left(\int_{G_c} F_c(x)\ d\mu_c(x)\right)^4$$
where $\mu_c$ is the probability Haar measure on $G_c$.  By the triangle inequality, we conclude that at least one of the assertions
\begin{align*} \bigg|\E( F^{(3)}_c( \Xi_c(\a), & \Xi_c(\a+\r), \Xi_c(\a+2\r)) | \c = c) \\ &  - \int_{G_c^3} F^{(3)}_c(x_0,x_1,x_2)\ d\mu_c(x_0) d\mu_c(x_1) d\mu_c(x_2) \bigg| \gg \eta\end{align*}
or
$$ \left|\E( F_c( \Xi_c(\a) ) | \c = c ) - \int_{G_c} F_c(x)\ d\mu_c(x) \right| \gg \eta$$
holds.  Defining $\tilde F: G_c^3 \to [-1,1]$ by $\tilde F(x_0, x_1, x_2) = $
$$ \frac{1}{10} \left( F^{(3)}_c(x_0,x_1,x_2) - \int_{G_c^3} F^{(3)}_c(x_0,x_1,x_2)\ d\mu_c(x_0) d\mu_c(x_1) d\mu_c(x_2)  \right)$$
in the former case and
$$\tilde F(x_0,x_1,x_2) \coloneqq \frac{1}{10} \left( F_c(x_0) - \int_{G_c} F_c(x_0)\ d\mu_c(x_0) \right)$$
in the latter case, we see that $\tilde F$ is $1$-Lipschitz and of mean zero, and
\begin{equation}\label{ill}
 |\E(\tilde F( \x_c ) | \c = c)| \gg \eta
\end{equation}
where $\x_c \in G_c^3$ is the random variable
$$ \x_c \coloneqq (\Xi_c(\a), \Xi_c(\a+\r), \Xi_c(\a+2\r)).$$
The Weyl equidistribution criterion, applied in the contrapositive, then suggests that there should be a non-zero dual frequency $k = (k_1,k_2,k_3) \in \hat G_c^3$ to $G_c^3$
such that $\E( e(k \cdot \x_c) | \c = c)$ is large.  The next lemma makes this intuition precise.

\begin{lemma}[Weyl equidistribution]\label{weyl}  With the notation and hypotheses as above, there exists a non-zero dual frequency $k = (k_1,k_2,k_3) \in \hat G_c^3$ to $G_c^3$ with $|k| \ll \exp( O( \eta^{-3C_2} ) )$ such that
$$ |\E( k \cdot \x_c | \c = c)| \gg \exp( - O( \eta^{-3C_2} ) ) / \vol(G_c).$$
\end{lemma}

A key point here is that the bound on $|k|$ does not depend on the volume of the dilated torus $G_c$, which will typically be much larger than $\eta^{-2C_2-10}$.

\begin{proof}  Write $G_c = \prod_{i=1}^d (\R/\lambda_i\Z)$, thus $\lambda_1,\dots,\lambda_d \geq 1$, and by \eqref{bound-2} one has 
\begin{equation}\label{d-size}
d \leq 8 \eta^{-2C_2}.
\end{equation}
The bound \eqref{ill} is not possible when $d = 0$, so we may assume $d \geq 1$.  We can write $G_c^3 = \prod_{i=1}^{3d} (\R/\lambda_i\Z)$, where we extend $\lambda_i$ periodically with period $d$.

Let $\varphi: \R \to \R$ be a fixed smooth even function supported on $[-1,1]$ that equals $1$ at the origin and whose Fourier transform $\hat \varphi(\xi) \coloneqq  \int_\R \phi(x) e(-x\xi)\ dx$ is non-negative; such a function may be easily constructed by convolving an $L^2$-normalised smooth function on $[0,1]$ with its reflection.   Let $A \geq 1$ be a parameter to be chosen later, and introduce the kernel $K: G_c^3 \to \R^+$ by the formula
$$ 
K(t_1, \dots, t_{3d}) \coloneqq  \prod_{i=1}^{3d} K_i(t_i)
$$
for $t_i \in \R/\lambda_i \Z$, where
$$ K_i(t_i) \coloneqq  \sum_{k_i \in \frac{1}{\lambda_i} \Z} \varphi\left(\frac{k_i}{A}\right) e(k_i t_i).$$
By Poisson summation, the $K_i$ and hence $K$ are non-negative.
A Fourier-analytic calculation using the smoothness of $\varphi$ gives
$$ \int_{\R/\lambda_i \Z} K_i(t_i) \ \frac{dt_i}{\lambda_i} = 1$$
and
$$ \int_{\R/\lambda_i \Z} K_i(t_i)  \sin^2(\pi t_i / \lambda_i)\ \frac{dt_i}{\lambda_i} \ll \frac{1}{A^2 \lambda_i^2}$$
(where the implied constant is allowed to depend on $\varphi$) and hence by \eqref{bilipschitz} and Cauchy-Schwarz we have
$$ \int_{\R/\lambda_i \Z} K_i(t_i) \|t_i\|_{\R/\Z}\ \frac{dt_i}{\lambda_i} \ll \frac{1}{A},$$
which on taking tensor products gives
$$ \int_{G_c^3} K(x)\ d\mu_c^3(x) = 1$$
and
$$ \int_{G_c^3} K(x) \| x \|_{G_c^3}\ d\mu_c^3(x) \ll \frac{d}{A},$$
where $\mu^3_c$ is the Haar probability measure on $G_c^3$.  If we then take the convolution
$$ \tilde F*K(x) \coloneqq  \int_{G_c} \tilde F(x-y) K(y)\ d\mu_c^3(y)$$
then by the $1$-Lipschitz nature of $\tilde F$ we see that
$$ \tilde F*K(x) = \tilde F(x) + O\left( \frac{d}{A} \right).$$
Thus, if we choose
$$ A \coloneqq  \frac{C d}{\eta}$$
for a sufficiently large absolute constant $C$, we conclude from \eqref{ill} that
$$ | \E(\tilde F*K( \x_c ) | \c = c)| \gg \eta.$$
However, by Fourier expansion and the fact that $\tilde F$ has mean zero,
$$ \tilde F*K( \x_c )  = \sum_{k \in \hat G_c^3 \backslash \{0\}} \left(\prod_{i=1}^{3d} \varphi(\frac{k_i}{A})\right) \widehat{\tilde F(k)} \E e(k \cdot \x)$$
where $k = (k_1,\dots,k_{3d})$ with $k_i \in \frac{1}{\lambda_i} \Z$ for $i=1,\dots,3d$, and
$$ \widehat{\tilde F}(k) \coloneqq  \int_{G_c^3} \tilde F(x) e(-k \cdot x)\ d\mu_c^3(x).$$
Using the triangle inequality and crudely bounding $|\widehat{\tilde F}(k)|$ by $1$, we conclude that
$$
\sum_{k \in \hat G_c^3 \backslash \{0\}} \left(\prod_{i=1}^d \left|\varphi\left(\frac{k_i}{A}\right)\right|\right) |\E(e(k \cdot \x_c)|\c=c)| \gg \eta.
$$
The summand is only non-vanishing when $\sup_i |k_i| \leq A$, so that 
$$|k| \leq dA \ll \exp( O( \eta^{-3C_2} ) )$$
(thanks to \eqref{d-size} and the choice of $A$), and the number of such $k$ is 
$$O\left( \prod_{i=1}^{3d}(A \lambda_i) \right) \ll \exp( O( \eta^{-3C_2} ) ) \vol(T).$$
Since $\varphi$ is bounded, the claim now follows from the pigeonhole principle.
\end{proof}

We return to the proof of Proposition \ref{cprop}.
Applying Lemma \ref{weyl} and \eqref{bound-3}, we see that there exists a non-zero triplet $(k^0_c, k^1_c, k^2_c) \in \hat G_c^3$ with
\begin{equation}\label{kic}
 |k^0_c|, |k^1_c|, |k^2_c| \ll \exp( \eta^{-3C_2} ) 
\end{equation}
and
\begin{equation}\label{k012}
 \E \left( e( k^0_c \cdot \Xi_c(\a) + k^1_c \cdot \Xi_c( \a+\r ) + k^2_c \cdot \Xi_c( \a+2\r )) | \c=c \right) \gg \exp( -\eta^{-3C_3} ).
\end{equation}
Among other things, the non-zero nature of this triplet forces $G_c$ to be non-trivial, and thus
$$ d_2^{\mathrm{poor}}(v) \geq 1.$$
We also emphasise that the bound \eqref{kic} involves $C_2$ rather than $C_3$; this will become important when establishing the important upper bound of \eqref{loo-2} later in this proof.

We can use the exponential sum bound \eqref{k012} to control the ``second derivative'' of $\Xi_c$.  Indeed, for any $h_1,h_2 \in B(S_c, \rho_c/10)$, define the quantity $\partial^2 \Xi_c(h_1,h_2) \in \R/\Z$ by
$$ \partial^2 \Xi_c(h_1,h_2) \coloneqq \Xi_c(a+h_1+h_2) - \Xi_c(a+h_1) - \Xi_c(a+h_2) + \Xi_c(a)$$
for any $a \in n_c + B(S_c, \rho/2)$. Since $\Gamma_\c$ is locally quadratic on $n_c + B(S_c,\rho)$, this quantity is well-defined, symmetric in $h_1,h_2$, and is also locally bilinear in $h_1$ and $h_2$.

\begin{lemma} Let the notation and hypotheses be as above.  Then for any $i=0,1,2$, we have
$$ \left|\E\left( e( 2 k^i_c \cdot \partial^2 \Xi_c( \r - \r', \h - \h' ) ) | \c = c\right)\right| \gg \exp( -4\eta^{-3C_3} ),$$
where, conditioning on the event $\c = c$, the random variables $\r, \r', \h, \h'$ are drawn independently and regularly from the Bohr sets $B(S_c, \exp(-\eta^{-C_4}) \rho)$, $B(S_c, \exp(-\eta^{-C_4}) \rho)$, $B(S_c, \exp(-\eta^{-2C_4}) \rho)$, $B(S_c, \exp(-\eta^{-2C_4}) \rho)$ respectively, independently of $\a$.
\end{lemma}

\begin{proof} To simplify the notation we only consider the $i=2$ case, as the $i=0,1$ cases are similar.  This will be ``Weyl differencing'' argument that relies primarily on the Cauchy-Schwarz inequality.

Recall that after conditioning to the event $\c=c$, the random variable $\a$ is drawn regularly from $B(S_c, \rho/2)$.  Using Lemma \ref{ati}, we see that $\a$ and $\a-\h$ differ in total variation by $O( \exp(-\eta^{-C_4/2} ))$, hence from \eqref{k012} we have
\begin{align*} \bigg|\E\big( e( k^0_c \cdot \Xi_c(\a-\h) + k^1_c \cdot \Xi_c( \a-\h+\r ) + k^2_c \cdot & \Xi_c( \a-\h+2\r )) | \c=c \big)\bigg| \\ & \gg \exp( -\eta^{-3C_3} ).\end{align*}
Similarly we may use Lemma \ref{ati} to compare $\r$ and $\r+\h$, and conclude that
\begin{align*} \bigg|\E\big( e( k^0_c \cdot \Xi_c(\a-\h) + k^1_c \cdot \Xi_c( \a+\r ) + k^2_c \cdot & \Xi_c( \a+\h+2\r )) | \c=c \big)\bigg| \\ & \gg \exp( -\eta^{-3C_3} ),\end{align*}
By the pigeonhole principle (and independence of $\a,\h,\r$ relative to the event $\c=c$), we may thus find $a_c \in n_c + B(S_c, \rho/2)$ such that
\begin{align*} \bigg|\E\big( e( k^0_c \cdot \Xi_c(a_c-\h) + k^1_c \cdot \Xi_c( a_c+\r ) + k^2_c \cdot & \Xi_c( a_c+\h+2\r )) | \c=c \big)\bigg| \\ & \gg \exp( -\eta^{-3C_3} ).\end{align*}
Using the identity
$$ \Xi_c( a_c + \h + 2\r) = \Xi_c(a_c+ \h) + \Xi_c( a_c + 2\r ) - \Xi_c( a_c ) + \partial^2 \Xi_c(2\r, \h )$$
we can rewrite the left-hand side as
$$ \left|\E\left( b_1(\r) b_2(\h) e( k^2_c \cdot \partial^2 \Xi_c( 2\r, \h ) ) | \c = c \right)\right| \gg \exp( -\eta^{-3C_3} )$$
where $b_1, b_2: B(S_c,\rho) \to \C$ are the $1$-bounded functions
$$ b_1(r) \coloneqq e( k^1_c \cdot \Xi_\c( a_c+r ) + k^2_c \cdot \Xi_c( a_c+2r ) - k^2_c \cdot \Xi_c(a_c) )$$
and
$$ b_2(h) \coloneqq  e( k^0_c \cdot \Xi_c(a_c-h) + k^2_c \cdot \Xi_c( a_c+h ) ).$$
Applying Lemma \ref{cauchy-schwarz} to eliminate the $\b_1(\r)$ factor, we conclude that
$$ \left|\E\left(b_2(\h) \overline{b_2(\h')} e( k^2_c \cdot \partial^2 \Xi_c( 2\r, \h - \h') ) | \c = c \right)\right| \gg \exp( -2 \eta^{-3C_3} ).$$
Applying Lemma \ref{cauchy-schwarz} again to eliminate the $b_2(\h) \overline{b_2(\h')} $ factor, we obtain the claim.
\end{proof}

We return to the proof of Proposition \ref{cprop}.
Let $i = i_c \in \{0,1,2\}$ be such that $k^i_c$ is non-zero.  Let $\r,\r',\h,\h'$ be as in the above lemma, and let $\h''$ be a further independent copy of $\h$ or $\h'$, thus $\h''$ is also drawn regularly from $B(S_c, \exp(-\eta^{-2C_4}) \rho)$ and independently of $\r,\r',\h,\h'$ (after conditioning on $\c = c$).  Applying Lemma \ref{ati} to compare $\r$ with $\r + \h''$, we have
$$ |\E ( e( 2 k^i_c \cdot \partial^2 \Xi_c( \r - \r' + \h'', \h - \h' ) ) | \c = c)| \gg \exp( -4\eta^{-3C_3} ),$$
so by the pigeonhole principle we can find $r,r',h' \in B(S_c, \exp(-\eta^{-C_4}) \rho_c)$ (depending on $c$, of course) such that
$$ |\E ( e( 2 k^i_c \cdot \partial^2 \Xi_c( r - r' +  \h'', \h - h' ) ) | \c = c)| \gg \exp( -4\eta^{-3C_3} ).$$
By the local bilinearity of $\partial^2 \Xi_c$, we may thus have
$$ |\E ( e( 2 k^i_c \cdot \partial^2 \Xi_c( \h'', \h  ) + \psi(\h) + \psi''(\h'') ) | \c = c )| \gg \exp( -4\eta^{-3C_3} )$$
for some locally linear functions $\psi, \psi'': B(S_c, \rho/100) \to \R/\Z$ (which can depend on $c$).

Applying Proposition \ref{large-quadratic} (recalling from \eqref{bound-1} that $|S_c| \leq 8 \exp(-3C_2)$), we conclude that there exists a non-zero multiple $k_c \in \hat G_c$ of $k^i_c$ with 
\begin{equation}\label{kcc}
k_c \ll \exp( \eta^{-4C_3} )
\end{equation}
such that
\begin{equation}\label{milk}
\| k_c \cdot \partial^2 \Xi_c( n, m ) \|_{\R/\Z} \ll \exp( \eta^{-3C_4} ) \frac{\|n\|_{S_c} \|m\|_{S_c}}{\rho_c^2}
\end{equation}
for $n,m \in B(S_c, \exp(-\eta^{-3C_4}) \rho_c)$.  

Applying Corollary \ref{loc-to-glob}, we may thus find $\xi_c \in \Z/p\Z$ such that
\begin{equation}\label{kgam}
\left\| k_c \cdot \Xi_c(n_c + h) - k_c \cdot \Xi_c(n_c) - \frac{\xi_c h}{p} \right\|_{\R/\Z} \ll \exp( \eta^{-4C_4} ) \frac{\|h\|_{S_c}}{\rho_c}
\end{equation}
for all $n \in \Z/p\Z$ (of course, the bound is only non-trivial when $h$ lies in the Bohr set $B(S_c, \exp(-\eta^{-4C_4}) \rho )$).

The dual frequency $k_c \in \widehat{G_c}$ is non-zero, but not necessarily irreducible.  However, we may write $k_c = m_c k'_c$ where $m_c$ is a positive natural number and $k'_c \in \widehat{G_c}$ is irreducible, thus by \eqref{kcc} we have the bound \eqref{loo}.
The same argument gives the bound $k'_c \ll \exp(\eta^{-4C_3})$, but this is not sufficient to establish the upper bound in \eqref{loo-2}.  However, observe that $k^i_c$ must also be a multiple of the irreducible vector $k'_c$, and now the upper bound in \eqref{loo-2} follows from \eqref{kic}.

We can also obtain a lower bound on $k'_c$ by observing that the slab
$$\left\{ x \in G_c: \| k'_c \cdot x \|_{\R/\Z} \leq \frac{1}{2} |k'_c| \right\}$$
has measure at most $|k'_c| \vol(G_c)$, and contains the Euclidean ball of radius $1/2$ centred at the origin.  This gives the lower bound
$$ |k'_c| \gg \frac{1}{\dim(G_c)^{O(\dim(G_c))} \vol(G_c)}$$
which by \eqref{bound-2}, \eqref{bound-4} gives the lower bound in \eqref{loo-2}.

Now let $a \in B(S_c, \rho_c/2)$ and $h \in B(S_c \cup \{\xi_c\}, \exp(-\eta^{-5C_4}) \rho_c)$.  Then we have
$$j h \in B(S_c, 2m_c \exp(-\eta^{-5C_4}) \rho_c )$$
and
$$ \left\| \frac{j \xi_c h}{p} \right\|_{\R/\Z} \ll \exp(-\eta^{-5C_4}) \rho_c$$
for all $j$, $0 \leq j \leq 2m_c$.  From \eqref{kgam} and \eqref{loo}, we conclude that
$$
\| k_c \cdot \Xi_c(n_c + j h) - k_c \cdot \Xi_c(n_c) \|_{\R/\Z} \ll \exp( -\eta^{-4C_4} )$$
(say).  On the other hand, from \eqref{milk} we have
$$
\| k_c \cdot (\Xi_c(a + j h) - \Xi_c(a) - \Xi_c(n_c + j \xi_c h) + \Xi_c(n_c)) \|_{\R/\Z} \ll \exp( -\eta^{-4C_4} )$$
and hence by the triangle inequality we have
\begin{equation}\label{kc-lin}
\| k_c \cdot \Xi_c(a + j h) - k_c \cdot \Xi_c(a) \|_{\R/\Z} \ll \exp( -\eta^{-4C_4} )
\end{equation}
for all $j$, $0 \leq j \leq 2m_c$.  

This is close to \eqref{kpc-lin}, but we will need to replace the dual frequency $k_c$ here with the irreducible dual frequency $k'_c$.  To do this, we first observe that as $\Xi_c$ is locally quadratic on $n_c + B(S_c,\rho_c)$, we may write
\begin{equation}\label{xic}
 \Xi_c(a+j h) = \alpha + \beta j + \gamma j^2
\end{equation}
for all $j$, $0 \leq j \leq 2m_c$, and some $\alpha,\beta,\gamma \in G_c$ depending on $c, a, h$.  Inserting this formula into the preceding estimate, we conclude that
$$ \| j (k_c \cdot \beta) + j^2 (k_c \cdot \gamma) \|_{\R/\Z} \ll \exp(-\eta^{-4C_4})$$
for $j$, $0 \leq j \leq 2m_c$.  Applying this for $j=1,2$ and using the triangle inequality, we have
$$ \| k_c \cdot \beta\|, \| 2 (k_c \cdot \gamma) \|_{\R/\Z} \ll \exp(-\eta^{-4C_4})$$
Since $2m_c k'_c = 2k_c$ and $(2m_c)^2 k'_c = (2m_c) 2 k_c$, we conclude in particular (using \eqref{loo}) that
$$ \| 2m_c (k'_c \cdot \beta) \|_{\R/\Z}, \| (2m_c)^2 (k'_c \cdot \gamma) \|_{\R/\Z} \ll \exp(-\eta^{-3C_4})$$
and thus by \eqref{xic} we obtain \eqref{kpc-lin} as desired.  This finally completes the proof of Proposition \ref{cprop}.
\end{proof}

We now return to the proof of Theorem \ref{bad-dim}.  We are given a structured local approximant
$$ v = \left( C, \c, (n_c + B(S_c,\rho_c))_{c \in C}, (G_c)_{c \in C}, (F_c)_{c \in C}, (\Xi_c)_{c \in C} \right)$$
and need to construct a modification
$$ v' = \left( C', \c', (n'_{c'} + B(S'_{c'},\rho'_{c'}))_{c' \in C'}, (G'_{c'})_{c' \in C'}, (F'_{c'})_{c' \in C'}, (\Xi'_{c'})_{c' \in C'} \right)$$
that somehow incorporates the linear constraint identified in Proposition \ref{cprop} in order to decrement the poorly distributed quadratic dimension of $v'$, in the spirit of the third and fourth examples in Section \ref{overview-sec}.  To avoid confusion, we shall restore the subscripts $(\a_v, \r_v, \f_v)$ on the random variables associated to $v$ as per Definition \ref{sla-def}, in order to distinguish them from the corresponding random variables $(\a_{v'}, \r_{v'}, \f_{v'})$ that will be associated to $v'$.

We shall set $C' \coloneqq (\Z/p\Z) \times C$, and let $\c'$ be the random variable
$$ \c' \coloneqq (\a_v, \c).$$
Clearly $\c'$ takes values in the non-empty finite set $C'$.  Now we need to define $n'_{c'}, S'_{c'}, \rho'_{c'}, G'_{c'}, F'_{c'}, \Xi'_{c'}$ for any given $c' = (a,c)$ in $C'$.  In the case where $c$ is not poorly distributed, we simply carry over the corresponding data from $v$ without further modification.  That is to say, we define
$$ (n'_{c'}, S'_{c'}, \rho'_{c'}, G'_{c'}, F'_{c'}, \Xi'_{c'} \coloneqq (n_c, S_c, \rho_c, G_c, F_c, \Xi_c) $$
whenever $c' = (a,c)$ with $c$ not poorly distributed.  If instead $c' = (a,c)$ with $c$ poorly distributed, then we introduce the natural number $m_c$, the dual frequency $k'_c \in \hat G_c$, and the frequency $\xi_c \in \Z/p\Z$ from Proposition \ref{cprop}; of course we can arrange matters so that $m_c, k'_c, \xi_c$ depend only on $c$ and not on $a$.  Because of \eqref{loo} and the hypothesis \eqref{p-large}, the quantity $2m_c$ is invertible in the field $\Z/p\Z$, and so we may define the dilate $(2m_c)^{-1} \cdot S_c$ of $S_c$ inside $\Z/p\Z$, and can similarly define the dilate $(2m_c)^{-1} \xi_c$ of $\xi_c$.  We will need to do this division here in order to cancel some denominators appearing later in the argument.

In this poorly distributed case, we define the ``linear'' data $n'_{c'}, S'_{c'}, \rho'_{c'}$ by
\begin{align*}
n'_{c'} &\coloneqq a \\
S'_{c'} &\coloneqq (2m_c)^{-1} \cdot S_c \cup \{ (2m_c)^{-1} \xi_c \} \\
\rho'_{c'} &\coloneqq \exp(-\eta^{-6C_4} ) \rho_c ,
\end{align*}
thus the shifted Bohr set $n'_{c'} + B(S'_{c'}, \rho'_{c'})$ will be a small subset of $n_c + B(S_c,\rho_c)$ in which the radius $\rho_c$ has been reduced and an additional frequency $\xi_c/2m_c$ has been added.  As we shall see, this particular choice of this linear data will allow us to utilise the approximate constraint \eqref{kpc-lin}.

The constraint \eqref{kpc-lin} has the effect of approximately restricting $\Xi_c$ (on a suitable Bohr set) to a coset of the orthogonal complement $(k'_c)^\perp = \{ x \in G_c: k'_c \cdot x = 0 \}$ of $k'_c$ in $G_c$.  Applying Theorem \ref{nfoc}, \eqref{bound-2}, and the crucial bound \eqref{loo-2}, we may find a dilated torus
$\tilde G_c = \prod_{i=1}^{\dim(G_c)-1} (\R/\tilde \lambda_{c,i} \Z)$ with volume
\begin{equation}\label{voil}
 \vol(\tilde G_c) \ll \exp( \eta^{-4C_2} ) \vol(G_c)
\end{equation}
as well as a Lie group isomorphism $\psi_c: (k'_c)^\perp \to \tilde G_c$ obeying the bilipschitz bounds
$$ \| \psi \|_\Lip, \|\psi^{-1} \|_\Lip \leq \exp( \eta^{-4C_2} ).$$
In particular, if we define the even more dilated torus
$$ G'_c \coloneqq \prod_{i=1}^{\dim(G_c)-1} (\R/\exp(\eta^{-4C_2}) \tilde \lambda_{c,i} \Z) $$
and let $\delta_c: G'_c \to \tilde G_c$ be the rescaling map
$$ \delta_c: (x_i)_{i=1}^{\dim(G_c)-1} \mapsto (\exp(-\eta^{-4C_2}) x_i)_{i=1}^{\dim(G_c)-1} $$
then we see that $\psi^{-1} \circ \delta_c: G'_c \to (k'_c)^\perp$ is a $1$-Lipschitz Lie group isomorphism.

An element of $n'_{c'} + B(S'_c, \rho'_c)$ can be uniquely represented in the form $n'_{c'} + 2m_c h$ for $h \in B( S_c \cup \{\xi_c\}, \exp(-\eta^{-6C_4}) \rho_c)$.
From \eqref{kpc-lin}, we know that the point $\Xi_c(n'_{c'} + 2m_c h) - \Xi_c(n'_{c'})$ lies within a $O( \exp(-\eta^{-3C_4}) )$-neighbourhood of the subtorus $(k'_c)^\perp$.  Using the lower bound in \eqref{loo-2}, we can find a locally linear projection $\pi_c$ from this neighbourhood to the subtorus itself (e.g. by viewing the subtorus locally as a graph in $\dim(G_c)-1$ of the $\dim(G_c)$ coordinates and then projecting in the direction of the remaining coordinate), which moves each point in the neighbourhood by at most $O( \exp( -\eta^{-2C_4} ) )$.  From the $1$-Lipschitz nature of $F_c$, we thus have
\begin{align*} F_c(& \Xi_c( n'_{c'} + 2m_c h) ) \\ & = F_c\left( \pi_c\left(\Xi_c( n'_{c'} + 2m_c h) - \Xi_c(n'_{c'})\right) + \Xi_c(n'_{c'}) \right) + O( \exp( -\eta^{-2C_4} ) ).\end{align*}
We can rewrite this as
\begin{equation}\label{fern}
 F_c(\Xi_c( n'_{c'} + 2m_c h) ) = F'_{c'}( \Xi'_{c'}( n'_{c'} + 2m_c h ) ) + O( \exp( -\eta^{-2C_4} ) )
\end{equation}
where $F'_{c'}: G'_c \to [-1,1]$ is the $1$-Lipschitz function
$$ F'_{c'}(x) \coloneqq F_c( \psi_c^{-1}( \delta_c(x) ) ) + \Xi_c(n'_{c'})$$
and $\Xi'_{c'}: n'_{c'} + B(S'_c, \rho'_c) \to G'_c$ takes the form
$$ \Xi'_{c'}( n'_{c'} + 2mc h) \coloneqq \delta_c^{-1}\left( \psi_c \left( \pi_c\left(\Xi_c( n'_{c'} + 2m_c h) - \Xi_c(n'_{c'})\right) + \Xi_c(n'_{c'})\right)\right).$$
The map $\Xi'_{c'}$ is the composition of a locally quadratic map with three locally linear maps, and is hence also locally quadratic.  This concludes the construction of all the required quadratic data $G'_{c'}, F'_{c'}, \Xi'_{c'}$ when $c'$ arises from a poorly distributed $c$.

It remains to verify the claims \eqref{bd-1}-\eqref{bd-7} of Theorem \ref{bad-dim}.  The claim \eqref{bd-1} is clear; in fact, the frequency sets $S'_{c'}$ are either equal to their original counterparts $S_c$ or have the addition of just one further frequency $\xi_c$, so we even obtain the improved bound $d(v') \leq d(v) + 1$ in our construction here.  Since the dilated torus $G'_{c'}$ is either equal to $G_c$ when $c$ is not poorly distributed, or has one lower dimension than $G_c$ if $c$ is poorly distributed, we obtain the bounds \eqref{bd-2}, \eqref{bd-3}.  Since $\rho'_{c'}$ is either equal to $\rho_c$ when $c$ is not poorly distributed, or $\exp(-\eta^{-6C_4}) \rho_c$ when $c$ is poorly distributed, we obtain\eqref{bd-4} (with a little room to spare).  As for the volume bound, $G'_{c'}$ clearly has the same volume as $G_c$ when $c$ is not poorly distributed, and when $c$ is poorly distributed we have
$$ \vol(G'_{c'}) = \exp( -\eta^{-4C_2} \dim(\tilde G_{c'})) \vol(\tilde G_{c'})$$
which by \eqref{voil}, \eqref{bound-1} is bounded in turn by $\exp(-\eta^{-5C_2}) \vol(G_c)$, which yields \eqref{bd-5}, again with a little bit of room to spare (because the bounds here only increased the volume by factors that involved $C_2$ rather than $C_3$).

Now we establish \eqref{bd-6}.  From the triangle inequality we have
\begin{align*}
|\waste(v') - \waste(v)| &\leq |\E f(\a_{v'}) - \E f(\a_v) | \\
&\leq \sum_{c \in C} \P(\c=c) |\E(f(\a_{v'})|\c=c) - \E(f(\a_v)|\c=c) |
\end{align*}
so it will suffice to show that
\begin{equation}\label{fff}
 \left|\E(f(\a_{v'})|\c=c) - \E(f(\a_v)|\c=c)\right | \leq \eta^{C_3}
\end{equation}
for each $c$ in the essential range of $\c$.

The claim is trivial when $c$ is not poorly distributed, since in this case $\a_v$ and $\a_{v'}$ have identical distribution after conditioning to $\c=c$.  If $c$ is poorly distributed, then (after conditioning to $\c=c$) $\a_v$ is drawn regularly from $n_c + B(S_c, \rho_c/2)$, while $\a_{v'}$ has the distribution of $\a_v + 2m_c \h_c$ where $\h_c$ is drawn regularly from $B( S_c \cup \{\xi_c\}, \exp(-\eta^{-6C_4}) \rho_c)$
 independently of $\a_v$ (after conditioning to $\c=c$).  The required bound \eqref{bd-6} now follows from Lemma \ref{ati} (and \eqref{bound-1}).

Finally, we prove \eqref{bd-7}.  Our task is to show that
$$ 
\E |f(\a_{v'}) - \f_{v'}(\a_{v'})|^2 \leq \E |f(\a_v) - \f_v(\a_v)|^2 + \eta^{3C_2}.$$
By the triangle inequality as before, it suffices to show that
$$ 
\E\left(|f(\a_{v'}) - \f_{v'}(\a_{v'})|^2|\c=c\right) \leq \E\left(|f(\a_v) - \f_v(\a_v)|^2|\c=c\right) + \eta^{3C_2}$$
for all $c$ in the essential range of $\c$.  This is trivial for $c$ not poorly distributed, so assume $c$ is poorly distributed.  From \eqref{fern} we then have
$$
\f_{v'}(\a_{v'}) = \f_v(\a_{v'}) + O( \exp( - \eta^{-2C_4} ) )$$
and also
$$ \f_v( a ) = F_c( \Xi_c( a ) )$$
for $a \in B(S_c, \rho_c)$, so by the triangle inequality it suffices to show that
$$ 
\E\left(|f(\a_{v'}) - F_c(\Xi_c(\a_{v'}))|^2|\c=c\right) \leq \E\left(|f(\a_v) - F_c(\Xi_c(\a_v))|^2|\c=c\right) + \eta^{4C_2}$$
(say).  But this follows by repeating the proof of \eqref{fff}, with the function $f$ replaced by $|f - F_c \circ \Xi_c|^2$.  This completes the proof of Theorem \ref{bad-dim}.

\section{Bad approximation implies energy decrement}\label{baded-sec}

The remaining task in the paper is to prove Theorem \ref{bad-ed}. In this section we will establish this result contingent on a local inverse Gowers norm theorem (Theorem \ref{locu3}) that will be proven in later sections.  We begin by stating the (rather technical) precise form of that theorem that we will need.

\begin{theorem}[Local inverse $U^3$ theorem]\label{locu3} Let $p$ be a prime, and let $S$ be a subset of $\Z/p\Z$ containing at least one non-zero element. Let $\eta$ be a real parameter with $0 < \eta < \frac{1}{2}$.  
Let $K$ be the quantity
\begin{equation}\label{K-def}
 K \coloneqq \frac{1}{\eta} + |S|,
\end{equation}
and let $\rho_0, \rho_1,\rho_2,\dots, \rho_{10}$ be real numbers satisfying
$$ 0 < \rho_{10} < \dots < \rho_0 < 1/2$$
as well as the separation condition
\begin{equation}\label{rho-sep}
\rho_{i+1} \geq \exp(K^{C_2}) \rho_i
\end{equation}
for all $i=0,\dots,9$.  Assume that the prime $p$ is huge relative to the reciprocal of these parameters, in the sense that
\begin{equation}\label{cstar2}
 p \geq \rho_{10}^{-K^{C_2^3}}.
\end{equation}
Let $f: \Z/p\Z \to \C$ be a $1$-bounded function such that
\begin{equation}\label{soap-3}
\begin{split}
& |\E f(\h_0+\h_1+\h_2) \overline{f}(\h_0+\h'_1+\h_2) \overline{f}(\h'_0+\h_1+\h_2) f(\h'_0+\h'_1+\h_2) \\
& \quad \overline{f}(\h_0+\h_1+\h'_2) f(\h_0+\h'_1+\h'_2) f(\h'_0+\h_1+\h'_2) \overline{f}(\h'_0+\h'_1+\h'_2)|\\
&\quad \quad  \geq \eta
\end{split}
\end{equation}
whenever $\h_0,\h'_0,\h_1,\h'_1,\h_2,\h'_2$ are drawn independently and regularly from $B(S,\rho_0)$, $B(S,\rho_0)$, $B(S,\rho_1)$, $B(S,\rho_1)$, $B(S,\rho_2)$, and $B(S,\rho_2)$ respectively.
Then there exists a positive integer $k < \exp(K^{O(C_1)})$, a set $S' \subset \Z/p\Z$, $S' \supset S$, with
\begin{equation}\label{dim-add}
|S'| \leq |S| + O(\eta^{-O(C_1)}),
\end{equation}
a locally quadratic phase $\phi: B( S', \rho_9 ) \to \R/\Z$, and a function $\beta: \Z/p\Z \to \Z/p\Z$ such that
\begin{equation}\label{u3-cor}
 \sum_{n \in \Z/p\Z} \P( \n = n) \left|\E f(n + k \m ) e\left( - \phi(\m) - \frac{\beta(n) \m}{p} \right)\right| \gg \eta^{O(C_1)}
\end{equation}
if $\n, \m$ are drawn independently and regularly from $B_S(0,\rho_0)$ and $B_{S'}(0, \rho_{10})$ respectively.
\end{theorem}

\emph{Remarks.} The parameters $\rho_3,\dots, \rho_8$ do not have any role in the statement of this result, but they appear in the proof. We have retained them to avoid a potentially confusing relabelling.

Informally, this theorem asserts that if $f$ has a large $U^3$ norm on $B(S,\rho_0)$, then $f$ will correlate with a locally quadratic phase $n+km \mapsto \phi(m) + \frac{\beta(n) m}{p}$ on translates $n + k \cdot B_{S'}(0,\rho_{10})$ of $k \cdot B_{S'}(0,\rho_{10})$, with polynomial bounds on the correlation.  Although we will not make crucial use of this fact in our arguments, it may be noted that the homogeneous component $\phi$ of this locally quadratic phase does not depend on the translation parameter $n$.  In the bounded rank case $|S| = O(1)$, a theorem very roughly of this form was established in \cite{gt-inverseu3}; the key point in Theorem \ref{locu3} is that the inverse theory of \cite{gt-inverseu3} can be localised to a Bohr set without having the lower bound $\eta^{O(C_1)}$ on the correlation appearing in
\eqref{u3-cor} depend on the rank $|S|$ or radius $\rho_0$ of the Bohr set (although these parameters certainly influence the range of the variables $\n,\m$ appearing in \eqref{u3-cor}).\vspace{11pt}

The proof of Theorem \ref{locu3} will occupy most of the remainder of this paper. To a large extent, it may be understood separately of our main arguments, requiring little of the notation of Section \ref{overview-sec}, for example. In this section, we will assume Theorem \ref{locu3} and use it to establish Theorem \ref{bad-ed}.

For the remainder of this section, the notation and hypotheses will be as in Theorem \ref{bad-ed}.  Namely, we fix a prime $p$, a function $f: \Z/p\Z \to [-1,1]$, and a parameter $0 < \eta \leq 1/10$, and assume \eqref{p-large}.  We also suppose that 
$$ v = \left( C, \c, (n_c + B(S_c,\rho_c))_{c \in C}, (G_c)_{c \in C}, (F_c)_{c \in C}, (\Xi_c)_{c \in C} \right)$$
is a structured local approximant obeying \eqref{bound-1}-\eqref{bound-4}, and one of \eqref{tp-bad-again-0} or \eqref{tp-bad-again} holds.  Our objective is to construct a structured local approximant 
$$ v' = \left( C', \c', (n'_{c'} + B(S'_{c'},\rho'_{c'}))_{c' \in C'}, (G'_{c'})_{c' \in C'}, (F'_{c'})_{c' \in C'}, (\Xi'_{c'})_{c' \in C'} \right)$$
obeying the bounds \eqref{be-1}-\eqref{be-6}.   The situation here is a formalisation of Example 8 from Section \ref{overview-sec}.

Let $\a = \a_v,\r = \r_v,\f = \f_v$ be the random variables associated to $v$ in Definition \ref{sla-def}.
We can unify the hypotheses \eqref{tp-bad-again-0}, \eqref{tp-bad-again} by introducing the quadrilinear form
$$ \Lambda_{\a,\r}(\f_0,\f_1,\f_2,\f_3) \coloneqq \E \f_0(\a) \f_1(\a+\r) \f_2(\a+2\r) \f_3(\a+3\r),$$
defined for arbitrary random (or deterministic) bounded functions $\f_0,\f_1,\f_2,$ $\f_3: \Z/p\Z \to \R$.  From the definitions of $\Err_1$ and $\Err_4$ (just prior to \eqref{coo}), the hypothesis \eqref{tp-bad-again-0} may be written as
$$ |\Lambda_{\a,\r}(f,1,1,1) - \Lambda_{\a,\r}(\f, 1, 1, 1)| > \eta,$$
while \eqref{tp-bad-again} can be similarly written as
$$ |\Lambda_{\a,\r}(f,f,f,f) - \Lambda_{\a,\r}(\f, \f, \f, \f)| > \eta.$$
Applying the triangle inequality and the quadrilinearity of $\Lambda_{\a,\r}$, we conclude that
$$ |\Lambda_{\a,\r}(\f_0,\f_1,\f_2,\f_3)| \gg \eta$$
for some random functions $\f_0,\f_1,\f_2,\f_3$, each of which is either equal to $1$, $f$, or $f-\f$, and with at least one of the functions $\f_0,\f_1,\f_2,\f_3$ equal to $f-\f$.  For sake of concreteness we will assume that it is $\f_3$ that is equal to $f-\f$, thus
\begin{equation}\label{tff}
 |\Lambda_{\a,\r}(\f_0,\f_1,\f_2,f-\f)| \gg \eta;
\end{equation}
the other cases are treated similarly (with some changes to the numerical constants below) and are left to the interested reader.

We can write the left-hand side of \eqref{tff} as
$$ \left|\sum_{c \in C} \P(\c=c) \E\left( \f_0(\a) \f_1(\a+\r) \f_2(\a+2\r) (f-\f)(\a+3\r)|\c = c\right)\right|.$$
Applying Lemma \ref{popular}, we conclude that with probability $\gg \eta$, the variable $\c$ attains a value $c$ for which we have
the lower bound
\begin{equation}\label{lowerc}
|\E( \f_0(\a) \f_1(\a+\r) \f_2(\a+2\r) (f-\f)(\a+3\r)|\c = c)| \gg \eta.
\end{equation}

We now use a local version of the standard ``generalised von Neumann theorem'' argument (based on several applications of the Cauchy-Schwarz inequality) to obtain some local correlation of $f-f_c$ with a quadratic phase.

\begin{proposition}\label{ait}  Let the notation and hypotheses be as above.  For each $(a,c)$ in the essential range of $(\a,\c)$, there exists a natural number $k_{a,c}$ with
\begin{equation}\label{kac}
 1 \leq k_{a,c} < \eta^{-C_3},
\end{equation}
a set $\tilde S_{a,c} \subset \Z/p\Z$ with $\tilde S_{a,c} \supset S_c$ and
\begin{equation}\label{spc-sc}
|\tilde S_{a,c}| \leq |S_c| + \eta^{-C_2},
\end{equation}
and a locally quadratic function $\gamma_{n,a,c}: B(\tilde S_{a,c}, \exp(-\eta^{-11C_4}) \rho_c) \to \R/\Z$ for each $n \in \Z/p\Z$, such that
\begin{equation}\label{ait-conc}
\begin{split}
\mathrm{Re} &\sum_{a,c \in \Z/p\Z} \P( \a=a, \c=c) \\
& \quad \times \E\left( (f-f_c)(a + 6\n + 6 k_{a,c} \m) e(-\gamma_{\n,a,c}(\m)) | \a = a, \c=c\right) \geq \eta^{C_2/10},
\end{split}
\end{equation}
where, after conditioning to the event $\a=a, \c=c$, the random variables $\n$ and $\m$ are drawn regularly and independently from the Bohr sets $B(S_c, \exp(-\eta^{-2C_4}) \rho)$ and $B(\tilde S_{a,c}, \exp(-\eta^{-12C_4}) \rho_c)$ respectively.
\end{proposition}

\begin{proof}  Suppose for now that $c$ obeys \eqref{lowerc}.  From Definition \ref{sla-def}, once we condition to the event $\c=c$, the random variables $\a, \r$ are independent and regularly drawn from $B(S_c, \rho_c/2)$ and $B(S_c, \exp(-\eta^{-C_4}) \rho_c )$ respectively; from \eqref{bound-2} we have the bounds 
\begin{equation}\label{scb}
|S_c| \leq 8 \eta^{-3C_2} \quad \mbox{and} \quad \rho_c \geq \exp( -\eta^{-2C_5} ).
\end{equation}  
Also, the function $\f$ is now the deterministic function
$$ f_c(a) \coloneqq F_c( \Xi_c(a) )$$
on the Bohr set $B( S_c, \rho_c)$, and $\f_0,\f_1,\f_2$ become deterministic functions $f_{0,c}$, $f_{1,c}$ and $f_{2,c}$ taking values in $[-2,2]$.  Thus we have
$$
|\E( f_{0,c}(\a) f_{1,c}(\a+\r) f_{2,c}(\a+2\r) f_{3,c}(\a+3\r)|\c = c)| \gg \eta
$$
where $f_{3,c} \coloneqq f - f_c$.

We now do a linear change of variable with conveniently chosen numerical coefficients that will facilitate a certain use of the Cauchy-Schwarz inequality to eliminate the bounded functions $f_{0,c},f_{1,c},f_{2,c}$, leaving only the function $f_{3,c}$.  Continuing to condition on the event that $\c=c$, let $\n_1, \n_2$ and $\n_3$ be drawn regularly and independently from the Bohr sets $B(S_c, \exp(-\eta^{-2C_4}) \rho_c)$, $B(S_c, \exp(-\eta^{-3C_4}) \rho_c)$, and $B(S_c, \exp(-\eta^{-4C_4}) \rho_c )$ respectively, independently of the previous random variables.  We can use Lemma \ref{ati} (and \eqref{scb}) to compare $\a$ with $\a - 3\n_2 - 12\n_3$, and conclude that
\begin{align*} | \E(f_{0,c}(\a - 3 \n_2 - 12 \n_3) & f_{1,c}(\a + \r - 3 \n_2 - 12 \n_3) f_{2,c}(\a + 2\r - 3 \n_2 - 12 \n_3) \times \\ & \times f_{3,c}(\a + 3\r - 3 \n_2 - 12 \n_3)|\c=c) | \gg \eta.\end{align*}
By another application of Lemma \ref{ati}, we may compare $\r$ with $\r + 2\n_1 + 3 \n_2 + 6 \n_3$, and conclude that
\begin{align*} | \E(f_{0,c}(\a - 3 \n_2 - 12 \n_3) & f_{1,c}(\a + \r + 2\n_1 - 6 \n_3) f_{2,c}(\a + 2\r + 4\n_1 + 3 \n_2 ) \times \\ & \times f_{3,c}(\a + 3\r + 6 (\n_1+\n_2+\n_3))|\c=c) | \gg \eta.\end{align*}
Finally, we use Lemma \ref{ati} to replace $\a$ by $\a-3\r$, so that
\begin{align*} | \E(f_{0,c}(\a & - 3\r - 3 \n_2 - 12 \n_3)  f_{1,c}(\a - 2 \r + 2\n_1 - 6 \n_3) \times \\ & \times f_{2,c}(\a - \r + 4\n_1 + 3 \n_2 )f_{3,c}(\a + 6 (\n_1+\n_2+\n_3))|\c=c) | \gg \eta.\end{align*}
The purpose of this odd-seeming change of variables is that each of the functions $f_{0,c},f_{1,c},f_{2,c}$ now has an argument that involves only two of the three random variables $\n_1,\n_2,\n_3$, whilst the argument of the key function $f_{3,c}$ depends on $\n_1,\n_2,\n_3$ only through their sum $\n_1+\n_2+\n_3$.

One can achieve a similar effect for the other three choices $\f_0,\f_1,\f_2$ for key function by suitable adjustment to the constants above; we leave the details to the interested reader.

By Lemma \ref{popular}, we see that with probability $\gg \eta$ (conditioning on $\c=c$), the random variable $\a$ attains a value $a$ such that
\begin{align}\nonumber | \E( &f_{0,c}(a - 3\r - 3 \n_2 - 12 \n_3) f_{1,c}(a - 2\r + 2\n_1 - 6 \n_3)  \times \\ & \times f_{2,c}(a - \r + 4\n_1 + 3 \n_2 ) f_{3,c}(a + 6 (\n_1+\n_2+\n_3))|\a=a, \c=c) | \gg \eta.\label{abound}
\end{align}
Let $a$ be such that \eqref{abound} holds.  We can then find an $r \in \Z/p\Z$ (depending on $a,c$) such that
\begin{align*}
 | \E(& f_{0,c}(a - 3r - 3 \n_2 - 12 \n_3) f_{1,c}(a - 2r + 2\n_1 - 6 \n_3)  \times \\ & \times  f_{2,c}(a - r + 4\n_1 + 3 \n_2 )f_{3,c}(a + 6 (\n_1+\n_2+\n_3))|\a=a, \c=c) | \gg \eta.
\end{align*}
We now suppress the additive structure on the first three arguments by rewriting the above bound as
$$ | \E(f_{0,c,a}(\n_2, \n_3) f_{1,c,a}(\n_1, \n_3) f_{2,c,a}(\n_1, \n_2) f_{3,c}(a + 6 (\n_1+\n_2+\n_3))|\c=c) | \gg \eta$$
where $f_{0,c,a}, f_{1,c,a}, f_{2,c,a}: \Z/p\Z \times \Z/p\Z \to [-2,2]$ are bounded functions whose exact form
\begin{align*}
f_{0,c,a}(n_2,n_3) & \coloneqq f_{0,c}( a - 3r - 3 n_2 - 12 n_3 ) \\
f_{1,c,a}(n_1,n_3) & \coloneqq f_{1,c}( a - 2r + 2 n_1 - 6 n_3 ) \\
f_{2,c,a}(n_1,n_2) & \coloneqq f_{2,c}( a - r + 4 n_1 + 3 n_2 )
\end{align*}
will not be relevant in the arguments that follow.

We can eliminate the factor $f_{0,c,a}$ using Lemma \ref{cauchy-schwarz} to conclude that
\begin{align*}
&|\E( f_{1,c,a}(\n_1, \n_3) f_{1,c,a}(\n'_1, \n_3) f_{2,c,a}(\n_1, \n_2) f_{2,c,a}(\n'_1, \n_2) \\
&\quad f_{3,c}(a + 6 (\n_1+\n_2+\n_3)) f_{3,c}(a + 6 (\n'_1+\n_2+\n_3))|\a = a, \c=c) | \gg \eta^2
\end{align*}
where $\n'_1$ is an independent copy of $\n_1$ (and also independent of $\n_2,\n_3$) on the event $\a = a, \c=c$.
We can similarly apply Lemma \ref{cauchy-schwarz} to eliminate the $f_{1,c,a}(\n_1, \n_3) f_{1,c,a}(\n'_1, \n_3)$ variables to conclude that
\begin{align*}
&|\E( f_{2,c,a}(\n_1, \n_2) f_{2,c,a}(\n'_1, \n_2) f_{2,c,a}(\n_1, \n'_2) f_{2,c,a}(\n'_1, \n'_2) \\
&\quad f_{3,c}(a + 6 (\n_1+\n_2+\n_3)) f_{3,c}(a + 6 (\n'_1+\n_2+\n_3) )\\
&\quad f_{3,c}(a + 6 (\n_1+\n'_2+\n_3)) f_{3,c}(a + 6 (\n'_1+\n'_2+\n_3))|\a = a, \c=c) | \gg \eta^4
\end{align*}
and finally apply Lemma \ref{cauchy-schwarz} to eliminate the $f_{2,c,a}$ terms and arrive at
\begin{align*}
&|\E( f_{3,c}(a + 6 (\n_1+\n_2+\n_3)) f_{3,c}(a + 6 (\n'_1+\n_2+\n_3))   \\
&\quad f_{3,c}(a + 6 (\n_1+\n'_2+\n_3)) f_{3,c}(a + 6 (\n'_1+\n'_2+\n_3)) \\ & \quad f_{3,c}(a + 6 (\n_1+\n_2+\n'_3)) f_{3,c}(a + 6 (\n'_1+\n_2+\n'_3)) \\
&\quad f_{3,c}(a + 6 (\n_1+\n'_2+\n'_3)) f_{3,c}(a + 6 (\n'_1+\n'_2+\n'_3))|\a = a, \c=c) | \gg \eta^8,
\end{align*}
where $\n'_2, \n'_3$ are independent copies of $\n_2,\n_3$ respectively on $\a = a, \c=c$, with $\n_1,\n_2,\n_3,\n'_1,\n'_2,\n'_3$ all independent relative to $\a = a, \c=c$.

We now apply Theorem \ref{locu3}, replacing $\eta$ by a small multiple of $\eta^8$, and choosing $\rho_i \coloneqq \exp(-\eta^{-(i+2)C_4}) \rho$ for $i=0,\dots,10$, and using the bounds \eqref{scb}, \eqref{p-large} to justify the hypothesis \eqref{cstar2}.  We conclude that for $c$ obeying \eqref{lowerc} and $a$ obeying \eqref{abound}, we can find a natural number $k_{a,c}$ obeying \eqref{kac}, a set $\tilde S_{a,c}$ with $S_c \subset \tilde S_{a,c} \subset \Z/p\Z$ obeying \eqref{spc-sc},
a locally quadratic function $\phi_{a,c}: B( \tilde S_{a,c}, \exp(-\eta^{-11C_4}) \rho) \to \R/\Z$, and a function $\beta_{a,c}: \Z/p\Z \to \Z/p\Z$ such that
\begin{align*}
&\sum_{n \in \Z/p\Z} \P( \n = n | \a = a, \c = c) \\
&\quad |\E( f_3(a + 6n + 6k \m ) e( - \phi_{a,c}(\m) - \beta_{a,c}(n) \m ) | \a=a,\c=c )| \gg \eta^{C_2/20}
\end{align*}
if $\n, \m$ are drawn independently and regularly from $B(S_c,\exp(-\eta^{-2C_4})\rho_c)$ and $B(S_{a,c}, \exp(\eta^{-12C_4}) \rho_c)$ respectively on the event $\a=a, \c=c$.  Taking expectations in $\a$ (and choosing $S_{a,c} = S_c$, $\phi_{a,c} = 0$ and $\beta_{a,c} = 0$ if \eqref{lowerc} or \eqref{abound} is not satisfied), we conclude that
\begin{align*}
&\sum_{n, a, c \in \Z/p\Z} \P( \n = n, \a = a, \c = c) \\
&\quad |\E( f_3(a + 6n + 6k \m ) e( - \phi_{a,c}(\m) - \beta_{a,c}(n) \m ) | \a=a,\c=c )| \geq \eta^{C_2/10}.
\end{align*}
In particular, if we set $\gamma_{n,a,c}(m) \coloneqq \phi_{a,c}(m) + \beta_{a,c}(n) m + \theta_{n,a,c}$ for a suitable phase $\theta_{n,a,c} \in \R/\Z$, then $\gamma_{n,a,c}$ is locally quadratic on $B( \tilde S_{a,c}, \exp(-\eta^{-11C_4}) \rho)$ and
\begin{align*}
&\mathrm{Re} \sum_{n,a,c \in \Z/p\Z} \P( \n = n, \a = a, \c = c) \\
&\quad \E( f_3(a + 6n + 6k \m ) e( - \gamma_{n,a,c}(\m) ) | \a=a,\c=c )| \geq \eta^{C_2/10},
\end{align*}
giving the claim.
\end{proof}

Let $\n, \m, k_{a,c}, \tilde S_{a,c}, \gamma_{n,a,c}$ be as in the above proposition.
The conclusion \eqref{ait-conc} of Proposition \ref{ait} may be rewritten more compactly as
\begin{equation}\label{ait-conc2}
\mathrm{Re} \E( (f-\f)(\a + 6\n + 6 k_{\a,\c} \m) e(-\gamma_{\n,\a,\c}(\m))) \geq \eta^{C_2/10}.
\end{equation}
We now introduce the modified random function $\f': \Z/p\Z \to [-2,2]$ by the formula
\begin{equation}\label{flp}
 \f'(l) \coloneqq  \f(l) + \eta^{C_2/2} \cos\left(2\pi \gamma_{\n,\a,\c}\left( \frac{l - \a - 6\n}{6k_{\a,\c}} \right)\right),
\end{equation}
where we extend $\gamma_{n,a,c}$ arbitrarily outside of $B(S'_c, \exp(-\eta^{-11C_4}) \rho_c)$.
Note from \eqref{kac} and \eqref{p-large} that we can divide by $6k_{\a,\c}$ in $\Z/p\Z$ without difficulty.

We claim that the function $\f'$ is a little closer to $f$ than $\f$ is.

\begin{lemma}\label{ped}  We have
$$ \E |(f-\f')(\a + 6\n + 6k_{\a,\c} \m)|^2 \leq \Energy(v) - \eta^{C_2}.$$
\end{lemma}

\begin{proof}  From \eqref{flp} we have
$$ \f'( \a + 6\n + 6k_{\a,\c} \m ) = \f( \a + 6\n + 6 k_{\a,\c} \m ) + \eta^{C_2/2} \cos(2\pi \gamma_{\n,\a,\c}( \m )),$$
and so
\begin{align}\nonumber
 |(f-\f')(& \a + 6\n + 6k_{\a,\c} \m)|^2 
\\ &= |(f-\f)(\a + 6\n + 6k_{\a,\c} \m)|^2 \nonumber \\ \nonumber
&\quad - 2 \eta^{C_2/2} \E (f-\f)(\a + 6\n + 6k_{\a,\c} \m) \cos(2\pi \gamma_{\n,\a,\c}( \m )) \\
&\quad + O( \eta^{C_2} ).\label{fin}
\end{align}
On the other hand, for any $(a,c)$ in the essential range of $(\a,\c)$, we may use Lemma \ref{ati} to compare $\n$ with $\n + k_{\a,\c} \m$, and conclude that
\begin{align*} \E( |(f-\f)(a + 6\n + & 6k_{a,c} \m)|^2 | \a = a, \c = c) \\ & = \E( |(f-\f)(a + 6\n)|^2 | \a = a, \c = c) + O(\eta^{2C_3})\end{align*}
(say), and hence on taking expectations in $\a$
\begin{align*} \E( |(f-\f)(\a + 6\n + & 6k_{\a,c} \m)|^2 | \c = c) \\ & = \E( |(f-\f)(\a + 6\n)|^2 | \c = c) + O(\eta^{2C_3}).\end{align*}
Applying Lemma \ref{ati} again to compare $\a$ with $\a+6\n$, we conclude that
$$ \E( |(f-\f)(\a + 6\n + 6k_{\a,c} \m)|^2 | \c = c) = \E( |(f-\f)(\a)|^2 | \c = c) + O(\eta^{2C_3}).$$
and hence on taking averages in $\c$
\begin{equation}\label{fin2}
 \E( |(f-\f)(\a + 6\n + 6k_{\a,c} \m)|^2 | \c = c) = \Energy(v) + O(\eta^{2C_3}).
\end{equation}
Taking expectations in \eqref{fin} and using \eqref{flp}, \eqref{fin2}, we obtain the claim.
\end{proof}

There is a very minor technical issue that $\f'$ does not quite take values in $[-1,1]$, which is what is needed in the definition of an approximant.  However, this is easily fixed by truncation, or more precisely by introducing the random function $\f'': \Z/p\Z \to [-1,1]$ defined by
\begin{equation}\label{fail}
 \f''(l) \coloneqq \min(\max(\f'(l),-1), 1).
\end{equation}
Since $f(l)$ already lies in $[-1,1]$, we see that $\f''(l)$ is at least as close to $f(l)$ as $\f'(l)$ is, thus we have the pointwise bound
$$ |(f-\f'')(l)| \leq |(f-\f')(l)|$$
for any $l \in \Z/p\Z$.  From the above lemma, we thus have
\begin{equation}\label{endoc}
 \E |(f-\f'')(\a + 6\n + 6k_{\a,\c} \m)|^2 \leq \Energy(v) - \eta^{C_2}.
\end{equation}

We can now construct the new structured approximant
$$ v' = \left( C', \c', (n'_{c'} + B(S'_{c'},\rho'_{c'}))_{c' \in C'}, (G'_{c'})_{c' \in C'}, (F'_{c'})_{c' \in C'}, (\Xi'_{c'})_{c' \in C'} \right)$$
as follows.  We write the dilated torus $G_c$ as $G_c = \prod_{i=1}^{\dim(G_c)} \R / \lambda_{i,c}\Z$.

\begin{itemize}
\item[(i)]  We set $C' \coloneqq (\Z/p\Z) \times (\Z/p\Z) \times C$ and $\c' \coloneqq (\n,\a,\c)$.
\item[(ii)]  If $c' = (n,a,c)$ is in $C'$, we set
\begin{align*}
n'_{c'} &\coloneqq a + 6n \\
S'_{c'} &\coloneqq (6k_{a,c})^{-1} \cdot \tilde S_{a,c} \\
\rho'_{c'} &\coloneqq \exp(-\eta^{-12C_4}) \rho_c \\
G'_{c'} &\coloneqq \prod_{i=1}^{\dim(G_c)} (\R / 100 \lambda_{i,c} \Z) \times (\R/\Z)
\end{align*}
\item[(iii)] If $c' = (n,a,c)$ is in $C'$, we define $F'_{c'}: G'_{c'} \to [-1,1]$ to be the function
$$ F'_{c'}( x, y ) \coloneqq \min\left(\max\left(F_c\left( \frac{1}{100} \cdot x \right) + \eta^{C_2/2} \cos(2\pi y),-1\right),1\right)$$
for $x \in \prod_{i=1}^{\dim(G_c)} (\R / 100 \lambda_{i,c} \Z)$ and $y \in \R/\Z$, where $x \mapsto \frac{1}{100} \cdot x$ is the obvious contraction map from $\prod_{i=1}^{\dim(G_c)} (\R / 100 \lambda_{i,c} \Z)$ to $\prod_{i=1}^{\dim(G_c)} (\R / \lambda_{i,c} \Z)$.
\item[(iv)] If $c' = (n,a,c)$ is in $C'$, we define $\Xi'_{c'}: n'_{c'} + B( S'_{c'}, \rho'_{c'} ) \to G'_{c'}$ by the formula
$$
\Xi'_{c'}(l) \coloneqq \left( 100 \cdot \Xi_c(l), \gamma_{n,a,c}\left( \frac{l-a-6n}{6k_{a,c}} \right) \right) $$
for $l \in n'_{c'} + B( S'_{c'}, \rho'_{c'} )$ (which implies in particular that $\frac{l-a-6n}{6k_{a,c}} \in B(\tilde S_{a,c}, \exp(-\eta^{-12C_4}) \rho_c)$), where $x \mapsto 100 \cdot x$ is the obvious dilation map from $\prod_{i=1}^{\dim(G_c)} (\R / \lambda_{i,c} \Z)$ to $\prod_{i=1}^{\dim(G_c)} (\R / 100 \lambda_{i,c} \Z)$ (the inverse of the map $x \mapsto \frac{1}{100} \cdot x$ from part (iii)).
\end{itemize}

Since $F_c$ is $1$-Lipschitz, it is easy to see (thanks to the contraction by $\frac{1}{100}$) that $F'_{c'}$ is also $1$-Lipschitz; similarly, as $\Xi_c$ and $\gamma_{n,a,c}$ are locally quadratic on $n_c + B(S_c, \rho_c)$ and $B( \tilde S_{a,c}, \exp(\eta^{-11C_4}) \rho_c)$ respectively, we see that $\Xi'_{c'}$ is also locally quadratic on $n'_{c'} + B(S'_{c'}, \rho'_{c'})$.  From \eqref{flp}, \eqref{fail}, Definition \ref{sla-def}, and the above constructions we see that
$$ \f'' = \f_{v'}$$
and hence by \eqref{endoc}
$$ \E |(f-\f_{v'})(\a + 6\n + 6k_{\a,\c} \m)|^2 \leq \Energy(v) - \eta^{C_2}.$$
From Definition \ref{sla-def} and the above constructions, we also see that $\a_{v'}$ has the same distribution as $\a + 6\n + 6k_{\a,\c} \m$ (after conditioning to any positive probability event of the form $(\n,\a,\c)=(n,a,c)$), which gives the required energy decrement \eqref{be-6}.

The bound \eqref{be-1} follows from \eqref{spc-sc}, while from construction we clearly have $\dim(G'_{c'}) = \dim(G_c)+1$, which gives \eqref{be-2}.   Since we have $\rho'_{c'} \coloneqq \exp(-\eta^{-12C_4}) \rho_c$, the bound \eqref{be-3} is clear; also, from \eqref{bound-2} we have
\[ \vol(G'_{c'}) = 100^{\dim(G'_{c'})} \vol(G_c) \leq \exp( O( \eta^{-2C_2} ) \vol(G_c) 
\]
which gives \eqref{be-4}.  It remains to establish \eqref{be-5}. By the definition of $\Err_1$ (just before \eqref{coo}) and the triangle inequality, it suffices to show that
$$ |\E f( \a_{v'} ) - \E f( \a ) | \leq \eta^{C_3}.$$
But as mentioned previously, $\a_{v'}$ has the same distribution as $\a + 6\n + 6k_{\a,\c}\m$, and by using Lemma \ref{ati} as in the proof of Lemma \ref{ped} we have
$$ \E f( \a + 6\n + 6k_{\a,\c}\m ) = \E f(\a) + O( \eta^{2C_3} )$$
giving the claim.  This completes the proof of Theorem \ref{bad-ed}, assuming the local inverse Gowers norm theorem (Theorem \ref{locu3}).

\section{Local inverse $U^3$ theorem}\label{inverse-sec}

We now turn to the proof of Theorem \ref{locu3}, which is the last component needed in the proof of Theorem \ref{main}.  Let us begin by recalling the setup of this theorem.  We let $S$ be a subset of $\Z/p\Z$, take a parameter $\eta$ satisfying $0 < \eta < \frac{1}{2}$, and define the quantity $K$ by \eqref{K-def}, thus
\begin{equation}\label{K-def-2}
 \frac{1}{\eta}, |S| \leq K.
\end{equation}
We suppose that
$$ 0 < \rho_{10} < \dots < \rho_0 < 1/2$$
are scales obeying the separation condition \eqref{rho-sep} and the largeness condition \eqref{cstar2}, and suppose that $f: \Z/p\Z \to \C$ is a $1$-bounded function obeying \eqref{soap-3}.  Our task is to locate a natural number $k$ with $k < \exp(K^{O(C_1)})$, a set $S'$ with $S \subset S' \subset \Z/p\Z$ obeying \eqref{dim-add}, a locally quadratic phase $\phi: B(S', \rho_9 ) \to \R/\Z$, and a function $\beta: \Z/p\Z \to \Z/p\Z$ obeying \eqref{u3-cor}.  We will initially work at the scale $\rho_0$, but retreat to smaller scales as the argument progresses (mainly in order to ensure that the error terms in Lemma \ref{ati} are negligible), until we are working at the final scales $\rho_9$ and $\rho_{10}$. Let us comment once more that the intermediate scales $\rho_3,\dots, \rho_8$ play no role in the actual statement of Theorem \ref{locu3}.

In this section, all sums will be over $\Z/p\Z$ unless otherwise stated.

\subsection{First step: associate a frequency $\xi(n_2)$ to each derivative of $f$}

We now begin the (lengthy) proof of this theorem, which broadly follows the same inverse $U^3$ strategy in previous literature \cite{gowers-4-aps,gt-inverseu3}, but localised to a Bohr set, the key aim being to reduce the dependence of constants on the rank or radius of this Bohr set as much as possible.  

The first step is to use the local inverse $U^2$ theorem (Theorem \ref{locu2}) to associate a frequency $\xi(n_2) \in \Z/p\Z$ to many ``derivatives'' $x \mapsto f(x+n_2) \overline{f(x)}$ of $f$.  

\begin{theorem}\label{first-step}  Let the notation and hypotheses be as in Theorem \ref{locu3}.  Then there exists a set $\Omega \subset B(S,2\rho_2)$ obeying the largeness condition
\begin{equation}\label{eta2}
\P( \h_2 - \h'_2 \in \Omega ) \geq \eta/4
\end{equation}
when $\h_2,\h'_2$ are drawn independently and regularly from $B(S,\rho_2)$, and a function $\xi: \Z/p\Z \to \Z/p\Z$ such that
\begin{equation}\label{sam}
\sum_{n_0 \in \Z/p\Z} \P( \n_0 = n_0) \left|\E f(n_0+\n_1+n_2) \overline{f}(n_0+\n_1) e_p( -\xi(n_2) \n_1 ) \right|^2 \geq \frac{\eta}{8} 1_{\Omega}(n_2)
\end{equation}
for all $n_2 \in \Z/p\Z$, and $\n_0, \n_1$ are drawn independently and regularly from $B(S,\rho_0), B(S,\rho_1)$ respectively.
\end{theorem}

\begin{proof}  For each $n_2 \in \Z/p\Z$, let $f_{n_2}: \Z/p\Z \to \C$ denote the $1$-bounded function
$$ f_{n_2}(n) \coloneqq f(n+n_2) \overline{f}(n).$$
Then we may rewrite the left-hand side of \eqref{soap-3} as
\begin{align*}
\bigg|
\E f_{\h_2-\h'_2}(\h_0+\h_2+\h_1) & \overline{f_{\h_2-\h'_2}}(\h_0+\h_2+\h'_1)\times \\ & \times \overline{f_{\h_2-\h'_2}}(\h'_0+\h_2+\h_1) f_{\h_2-\h'_2}(\h'_0+\h_2+\h'_1)\bigg|.
\end{align*}
By Lemma \ref{ati} and \eqref{rho-sep}, the random variables $\h_0, \h'_0$ differ in total variation from $\h_0+\h_2, \h'_0+\h_2$ respectively by at most $\eta/4$ (say).  We conclude that
$$  
|\E f_{\h_2-\h'_2}(\h_0+\h_1) \overline{f_{\h_2-\h'_2,0}}(\h_0+\h'_1) \overline{f_{\h_2-\h'_2}}(\h'_0+\h_1) f_{\h_2-\h'_2}(\h'_0+\h'_1)| \geq \eta/2.
$$
By the triangle inequality, the left-hand side is at most
$$ \sum_h \P( \h_2-\h'_2=h) |\E f_{h}(\h_0+\h_1) \overline{f_{h}}(\h_0+\h'_1) \overline{f_{h}}(\h'_0+\h_1) f_{h}(\h'_0+\h'_1)|.$$
The inner expectation is bounded by $1$.  Applying Lemma \ref{popular} (with $\a = \h_2-\h'_2$), we conclude that there is a set $\Omega \subset \Z/p\Z$ obeying \eqref{eta2} such that
$$ |\E f_{n_2}(\h_0+\h_1) \overline{f_{n_2}}(\h_0+\h'_1) \overline{f_{n_2}}(\h'_0+\h_1) f_{n_2}(\h'_0+\h'_1)| \geq \eta/4$$
for all $n_2 \in \Omega$.  Applying Theorem \ref{locu2}, we see that for each $n_2 \in \Omega$, there exists $\xi(n_2) \in \Z/p\Z$ such that
$$ \sum_{n_0 \in \Z/p\Z} \P( \n = n_0) \left| \E f_{n_2}(n_0+\n_1) e_p( -\xi(n_2) \n_1 )\right|^2 \geq \eta/8.$$
For $n_2 \not \in \Omega$, we set $\xi(n_2)$ arbitrarily (e.g. to zero).  The claim follows.
\end{proof}

\subsection{Second step: $\xi$ is approximately linear $1\%$ of the time}

The next step, following Gowers \cite{gowers-4-aps}, is to obtain some approximate linearity control on the function $\xi: \Z/p\Z \to \Z/p\Z$.  Define an \emph{additive quadruple} to be a quadruplet $\vec a = (a_{(1)},a_{(2)},a_{(3)},a_{(4)}) \in (\Z/p\Z)^4$ such that
\begin{equation}\label{aq-def}
 a_{(1)} + a_{(2)} = a_{(3)} + a_{(4)},
\end{equation}
and let $\operatorname{Q} \subset (\Z/p\Z)^4$ denote the space of all additive quadruples.  We call an additive quadruple $(a_{(1)},a_{(2)},a_{(3)},a_{(4)}) \in \operatorname{Q}$ \emph{bad} if 
\begin{equation}\label{bad-quad}
 \| \xi(a_{(1)}) + \xi(a_{(2)}) - \xi(a_{(3)}) - \xi(a_{(4)}) \|_{S} > \frac{K^{C_1}}{\rho_1},
\end{equation}
where the word norm $\|\|_{S}$ was defined in Definition \ref{word-def}.  Let $\BQ \subset \operatorname{Q}$ denote the space of all bad additive quadruples.

\begin{theorem}\label{second-step} Let the notation and hypotheses be as in Theorem \ref{locu3}, and let $\Omega$ and $\xi: \Z/p\Z \to \Z/p\Z$ be as in Theorem \ref{first-step}.  If $\h_2, \h'_2, \k_2, \k'_2$ are drawn independently and regularly from $B(S,\rho_2)$, then with probability $\gg \eta^{O(1)}$, one has
\begin{equation}\label{omi}
 (\h_2-\h'_2, \k_2-\k'_2, \k_2 - \h'_2, \h_2-\k'_2) \in \Omega^4 \cap (\operatorname{Q} \backslash \BQ).
\end{equation}
\end{theorem}

\begin{proof}  Let $\n_0,\n_1$ be drawn independently and regularly from the Bohr sets $B(S,\rho_0)$, $B(S,\rho_1)$ respectively. From \eqref{sam} we have
$$ \sum_{n_0} \P( \n_0=n_0) \left| \E f(n_0+\n_1+n_2) \overline{f}(n_0+\n_1) e_p( -\xi(n_2) \n_1 ) \right| \gg \eta$$
for any $n_2 \in \Omega$.  Using \eqref{eta2}, we conclude that
\begin{align*} \sum_{n_0} \sum_{n_2 \in \Omega} \P( \n_0=n_0, \h_2-\h'_2 = n_2) \big| \E & f(n_0+\n_1+n_2) \overline{f}(n_0+\n_1)\times  \\ & \times e_p( -\xi(n_2) \n_1 ) \big|\gg \eta^2,\end{align*}
where $\h_2,\h'_2$ are drawn independently and regularly from $B(S,\rho_2)$, and are independent of $\n_0,\n_1$.  By the pigeonhole principle, one can thus find $n_0 \in \Z/p\Z$ such that
$$ \sum_{n_2 \in \Omega} \P( \h_2-\h'_2 = n_2) \left| \E f(n_0+\n_1+n_2) \overline{f}(n_0+\n_1) e_p( -\xi(n_2) \n_1 ) \right| \gg \eta^2.$$
We can rewrite the left-hand side as
$$ \E F_{n_0}(\h_2-\h'_2) f(n_0+\n_1+\h_2-\h'_2) \overline{f}(n_0+\n_1) e_p( -\xi(\h_2-\h'_2) \n_1 ) $$
for some $1$-bounded function $F_{n_0}: \Z/p\Z \to \C$ depending on $n_0$. Using Lemma \ref{ati} to compare $\n_1$ with $\n_1 + \h'_2$, we conclude that
\begin{align*} |\E F_{n_0}(\h_2-\h'_2) f(n_0+\n_1+\h_2) \overline{f}(n_0+\n_1+\h'_2) & e_p( -\xi(\h_2-\h'_2) (\n_1+\h'_2) ) | \\ & \gg \eta^2.\end{align*}
We rearrange the left-hand side as
$$ \sum_{n_1} \P( \n_1 = n_1) \E f(n_0+n_1+\h_2) \overline{f}(n_0+n_1+\h'_2) G_{n_0,n_1}(\h_2,\h'_2)$$
where $G_{n_0,n_1}: \Z/p\Z \times \Z/p\Z \to \C$ is the $1$-bounded function
\begin{equation}\label{gn-def} 
G_{n_0,n_1}(h_2,h'_2) \coloneqq F_{n_0}(h_2-h'_2) e_p( -\xi(h_2-h'_2) (n_1+h'_2) ).
\end{equation}
By H\"older's inequality, we conclude that
$$ \sum_{n_1}\P( \n_1 = n_1) |\E f(n_0+n_1+\h_2) \overline{f}(n_0+n_1+\h'_2) G_{n_0,n_1}(\h_2,\h'_2) |^4 \gg \eta^{O(1)}$$
From this point onward we cease to keep careful track of powers of $\eta$.  On the other hand, by using two applications of Lemma \ref{cauchy-schwarz} to eliminate the $1$-bounded functions $f$, we have 
\begin{align*}
& |\E f(n_0+n_1+\h_2) \overline{f}(n_0+n_1+\h'_2) G_{n_0,n_1}(\h_2,\h'_2) |^4 \\
&\quad \leq \E G_{n_0,n_1}(\h_2,\h'_2) \overline{G_{n_0,n_1}}(\h_2,\k'_2)  \overline{G_{n_0,n_1}}(\k_2,\h'_2)  G_{n_0,n_1}(\k_2,\k'_2)
\end{align*}
where $(\k_2,\k'_2)$ is an independent copy of $(\h_2,\h'_2)$.  We thus have
$$
\E G_{n_0,\n_1}(\h_2,\h'_2) \overline{G_{n_0,\n_1}}(\h_2,\k'_2)  \overline{G_{n_0,\n_1}}(\k_2,\h'_2)  G_{n_0,\n_1}(\k_2,\k'_2) 
\gg \eta^{O(1)}
$$
which by the triangle inequality and \eqref{gn-def} gives
\begin{align*}
&\sum_{h_2,k_2,h'_2,k'_2} 1_{h_2-h'_2, k_2-k'_2, k_2-h'_2, h_2-k'_2 \in \Omega} \P( \h_2 = h_2; \k_2 = k_2; \h'_2 = h'_2; \k'_2 = k'_2)\\
&\quad | \E e_p( - (\xi(h_2-h'_2)  + \xi(k_2-k'_2) - \xi(k_2-h'_2) - \xi(h_2-k'_2) ) \n_1 ) | \\
&\quad\quad \gg \eta^{O(1)}.
\end{align*}
By Lemma \ref{popular}, we conclude that with probability $\gg \eta^{O(1)}$, the tuple $(\h_2,\k_2,$ $\h'_2,\k'_2)$ attains a value $(h_2,k_2,h'_2,k'_2)$ for which
$$ h_2-h'_2, k_2-k'_2, h_2-k'_2, k_2-h'_2 \in \Omega$$
and
\begin{equation}\label{stone}
 | \E e_p( - (\xi(h_2-h'_2)  + \xi(k_2-k'_2) - \xi(k_2-h'_2) - \xi(h_2-k'_2)) \n_1 ) | \gg \eta^{O(1)} \gg K^{-O(1)}
\end{equation}
thanks to \eqref{K-def-2}.  Since $(h_2-h'_2, k_2-k'_2, h_2-k'_2, k_2-h'_2)$ is an additive quadruple, the claim now follows from Lemma \ref{fde}, \eqref{rho-sep}, and \eqref{K-def-2}.
\end{proof}

We localise this claim slightly, though for notational reasons we will not move from $\rho_2$ immediately to $\rho_3$ and beyond, but instead first work in some intermediate scales between $\rho_2$ and $\rho_3$.  For any natural number $j$, define
$$ \rho_{2,j} \coloneqq \exp(-C_1 jK) \rho_2,$$
thus 
$$\rho_2 = \rho_{2,0} > \rho_{2,1} > \dots > \rho_{2,j} \geq \rho_3$$
if (say) $j \leq K^{C_1^2}$.

It will be necessary to break the symmetry between the four components of an additive quadruple, by restricting the second component to a tiny Bohr set, the third component to a larger Bohr set, and the first and fourth components to an even larger Bohr set.  More precisely, given an additive quadruple $\vec a_0 = (a_{(1),0},a_{(2),0},a_{(3),0},a_{(4),0}) \in \operatorname{Q}$, a subset $S' \subset \Z/p\Z$, and radii $0 < r_2 \leq r_3 \leq r_4 \leq 1/2$, we say that a random additive quadruple $\vec \a = (\a_{(1)}, \a_{(2)}, \a_{(3)}, \a_{(4)}) \in \operatorname{Q}$ is \emph{centred at $\vec a_0$ with frequencies $S'$ and scales $r_2,r_3,r_4$} if $\a_{(2)}, \a_{(3)}, \a_{(4)}$ are drawn independently and regularly from $a_{(2),0} + B(S', r_2)$, $a_{(2),0} + B(S', r_2)$, and $a_{(2),0} + B(S', r_2)$ respectively.  Note that this property also describes the distribution of $\a_{(1)}$, since we have the constraint
$$ \a_{(1)} = \a_{(3)} + \a_{(4)} - \a_{(2)}.$$
In practice, $r_4$ will be much larger than $r_2,r_3$, so (by Lemma \ref{ati}) $\a_{(1)}$ will be approximately regularly drawn from $a_{(1),0} + B(S', r_4)$, but will be highly coupled to the other three components of the quadruple (in particular, it will stay close to $\a_{(4)}$).  We thus see that for $i=1,2,3,4$, each $\a_{(i)}$ is either exactly or approximately drawn regularly from $a_{(i),0} + B(S', r_{l_i})$, where $l_i \in \{0,1,2\}$ is the quantity defined by the formulae
\begin{equation}\label{li}
 l_1 \coloneqq 0; \quad l_2 \coloneqq 2; \quad l_3 \coloneqq 1; \quad l_4 \coloneqq 0.
\end{equation}

\begin{corollary}\label{sss}   Let the notation and hypotheses be as in Theorem \ref{locu3}, and let $\Omega$ and $\xi$ be as in Theorem \ref{first-step}.  Then there exists a random additive quadruple $\vec \a \in \operatorname{Q}$ centred at some quadruple $\vec a_0 \in \operatorname{Q}$ with frequencies $S$ and scales $\rho_{2,2}, \rho_{2,1}, \rho_{2,0}$, such that $\vec \a \in \Omega^4 \cap (\operatorname{Q} \backslash \BQ)$ with probability $\gg \eta^{O(1)}$.
\end{corollary}

\begin{proof}  Let $\h_2, \k_2, \h'_2, \k'_2, \n_{2,1}, \n_{2,2}$ be drawn independently and regularly from $B(S,\rho_{2,0})$, $B(S,\rho_{2,0})$, $B(S,\rho_{2,0})$, $B(S,\rho_{2,0})$, $B(S,\rho_{2,1})$ and $B(S,\rho_{2,2})$ respectively.  From Theorem \ref{second-step}, we have
$$ (\h_{2}-\h'_{2}, \k_{2}-\k'_{2}, \h_{2} - \k'_{2}, \k_{2}-\h'_{2}) \in \Omega^4 \cap (\operatorname{Q} \backslash \BQ)$$
with probability $\gg \eta^{O(1)}$.  Using Lemma \ref{ati}, we may replace $\k'_2$ by $\k'_2 - \n_{2,2}$, and similarly replace $\h_2$ by $\h_2 + \n_{2,1} - \n_{2,2}$, to conclude that
$$ (\h_{2}-\h'_{2} + \n_{2,1}, \k_{2}-\k'_{2} + \n_{2,2}, \h_{2} - \k'_{2} + \n_{2,1}, \k_{2}-\h'_{2}) \in \Omega^4 \cap (\operatorname{Q} \backslash \BQ)$$
with probability $\gg \eta^{O(1)}$.  By the pigeonhole principle, we may thus find $k_2, k'_2, h_2 \in \Z/p\Z$ such that
$$ (h_{2}-\h'_{2} + \n_{2,1}, k_{2}-k'_{2} + \n_{2,2}, h_{2} - k'_{2} + \n_{2,1}, k_{2}-\h'_{2}) \in \Omega^4 \cap (\operatorname{Q} \backslash \BQ)$$
with probability $\gg \eta^{O(1)}$.  The left-hand side is an additive quadruple centred at $(h_2, k_2-k'_2, h_2-k'_2, k_2)$ with frequencies $S$ and scales $\rho_{2,2}, \rho_{2,1}, \rho_{2,0}$, and the claim follows.
\end{proof}

\subsection{Third step: $\xi$ is approximately linear $99\%$ of the time on a rough set}

The next general step in the standard inverse $U^3$ argument is to upgrade this weak additive structure, which is of a ``1 percent'' nature, to a more robust ``99 percent'' additive structure .  There are two basic ways to proceed here.  The first way is to invoke the Balog-Szemer\'edi-Gowers theorem \cite{balog,gowers-4-aps}, followed by standard sum set estimates including Freiman's theorem (see e.g. \cite[Chapter 2]{taovu-book}).  It is likely that this approach will eventually work here, but these results need to be localised efficiently to Bohr sets, and also to allow for the fact that $\xi(a_{(1)})+\xi(a_{(2)})-\xi(a_{(3)})-\xi(a_{(4)})$ no longer vanishes, but instead has controlled word norm.  This would require reworking of large portions of the standard additive combinatorics literature.  We have thus elected instead to follow the second approach, also due to Gowers \cite{gowers-long-aps}, in which a certain probabilistic argument is used to ``purify'' a 1 percent additive map to a 99 percent additive map, albeit on a set which has no particular structure itself.  To deal with this set we will use a more recent innovation, namely a variant\footnote{The actual arithmetic regularity lemma, which creates arithmetic regularity on almost all regions of space, has quantitative bounds of tower-exponential type, which are far too poor for our application; however we will only need to create a single neighbourhood in which arithmetic regularity exists, and this can be done with much more efficient quantitative bounds.} of the arithmetic regularity lemma \cite{green-reg}, \cite{gt-reg} to make the subsets of $\Z/p\Z$ on which one has good control of $\xi$ suitably ``pseudorandom'' in the sense of Gowers.

We turn to the details.  We first locate a reasonably large quadruple of sets $A_{(1)}, A_{(2)}, A_{(3)}, A_{(4)}$ on which $\xi$ is ``almost a Freiman homomorphism'' in the sense that most quadruples falling inside $A_{(1)} \times A_{(2)} \times A_{(3)} \times A_{(4)}$ are somewhat good.  
We call an additive quadruple $(a_{(1)},a_{(2)},a_{(3)},a_{(4)}) \in \operatorname{Q}$ \emph{very bad} if 
\begin{equation}\label{bad-quad-2}
 \| \xi(a_{(1)}) + \xi(a_{(2)}) - \xi(a_{(3)}) - \xi(a_{(4)}) \|_{S} > \frac{1}{\rho_3},
\end{equation}
and let $\VBQ \subset \BQ$ denote the space of all very bad additive quadruples.

\begin{theorem}\label{third-step} Let the notation and hypotheses be as in Theorem \ref{locu3}, and let $\Omega$ and $\xi$ be as in Theorem \ref{first-step}.  Let $\vec a$ be the random additive quadruple from  Corollary \ref{sss}.  Then there exist sets $A_{(1)}, A_{(2)}, A_{(3)}, A_{(4)} \subset \Omega$ such that
\begin{equation}\label{ewa}
\E W( \vec a ) \gg \eta^{C_1 + O(1)},
\end{equation}
where $W: \mathrm{Q} \to \R$ is the weight function
\begin{equation}\label{W-def}
 W( \vec a ) \coloneqq 1_{A_{(1)} \times A_{(2)} \times A_{(3)} \times A_{(4)}}(\vec \a) \left(1 - \eta^{-C_1/100} 1_{\VBQ}(\vec \a) \right).
\end{equation}
\end{theorem}

The idea here is that $W$ is a weight function that strongly penalises very bad quadruples, and so Theorem \ref{third-step} is asserting that ``most'' of the quadruples in $A_{(1)} \times A_{(2)} \times A_{(3)} \times A_{(4)}$ are not very bad.

\begin{proof}  We will construct the sets $A_{(i)}$ by the probabilistic method, adapting an argument from \cite{gowers-long-aps} in which the $A_{(i)}$ are created by applying a number of random linear ``filters'' to the graph of $\xi$ to eliminate most of the additive quadruples that are not (almost) preserved by $\xi$.  

We turn to the details.  Let $m$ be the integer
\begin{equation}\label{mdef}
 m \coloneqq \left\lfloor \frac{\log \eta^{C_1}}{3 \log 100} \right\rfloor.
\end{equation}
We then select jointly independent random variables $\h_j \in \Z/p\Z$ and $\lambdab_j \in \Z/p\Z$ for each for $j=1,\dots,m$, by selecting each $h_j$ regularly from $B(S,\rho_2)$, and selecting $\lambdab_j$ uniformly at random from $\Z/p\Z$; we also choose these random variables to be independent  of $\vec \a$.  For $j=1,\dots,m$, we then let $\Xib_j: \Z/p\Z \to \R/\Z$ be the random map
\begin{equation}\label{ximp}
 \Xib_j(n) \coloneqq \xi(n) \h_j + \frac{\lambdab_j n}{p}
\end{equation}
and then define the random sets
$$ \A_{(i)} \coloneqq \bigcap_{j=1}^m \A_{(i),j}$$
for $i=1,2,3,4$, where
$$ \A_{(1),j}=\A_{(2),j}=\A_{(3),j} \coloneqq \left\{ n \in \Omega: \| \Xib_j(n) \|_{\R/\Z} \leq \frac{1}{200} \right\}$$
and
$$ \A_{(4),j} \coloneqq \left\{ n \in \Omega: \| \Xib_j(n) \|_{\R/\Z} \leq\frac{1}{10} \right\}.$$

We will show that
\begin{equation}\label{bonus}
\E 1_{A_{(1)} \times A_{(2)} \times A_{(3)} \times A_{(4)}}(\vec \a) \gg \eta^{O(1)} 100^{-3m}
\end{equation}
and
\begin{equation}\label{minus}
\E 1_{A_{(1)} \times A_{(2)} \times A_{(3)} \times A_{(4)}}(\vec \a) 1_{\BQ}(\vec a) \ll 2^{-m} \times 100^{-3m}
\end{equation}
which will give the claim thanks to \eqref{mdef} and \eqref{W-def}, if $C_1$ is large enough.

We first show \eqref{bonus}.  By Corollary \ref{sss} and linearity of expectation, it suffices to show that
\begin{equation}\label{lama}
\P(a_{(i)} \in \A_{(i)} \hbox{ for } i=1,2,3,4 ) \gg 100^{-3m}
\end{equation}
whenever $(a_{(1)},a_{(2)},a_{(3)},a_{(4)})$ lies in $\Omega^4 \cap (\operatorname{Q} \backslash \BQ)$.
Actually, we will only show the weaker assertion that \eqref{lama} holds for all but at most $O( m^{O(1)} p^2 )$ of the available additive quadruples $(a_{(1)},a_{(2)},a_{(3)},a_{(4)})$; this still suffices, since by \eqref{theta-crude}, \eqref{K-def-2} each exceptional additive quadruple is attained with probability $O( \frac{1}{\rho_3^{O(K)} p^3})$, and the additional factor of $p$ will dominate all the losses in $m, K, \rho_3$ thanks to \eqref{cstar2}, \eqref{mdef}.

Fix an additive quadruple $\vec a = (a_{(1)},a_{(2)},a_{(3)},a_{(4)})$ in $\Omega^4 \cap (\operatorname{Q} \backslash \BQ)$.  The left-hand side of \eqref{lama} factors as
\begin{equation}\label{exp}
 \prod_{j=1}^m \P(a_{(i)} \in \A_{(i)} \hbox{ for } i=1,2,3,4 ) 
\end{equation}
so it will suffice to show that for each $j=1,\dots,m$, one has
$$ \P (a_{(i)} \in \A_{(i),j} \hbox{ for } i=1,2,3,4 ) \geq 100^{-3} - O\left(\frac{1}{m} \right)$$
for all but $O( m^{O(1)} p^2 )$ quadruples $(a_{(1)},a_{(2)},a_{(3)},a_{(4)}) \in \operatorname{Q} \backslash \BQ$.  Note however that from \eqref{ximp} we have
\begin{align*} \Xib_j(a_{(1)}) + \Xib_j(a_{(2)}) - & \Xib_j(a_{(3)}) - \Xib_j(a_{(4)}) \\ & =
\left(\xi(a_{(1)}) + \xi(a_{(2)}) - \xi(a_{(3)}) - \xi(a_{(4)})\right) \h_j\end{align*}
and hence by the hypothesis $(a_{(1)},a_{(2)},a_{(3)},a_{(4)}) \in \operatorname{Q} \backslash \BQ$ and the range of $\h_j$ we have
$$ \left\| \frac{\Xib_j(a_{(1)}) + \Xib_j(a_{(2)}) - \Xib_j(a_{(3)}) - \Xib_j(a_{(4)})}{p}\right \|_{\R/\Z} \leq \frac{1}{100}$$
(say).  In particular, we see from the triangle inequality that the claim $a_{(4)} \in \A_{(4),j}$ is implied by the claims $a_{(i)} \in \A_{(i),j}$ for $i=1,2,3$.  Thus it suffices to show that
$$ \P (a_{(i)} \in \A_{(i),j} \hbox{ for } i = 1,2,3 ) \geq 100^{-3} - O\left(\frac{1}{m} \right)$$
for all but $O( m^{O(1)} p^2 )$ triples $(a_{(1)}, a_{(2)}, a_{(3)}) \in (\Z/p\Z)^3$, noting that $a_{(4)}$ is determined by $a_{(1)}, a_{(2)}, a_{(3)}$.  We can write the left-hand side as
$$ \P\left( \frac{(\xi(a_{(1)}), \xi(a_{(2)}), \xi(a_{(3)})) \h_j + (a_{(1)},a_{(2)},a_{(3)}) \lambdab_j}{p} \in [-1/200,1/200]^3\right),$$
where we view the interval $[-1/200,1/200]$ as a subset of $\R/\Z$.  Thus it will suffice to show the equidistribution property
$$ \inf_{x \in (\R/\Z)^3} \P\left( \frac{(a_{(1)},a_{(2)},a_{(3)}) \lambdab_j}{p} \in x + [-1/200,1/200]^3 \right) \geq 100^{-3} - O\left(\frac{1}{m} \right).$$
Let $\psi: (\R/\Z)^3 \to [0,1]$ be a Lipschitz cutoff supported on $[-1/20,1/20]^3$ that equals one on $[-1/200+1/m, 1/200-1/m]^3$ and has Lipschitz constant $O(m)$.  Then we may lower bound the left-hand side by
\begin{equation}\label{joe}
 \inf_{x \in (\R/\Z)^3} {\mathbb E}_{\lambda \in \Z/p\Z} \psi\left( \frac{(a_{(1)},a_{(2)},a_{(3)}) \lambda}{p} - x \right).
\end{equation}
By standard Fourier expansion (see e.g. \cite[Lemma A.9]{gt-mobius}), we may write
$$ \psi(y) = \sum_{k \in \Z^3: k = O(m^{O(1)})} c_k e(k \cdot y) + O\left(\frac{1}{m} \right)$$
for all $y \in (\R/\Z)^3$ and some bounded Fourier coefficients $c_k = O(1)$; integrating in $x$, we see in particular that $c_0 = 10^{-3} + O\left(\frac{1}{m} \right)$.  We may thus write \eqref{joe} as
$$ 10^{-3} + O\left(\frac{1}{m} \right)+ O\left( \sum_{k \in \Z^3 \backslash \{0\}: k = O(m^{O(1)})} \left|{\mathbb E}_{\lambda \in \Z/p\Z} e_p( k \cdot (a_{(1)},a_{(2)},a_{(3)}) \lambda)\right| \right)$$
which gives the desired claim as long as there are no relations of the form
$$ k \cdot (a_{(1)},a_{(2)},a_{(3)})  = 0$$
for some non-zero $k \in \Z^3$ with $k = O(m^{O(1)})$.  But it is easy to see that the number of $(a_{(1)},a_{(2)},a_{(3)})$ with such a relation is $O( m^{O(1)} p^2 )$, thus concluding the proof of \eqref{bonus}.

Now we show \eqref{minus}.  By linearity of expectation as before, it suffices to show that
$$ \P (a_{(i)} \in \A_{(i)} \hbox{ for } i=1,2,3,4 ) \ll 2^{-m} \times 100^{-3m}$$
for all but $O( m^{O(1)} p^2)$ of the quadruples $(a_{(1)},a_{(2)},a_{(3)},a_{(4)})$ in $\VBQ$.  
Using the factorisation \eqref{exp}, it suffices to show that for each $j=1,\dots,m$, one has
$$ \P (a_{(i)} \in \A_{(i),j}\hbox{ for } i=1,2,3,4 ) \leq 2^{-1} \times 100^{-3} + O\left(\frac{1}{m} \right)$$
for all but $O( m^{O(1)} p^2)$ of the quadruples $(a_{(1)},a_{(2)},a_{(3)},a_{(4)})$ in $\VBQ$.  

The left-hand side may be written as
$$ \P\left( \frac{(\xi(a_{(1)}), \dots, \xi(a_{(4)})) \h_j}{p} + \vec a \lambdab_j \in [-1/200,1/200]^3 \times [-1/10,1/10] \right),$$
which we bound above by
\begin{align*}
\P\bigg( (\xi(a_{(1)}), \xi(a_{(2)}), \xi(a_{(3)})) \h_j + (a_{(1)},a_{(2)},a_{(3)}) \lambdab_j \in & [-1/200,1/200]^3, \\ & \left\| \frac{\sigma \h_j}{p} \right\|_{\R/\Z} \leq \frac{1}{8} \bigg),\end{align*}
where $\sigma \coloneqq  \xi(a_{(1)}) + \xi(a_{(2)}) - \xi(a_{(3)}) - \xi(a_{(4)})$.  By arguing as in the proof of \eqref{bonus}, we see that after deleting $O(m^{O(1)} p^2)$ exceptional tuples, one has
$$ \sup_{x \in (\R/\Z)^3} \P( (a_{(1)},a_{(2)},a_{(3)}) \lambdab_j \in x + [-1/200,1/200]^3 ) \leq 100^{-3} + O\left(\frac{1}{m} \right),$$
so by Fubini's theorem and the independence of $\h_j$ and $\lambdab_j$ it will suffice to show that
$$ \P\left( \left\| \frac{\sigma \h_j}{p} \right\|_{\R/\Z} \leq 1/8 \right) \leq 2^{-1} + O\left(\frac{1}{m} \right).$$ 
However, by Lemma \ref{edual} and the hypothesis $(a_{(1)},a_{(2)},a_{(3)},a_{(4)}) \in \VBQ$ we may find $h \in\Z/p\Z$ such that
$$
 \left\| \frac{\sigma h}{p} \right\|_{\R/\Z} > K^{-O(1)} {\| h \|_{S^\perp}}{\rho_3}.$$
In particular, $h$ is non-zero.  By repeatedly doubling $h$ until $\left\| \frac{\eta h}{p} \right\|_{\R/\Z}$ exceeds $1/4$, we may also assume that
$$ 1/2 \geq \left\|\frac{\eta h}{p}  \right\|_{\R/\Z} > 1/4$$
and thus
$$ \|h\|_{S^\perp} \ll K^{O(1)} \rho_3.$$
From Lemma \ref{ati} we conclude that
$$ \P\left( \left\| \frac{\eta (\h_j+h)}{p} \right\|_{\R/\Z} \leq 1/8 \right) = \P\left( \left\| \frac{\eta \h_j}{p} \right\|_{\R/\Z} \leq 1/8 \right) + O\left(\frac{1}{m}\right).$$ 
But from the triangle inequality we see that the events $\left\| \frac{\eta (\h_j+h)}{p} \right\|_{\R/\Z} \leq 1/8$, $\left\| \frac{\eta \h_j}{p} \right\|_{\R/\Z} \leq 1/8$ are disjoint.  The claim follows.
\end{proof}

\subsection{Fourth step: the rough set is pseudorandom in a Bohr set}

The sets $A_{(i)}$ provided by Theorem \ref{third-step} are currently rather arbitrary. In particular we have no control on the pseudorandomness of these sets (as measured by local Gowers $U^2$ norms) in the Bohr sets we are working with.  However, it is possible to use an ``energy decrement argument'' to pass to smaller\footnote{This is somewhat analogous to the variants of the Szemer\'edi regularity lemma \cite{szem:reg} in which one locates a single regular pair inside an arbitrary large random graph.  In contrast to the full regularity lemma which strives to ensure that \emph{almost all} pairs are regular, the ``one regular pair'' versions of the lemma enjoy significantly better quantitative bounds.  In our current application, such good quantitative bounds are essential, so we cannot appeal to analogues of the regularity lemma such as the arithmetic regularity lemma of the first author \cite{green-reg}.} Bohr sets in which the sets $A_{(i)}$ do enjoy good pseudorandomness properties, basically by converting any large Fourier coefficient of any of the $A_{(i)}$ in a Bohr set into a refinement of the Bohr sets (which add the frequency of the large Fourier coefficient to the frequency set $S$) on which the indicator function $1_{A_{(i)}}$
has smaller variance.  Furthermore, it is possible to shrink the Bohr sets in this fashion without destroying the conclusion \eqref{ewa} of Theorem \ref{third-step}.  

Here is a precise statement.

\begin{theorem}\label{fourth-step}  Let the notation and hypotheses be as in Theorem \ref{locu3}, and let $\Omega$ and $\xi$ be as in Theorem \ref{first-step}.  Let $A_{(1)}, A_{(2)}, A_{(3)}, A_{(4)}, W$ be as in Theorem \ref{third-step}.  Then there exists a natural number $j$, $j \leq \eta^{-10^3 C_1}$, an additive quadruple $\vec a_1 = (a_{(1),1}, a_{(2),1}, a_{(3),1}, a_{(4),1}) \in \operatorname{Q}$, and a set $S_1$, $S \subset S_1 \subset \Z/p\Z$ with $|S_1| \leq |S| + j$, with the following properties:
\begin{itemize}
\item[(i)] \textup{(Few very bad quadruples)} We have
\begin{equation}\label{stor}
 \E W(\vec \a)\gg \eta^{C_1 + O(1)},
\end{equation}
where $\vec \a$ is a random additive quadruple centred at $\vec a_1$ with frequencies $S_1$ and scales $\rho_{2,j+2}$, $\rho_{2,j+1}$, and $\rho_{2,j}$.
\item[(ii)] \textup{(Local Fourier pseudorandomness)} For each $i=1,2,3,4$, we have 
$$
|\E f_i(\a_{(i)}+\h_0+\h_1) f_i(\a_{(i)}+\h_0+\h'_1) f_i(\a_{(i)}+\h'_0+\h_1) f_i(\a_{(i)}+\h'_0+\h'_1)| \leq \eta^{100 C_1}$$
where $f_i: \Z/p\Z \to [-1,1]$ denotes the balanced function
\begin{equation}\label{f1-def}
 f_i(a_{(i)}) \coloneqq  1_{A_{(i)}}(a_{(i)}) - \alpha_i,
\end{equation}
$\alpha_i$ denotes the mean
\begin{equation}\label{alphai-def}
\alpha_i \coloneqq   \E 1_{A_{(i)}}(\a_{(i)}),
\end{equation}
and where $\a_{(i)}$ and $\h_0, \h'_0, \h_1, \h'_1$ are drawn independently and regularly from the Bohr sets $a_{(i),1} + B( S_1, \rho_{2,j+l_i} )$ and $B(S_1, \rho_{2,j+10} )$, $B(S_1, \rho_{2,j+10} )$, $B(S_1, \rho_{2,j+11} )$, $B(S_1, \rho_{2,j+11} )$ respectively, with the quantity $l_i$ given by \eqref{li}.
\end{itemize}
\end{theorem}

\begin{proof}  We will formulate the ``energy decrement'' argument here as a ``score maximisation'' argument.
Define a \emph{$4$-neighbourhood} to be a tuple
$$ N = ( \vec a_1, j, S_1 )$$
where $\vec a_1 \in \mathrm{Q}$ is an additive quadruple, $j$ is a natural number between $0$ and $\eta^{-10^3 C_1}$, and $S_1$ is a subset of $\Z/p\Z$ containing $S$ with $|S_1| \leq |S|+j$; we refer to $j$ as the \emph{depth} of the $4$-neighbourhood $N$.  Given such a neighbourhood, we define the \emph{score} $\operatorname{Score}(N)$ of the $4$-neighbourhood to be the quantity
\begin{equation}\label{rs-def}
\operatorname{Score}(N) \coloneqq 
\E W(\vec \a)  - \eta^{2C_1} \sum_{i=1}^4 \mathrm{E}_i(N) - \eta^{10^3 C_1} j
\end{equation}
where $\vec \a = (\a_{(1)}, \a_{(2)}, \a_{(3)}, \a_{(4)})$ is a random additive quadruple centred at $\vec a_1$ with frequencies $S_1$ and scales $\rho_{2,j+2}, \rho_{2,j+1}, \rho_{2,j}$, and $\operatorname{E}_i$ is the energy-type quantity 
\begin{equation}\label{energy-def}
 \operatorname{E}_i(N) \coloneqq \Var\ 1_{A_{(i)}}(\a_{(i)}).
\end{equation}

If we define $N_0$ to be the $4$-neighbourhood
$$ N_0 \coloneqq  (\vec a_0, 0, S),$$
then Theorem \ref{third-step} tells us that
\begin{equation}\label{stam}
 \operatorname{Score}(N_0) \gg \eta^{C_1 + O(1)}.
\end{equation}

We choose
$$ N \coloneqq  (\vec a_1, j, S_1)$$
to be a $4$-neighbourhood that comes within $\eta^{10^3C_1}$ (say) of maximising the adjusted score.  Then we must have
$$ \operatorname{Score}(N) \geq \operatorname{Score}(N_0) - \eta^{10^3C_1} \gg \eta^{C_1 + O(1)}$$
which from \eqref{rs-def} implies the bound \eqref{stor}, as well as the bound
$$ j \leq \eta^{-10^3 C_1} - 10^3$$
(say).  It will then suffice to show that property (ii) of the theorem holds.

It remains to show (ii).  Let $i=1,2,3,4$, and write
$$ \vec a_1 = (a_{(1),1}, a_{(2),1}, a_{(3),1}, a_{(4),1}).$$
Suppose for contradiction that
\begin{equation}\label{afg}
|\E f_i(\a_{(i)}+\h_0+\h_1) f_i(\a_{(i)}+\h_0+\h'_1) f_i(\a_{(i)}+\h'_0+\h_1) f_i(\a_{(i)}+\h'_0+\h'_1)| > \eta^{100 C_1}
\end{equation}
where $f_i$ is given by \eqref{f1-def}, and $\a_{(i)}, \h_0, \h'_0, \h_1, \h'_1$ are drawn independently and regularly from the Bohr sets $a_{(i),1} + B( S_1, \rho_{2,j+l_i} )$, $B(S_1, \rho_{2,j+10} )$, $B(S_1,$ $\rho_{2,j+10} )$, $B(S_1, \rho_{2,j+11} )$, $B(S_1, \rho_{2,j+11} )$, with $l_i$ given by \eqref{li}.
 
We will use \eqref{afg} to construct a random $4$-neighbourhood $\Nb$ of depth $j+20$ obeying the estimates
\begin{equation}\label{letitgo}
\E W(\Nb) = W(N) + O( \eta^{10^3 C_1} )
\end{equation}
and
\begin{equation}\label{estable}
\E \operatorname{E}_{i'}(\Nb) \leq \operatorname{E}_{i'}(N) - \eta^{500 C_1} 1_{i=i'} + O( \eta^{10^3 C_1} )
\end{equation}
for $i' = 1,2,3,4$.  If we have the estimates \eqref{letitgo}, \eqref{estable}, we conclude from \eqref{rs-def} and linearity of expectation that
$$ \E \operatorname{Score}(\Nb) > \operatorname{Score}(N) + \eta^{600C_1},$$
contradicting the near-maximality of $\operatorname{Score}(N)$.

It remains to construct $\Nb$ obeying \eqref{letitgo}, \eqref{estable}.  We begin by noting that for each $a_{(i)} \in \Z/p\Z$, the Gowers uniformity-type quantity
$$ \E  f_i(a_{(i)}+\h_0+\h_1) f_i(a_{(i)}+\h_0+\h'_1) f_i(a_{(i)}+\h'_0+\h_1) f_i(a_{(i)}+\h'_0+\h'_1)
$$
can be factored as
$$ \sum_{h_0, h'_0} \P( \h_0 = h_0, \h'_0 = h'_0 ) \left| \E f_i(a_{(i)}+h_0+\h_1) f_i(a_{(i)}+h'_0+\h_1) \right|^2$$
and thus takes values between $0$ and $1$. By \eqref{afg} and Lemma \ref{popular}, we may thus find a set $E \subset \Z/p\Z$ with
$$
\P(\a_{(i)} \in E) \gg \eta^{100 C_1}$$
such that
$$
\E f_i(a_{(i)}+\h_0+\h_1) f_i(a_{(i)}+\h_0+\h'_1) f_i(a_{(i)}+\h'_0+\h_1) f_i(a_{(i)}+\h'_0+\h'_1) \gg \eta^{100 C_1}$$
for all $a_{(i)} \in E$.  Applying Theorem \ref{locu2}, we may thus find, for each $a_{(i)} \in E$, a frequency $\xi(a_{(i)}) \in \Z/p\Z$ such that
$$ \sum_{n_0} \P(\n_0=n_0) \E\left| \E f_i(a_{(i)}+n_0+\n_1) e_p( -\xi(a_{(i)}) \n_1 )\right|^2 \gg \eta^{100 C_1},$$
where $\n_0,\n_1$ are drawn independently and regularly from $B(S_1, \rho_{2,j_*+10})$ and $B(S_1, \rho_{2,j_*+11})$ respectively, independently of the $\a_{(i)}$.

If we define $\xi(a_{(i)})$ arbitrarily for $a_{(i)} \not \in E$ (e.g. setting $\xi(a_{(i)})=0$), we thus have
\begin{align*}
\sum_{n_0, a_{(i)}} \P(\n_0=n_0, \a_{(i)} = a_{(i)}) \E \big| \E( f_i(a_{(i)}+n_0+\n_1) & e_p( -\xi(a_{(i)}) \n_1 ) ) \big|^2 \\ & \gg \eta^{200 C_1}.
\end{align*}
In particular, there exists a $1$-bounded function $g: \Z/p\Z \times \Z/p\Z \to \C$ such that
\begin{equation}\label{abang}
\left| \E  g(\n_0, \a_{(i)}) f_i(\a_{(i)}+\n_0+\n_1) e_p( -\xi(\a_{(i)}) \n_1 ) \right| \gg \eta^{200 C_1}.
\end{equation}

We now construct the random $4$-neighbourhood $\Nb$ as follows.  We first construct a random additive quadruple $\vec \k = (\k_1,\k_2,\k_3,\k_4)$ centred at the origin $(0,0,0,0)$ with frequency set $S_1$ and scales $\rho_{2,j+10+l_2-l_i}$, $\rho_{2,j+10+l_3-l_i}$, $\rho_{2,j+10+l_4-l_i}$, and independent of all previous random variables.  We then set
$$ \Nb \coloneqq (\vec \a + \vec \k, j+20, S_1 \cup \{\xi(\a_{(i)})\}).$$
It is easy to verify that $\Nb$ is a (random) $4$-neighbourhood.

We now verify \eqref{letitgo}.  The left-hand side of \eqref{letitgo} can be expanded as
$$
\E W(\vec \a +\vec \k + \vec \h)$$
where, once $\vec \a$ and $\vec \k$ are chosen, the random additive quadruple $\vec \h = (\h_1,\h_2,$ $\h_3,\h_4)$ is selected to be centred at $(0,0,0,0)$ with frequencies $S_1 \cup \{\xi(\a_{(i)})\}$ and scales $\rho_{2,j+22}, \rho_{2,j+21}, \rho_{2,j+20}$.

From two applications of Lemma \ref{ati} (and the fact that $W = O(\eta^{-C_1/100})$), we have
$$ \E W(\vec \a +\vec \k + \vec \h) = \E W(\vec \a +\vec \k ) + O( \eta^{10^3 C_1}) = \E W(\vec \a ) + O( \eta^{10^3 C_1})$$
(say).  The claim \eqref{letitgo} now follows from \eqref{rs-def}.

Now we verify \eqref{estable}.  By \eqref{energy-def}, we have
$$
 \operatorname{E}_{i'}(\Nb) = \sum_{\vec a, \vec k} \P( \vec \a = \vec a, \vec \k = \vec k )  \E \left|1_{A_{(i')}}(a_{(i')} + k_{i'} + \h_{i'}) - \alpha_{i',\vec a, \vec k}\right|^2
$$
where $\vec a = ( a_{(1)}, \dots, a_{(4)})$, $\vec k = (k_1,\dots, k_4)$, and $\alpha_{i',\vec a, \vec k}$ is the quantity
\begin{equation}\label{epo}
 \alpha_{i',\vec a, \vec k} \coloneqq  \E 1_{A_{(i')}}(a_{(i')} + k_{i'} + \h_{i'}).
\end{equation}
By Pythagoras' theorem, we thus have
$$
 \operatorname{E}_{i'}(\Nb) = \sum_{\vec a, \vec k} \P( \vec \a = \vec a, \vec \k = \vec k) \E\left |1_{A_{(i')}}(a_{(i')} + k_{i'} + \h_{i'}) - \alpha_{i'}\right|^2 - |\alpha_{i',\vec a, \vec k} - \alpha_{i'}|^2
$$
where $\alpha_{i'}$ is defined in \eqref{alphai-def}.
We shall shortly establish the bound
\begin{equation}\label{alpha-triv}
|\alpha_{i', \vec a, \vec k} - \alpha_{i'}|^2 \gg \eta^{400 C_1} 1_{i'=i}.
\end{equation}
Assuming this bound, we conclude that
\begin{align*}
\E \operatorname{E}_{i'}(\Nb) &\leq \sum_{\vec a, \vec k} \P( \vec \a = \vec a, \vec \k = \vec k) \E |1_{A_{(i')}}(a_{(i')} + k_{i'} + \h_{i'}) - \alpha_{i'}|^2 \\
&= \E |1_{A_{(i')}}(\a_{(i')} + \k_{i'} + \h_{i'}) - \alpha_{i'}|^2 - \eta^{500 C_1} 1_{i'=i}. 
\end{align*}
By applying Lemma \ref{ati} twice as in the proof of \eqref{letitgo} to replace $\a_{(i')} + \k_{i'} + \h_{i'}$ by $\a_{(i')}$ for $i'=2,3,4$ (and by using Lemma \ref{ati} six times for $i'=1$, after writing $\a_{(1)}$ in terms of $\a_{(2)}, \a_{(3)}, \a_{(4)}$, and similarly for $\k_{(1)}$ and $\h_{(1)}$) we thus have
$$
\E \operatorname{E}_{i'}(\Nb) \leq \E |1_{A_{(i')}}(\a_{(i')}) - \alpha_{i'}|^2 - \eta^{500 C_1} 1_{i'=i} + O( \eta^{10^3 C_1} ).
$$
This will give \eqref{estable} as soon as we establish \eqref{alpha-triv}.  This is trivial for $i' \neq i$, so suppose that $i=i$.
By \eqref{epo} and \eqref{f1-def}, it suffices to show that
\begin{equation}\label{bb}
 \sum_{\vec a, \vec k} \P( \vec \a = \vec a, \vec \k = \vec k) \left| \E f_i(a_{(i)} + k_{i} + \h_{i}) \right|^2 \gg \eta^{400 C_1}. 
\end{equation}
To prove this, we introduce random variables $\n_0, \n_1$ drawn independently and regularly from
$B(S_1, \rho_{2,j+10})$ and $B(S_1, \rho_{2,j+11})$ independently of all previous variables.  From \eqref{abang} we have
$$ \left|\E f_i(\a_{(i)} + \n_0 + \n_1) g( \n_0, \a_{(i)}) e_p( - \xi(\a_{(i)}) \n_1)\right| \gg \eta^{200 C_1}$$
for some $1$-bounded function $g$. After using Lemma \ref{ati} to compare $\n_1$ and $\n_1 + \h_i$ for each fixed choice of $\n_0$ and $\a_{(i)}$, we conclude that
$$ \left|\E f_i(\a_{(i)} + \n_0 + \n_1 + \h_i) g(\n_0, \a_{(i)}) e_p( - \xi(\a_{(i)}) (\n_1+\h_i))\right| \gg \eta^{200 C_1}.$$
But we have
\[
\left\| \frac{\xi(\a_{(i)}) \h_i}{p} \right\|_{\R/\Z} \leq \| \h_i \|_{S_1 \cup \{\xi(\a_{(i)})\}} \ll \rho_{j+l_i+20}
\]
and hence by \eqref{bilipschitz}
$$ e_p( - \xi(\a_{(i)}) (\n_1+\h_i)) = e_p( - \xi(\a_{(i)}) \n_1) + O( \eta^{10^3 C_1} ).$$
We conclude that
$$ |\E( f_i(\a_{(i)} + \n_0 + \n_1 + \h_i) g( \n_0, \a_{(i)}) e_p( - \xi(\a_{(i)}) \n_1)| \gg \eta^{200 C_1}.$$
For fixed choices of $\a_{(i)}, \h_{(i)}, \n_1$, we see from Lemma \ref{ati} that $\k_i$ and $\n_0 + \n_1$ differ in total variation by $O( \eta^{10^3 C_1} )$.  Thus we have
$$ |\E( f_i(\a_{(i)} + \k_{i} + \h_{i}) g( \k_i - \n_1, a_{(i)}) e_p( - \xi(\a_{(i)}) \n_1)| \gg \eta^{200 C_1},$$
and the claim now follows after using Lemma \ref{cauchy-schwarz} to eliminate the $g( \k_i - \n_1, a_{(i)}) e_p( - \xi(\a_{(i)}) \n_1)$ factor.
\end{proof}

A useful consequence of the bounds in Theorem \ref{fourth-step}(ii) is the following weak mixing bound, which roughly speaking asserts that the convolution of $1_{A_{(i)}}$ with a bounded function is essentially constant.

\begin{lemma}\label{wmix}  
Let the notation and hypotheses be as above, and let $\Omega$ and $\xi$ be as in Theorem \ref{first-step}.  Let $A_{(1)}, \dots, A_{(4)}$ be as in Theorem \ref{third-step}, and let $j, a_{(1),*}, \dots, a_{(4),*}, S_1, f_1,\dots,f_4$ be as in Theorem \ref{fourth-step}.
Then for any $i=1,2,3,4$, any $l_i < m \leq 10$, and any $1$-bounded function $g: \Z/p\Z \to \C$, one has
 \begin{equation}\label{po}
\sum_n \P( \n=n) \left| \E f_i( n - \k) g(\k) \right|^2 \ll \eta^{50C_1}
\end{equation}
where $\n, \k$ are drawn independently and regularly from $a_{(i),*} + B(S_1, \rho_{2,j} )$ and $B(S_1, \rho_{2,j+m} )$ respectively.
Dually, for any $1$-bounded function $G: \Z/p\Z \to \C$, one has
\begin{equation}\label{op}
\sum_k \P(\k=k) \left| \E f_i( \n - k ) G( \n ) \right| \ll \eta^{25 C_1}.
\end{equation}
\end{lemma}

\begin{proof}
In preparation for invoking Theorem \ref{fourth-step}(ii), we introduce random variables $\h_0, \h_1, \h'_1$ drawn independently and regularly from $B(S_1, \rho_{2,j_*+10})$, $B(S_1, \rho_{2,j_*+11})$, and $B(S_1, \rho_{2,j_*+11})$ respectively, independently of $\n$ and $\k$.  Using Lemma \ref{ati} to compare $\n,\k$ with $\n+\h_0$, $\k-\h_1$ respectively, we may transform \eqref{po} to the estimate
$$ \sum_{n,h_0} \P( \n=n, \h_0 = h_0) \left|\E(f_i(n+h_0-\k-\h_1) g(\k-\h_1))\right|^2 \ll \eta^{50C_1}.$$
By the triangle inequality in $L^2$, it thus suffices to show that
\begin{equation}\label{square}
\sum_{n,h_0} \P(\n=n, \h_0=h_0) \left|\E(f_i(n+h_0-k-\h_1) g(k-\h_1) )\right|^2 \ll \eta^{50C_1}
\end{equation}
for all $k \in B(S_1, \rho_{2,j_*+m})$.

Fix $k$.  We may expand out the left-hand side of \eqref{square} as 
$$ \E f_i( \n + \h_0 - \h_1 - k ) g( k - \h_1 ) f_i( \n + \h_0 - \h'_1 - k ) g( k - \h'_1 ).$$
Using Lemma \ref{ati} to compare $\n$ with $\n+\h_0-\h_1-\h'_1-k$, we can thus rewrite \eqref{square} as
$$ |\E f_i( \n + \h_0 + \h'_1 ) g( k - \h_1 ) f_i( \n + \h_0 + \h_1 ) g( k - \h'_1 )| \ll \eta^{50C_1},$$
which by the triangle inequality and the $1$-boundedness of $g$ would follow from
$$ \sum_{n,h_1,h'_1} \P(\n=n, \h_1 = h_1, \h'_1 = h_1) | \E f_i( n + \h_0 + h'_1 )f_i( n + \h_0 + h_1 ) | \ll \eta^{50C_1},$$
which by Cauchy-Schwarz will follow in turn from
$$ \sum_{n,h_1,h'_1} \P(\n=n, \h_1 = h_1, \h'_1 = h_1) | \E f_i( n + \h_0 + h'_1 )f_i( n + \h_0 + h_1 ) |^2 \ll \eta^{100C_1}.$$
But this follows from Theorem \ref{fourth-step}(ii) (relabeling $\n$ as $\a_{(i)}$).

Finally, we show \eqref{op}.  By subtracting $\E G(\n)$ from $G$ (and dividing by $2$ to recover $1$-boundedness), we may assume that $\E G(\n)=0$.  It then suffices to show that
$$
\sum_k \P(\k=k) g(k) \E 1_{A_{(i)}}( \n - k ) G( \n )  \ll \eta^{25 C_1}.
$$
for any $1$-bounded function $g$.  But the left-hand side may be rearranged as
$$
\sum_n \P(\n=n) G(n) (\E 1_{A_{(i)}}( n - \k ) g( \k )  - \alpha_i \E g(\k)) \ll \eta^{25 C_1},
$$
and the claim follows from \eqref{po} and the Cauchy-Schwarz inequality.
\end{proof}

\subsection{Fifth step: a frequency function $\xi'$ that is approximately linear $99\%$ of the time on a Bohr neighbourhood}

The next step is to obtain additive structure on almost all of a Bohr neighbourhood, rather than just the subsets $A_{(i)}$.  

\begin{theorem}\label{fifth-step} Let the notation and hypotheses be as in Theorem \ref{locu3}, and let $\xi$ be as in Theorem \ref{first-step}.  Let $A_{(1)}, \dots, A_{(4)}$ be as in Theorem \ref{third-step}, and let $j, a_{(1),1}, a_{(2),1}, a_{(3),1}, a_{(4),1}, S_1, \alpha_1,\dots,\alpha_4$ be as in Theorem \ref{fourth-step}.  Let $a_1 \in \Z/p\Z$ be the quantity
$$a_1 \coloneqq  a_{(1),1} + a_{(2),1} = a_{(3),1} + a_{(4),1},$$
and let $\a$ and $\a_{(2)}$ be drawn regularly and independently from $a_1 + B( S_1, \rho_{2,j} )$ and $a_{(2),1} + B(S_1, \rho_{2,j+2} )$ respectively.  Then there is a function $\xi': \Z/p\Z \to \Z/p\Z$, such that with probability at least $1 - O(\eta^{C_1/200})$, the random variable $\a$ attains a value $a$ for which we have the estimates
\begin{equation}\label{no}
\E 1_{A_{(2)}}(\a_{(2)}) 1_{A_{(1)}}(a - \a_{(2)})  = \alpha_1 \alpha_2 + O( \eta^{20C_1}),
\end{equation}
and 
\begin{equation}\label{no-2}
\begin{split}
\P&\left( a - \a_{(2)} \in A_{(1)}; a_{(2)} \in A_{(2)}; \| \xi'(a) - \xi(a-\a_{(2)})-\xi(\a_{(2)}) \|_{S} > \frac{1}{\rho_3} \right)\\
&\quad \ll \eta^{C_1/200} \alpha_1 \alpha_2.
\end{split}
\end{equation}
\end{theorem}

\begin{proof}  Let $\a$ be drawn regularly from $a_1 + B(S_1,\rho_{2,j})$, and let $(\a_{(1)}, \a_{(2)}, \a_{(3)},$ $\a_{(4)})$ be a random additive quadruple centred at $(a_{(1),1}, a_{(2),1}, a_{(3),1},a_{(4),1})$ with frequencies $S_1$ and scales $\rho_{2,j+2}, \rho_{2,j+1}, \rho_{2,j}$, independently of $\a$.   From the definition of an additive quadruple, we have $\a_{(1)} = \a_{(3)} + \a_{(4)} - \a_{(2)}$.  
From Theorem \ref{fourth-step}(i) we thus have
\begin{equation}\label{stor-again}
 \E W( \a_{(3)} + \a_{(4)} - \a_{(2)}, \a_{(2)}, \a_{(3)}, \a_{(4)}) \gg \eta^{C_1 + O(1)}.
\end{equation}
From Lemma \ref{ati} we see that once we condition $\a_{(2)}$ and $\a_{(3)}$ to be fixed, $\a_{(4)}$ and $\a - \a_{(3)}$ differ in total variation by $O(\eta^{100C_1})$.  Thus we may replace $\a_{(4)}$ by $\a - \a_{(3)}$ in \eqref{stor-again} to conclude that
$$
 \E W( \a - \a_{(2)}, \a_{(2)}, \a_{(3)}, \a - \a_{(3)}) \gg \eta^{C_1 + O(1)}.
$$
If we then define
$$ \sigma \coloneqq  \E 1_{A_{(1)}}(\a-\a_{(2)}) 1_{A_{(2)}}(\a_{(2)}) 1_{A_{(3)}}(\a_{(3)}) 1_{A_{(4)}}(\a-\a_{(3)})$$
then from \eqref{W-def} we see that
\begin{equation}\label{step}
 \sigma \gg \eta^{C_1 +O(1)}
\end{equation}
and
\begin{equation}\label{dance}
\begin{split}
&\E 1_{A_{(1)}}(\a-\a_{(2)}) 1_{A_{(2)}}(\a_{(2)}) 1_{A_{(3)}}(\a_{(3)}) 1_{A_{(4)}}(\a-\a_{(3)})\\
&\quad  1_{\VBQ}( \a - \a_{(2)}, \a_{(2)}, \a_{(3)}, \a - \a_{(3)} ) \ll \eta^{-C_1/100} \sigma.
\end{split}
\end{equation}
	
We can express $\sigma$ in the form
\begin{equation}\label{si}
 \sigma = \E g_{12}(\a) g_{34}(\a)
\end{equation}
where $g_{12}, g_{34}: \Z/p/\Z \to \R$ are the functions
\begin{equation}\label{g12-def}
 g_{12}(a) \coloneqq  \E 1_{A_{(1)}}(a - \a_{(2)}) 1_{A_{(2)}}(\a_{(2)})
\end{equation}
and
$$ g_{34}(a) \coloneqq  \E 1_{A_{(3)}}(\a_{(3)}) 1_{A_{(4)}}(a-\a_{(3)}).$$

From Lemma \ref{wmix}, we have
$$ \sum_n \P(\n = n) \left| \E f_1(n - \k) 1_{A_{(2)}}( a_{(2),1} + \k ) \right|^2 \ll \eta^{50 C_1} $$
if $\n, \k$ are drawn independently and regularly from $a_{(i),1} + B(S_1, \rho_{2,j} )$ and $B(S_1, \rho_{2,j+m} )$ respectively.  Note that the pair $(\n, \k)$ has the same distribution as $(\a - a_{(2),1}, \a_{(2)} - a_{(2),1})$, thus
$$ \sum_a \P(\a = a) \left| \E f_1(a - \a_{(2)}) 1_{A_{(2)}}( \a_{(2)} ) \right|^2 \ll \eta^{50 C_1}.$$
From 	\eqref{f1-def}, \eqref{alphai-def}, \eqref{g12-def} we have
$$	\E f_1(a - \a_{(2)}) 1_{A_{(2)}}( \a_{(2)} )  = g_{12}(a) - \alpha_1 \alpha_2$$
and thus
\begin{equation}\label{po-2}
 \sum_a \P(\a = a) \left| g_{12}(a) - \alpha_1 \alpha_2 \right|^2 \ll \eta^{50 C_1}.
\end{equation}
Similarly we have
\begin{equation}\label{po-3}
 \sum_a \P(\a = a) \left| g_{34}(a) - \alpha_3 \alpha_4 \right|^2 \ll \eta^{50 C_1}.
\end{equation}
From Cauchy-Schwarz and the triangle inequality we conclude that
$$ \sum_a \P(\a = a) \left| g_{12}(a) g_{34}(a) - \alpha_1 \alpha_2 \alpha_3 \alpha_4 \right| \ll \eta^{25 C_1},$$
and hence by \eqref{si} and the triangle inequality
\begin{equation}\label{sigmoid}
 \sigma = \alpha_1 \alpha_2 \alpha_3 \alpha_4 + O( \eta^{25C_1}).
\end{equation}
In particular, from \eqref{step} one has 
\begin{equation}\label{alphoid}
\alpha_1 \alpha_2 \alpha_3 \alpha_4 \gg \eta^{C_1+O(1)}.
\end{equation}

From \eqref{sigmoid}, \eqref{alphoid} and \eqref{dance} we have
$$
\E h(\a) \ll \eta^{C_1/100} \alpha_1 \alpha_2 \alpha_3 \alpha_4$$
where
\begin{equation}\label{hadef}
 h(a) \coloneqq  \E W(a-\a_{(2)}, \a_{(2)}, \a_{(3)}, a-\a_{(3)}).
\end{equation}
By Markov's inequality, we conclude that we have
\begin{equation}\label{glow}
 h(\a) \ll \eta^{C_1/200} \alpha_1 \alpha_2 \alpha_3 \alpha_4
\end{equation}
with probability $1 - O(\eta^{C_1/200})$.  Similarly, from \eqref{po-2}, \eqref{po-3} and Chebyshev's inequality we also have
\begin{equation}\label{rest}
g_{12}(\a) = \alpha_1 \alpha_2 + O( \eta^{20 C_1} )
\end{equation}
and
\begin{equation}\label{rest-2}
 g_{34}(\a) = \alpha_3 \alpha_4 + O( \eta^{20 C_1} )
\end{equation}
with probability $1 - O(\eta^{C_1/200})$.

Now let $a$ be a value of $\a$ be such that \eqref{glow}, \eqref{rest}, \eqref{rest-2} hold.  From \eqref{rest-2} we have in particular that
$$ \E 1_{A_{(3)}}(\a_{(3)}) 1_{A_{(4)}}(a-\a_{(3)}) \gg \alpha_3 \alpha_4;$$
comparing this with \eqref{glow} and \eqref{hadef}, we see that we may find $a_{(3)}(a) \in A_{(3)}$ (depending only on $a$) with $a - a_{(3)}(a) \in A_{(4)}$ such that
\begin{align*} \E 1_{A_{(1)}}(a-\a_{(2)}) 1_{A_{(2)}}(\a_{(2)}) 1_{\VBQ}( a - \a_{(2)}, \a_{(2)}, & a_{(3)}(a), a - a_{(3)}(a) ) \\ & \ll \eta^{C_1/200} \alpha_1 \alpha_2.\end{align*}
If we then set $\xi'(a) \coloneqq  \xi(a_{(3)}(a)) + \xi(a-a_{(3)}(a))$ (and define $\xi'(\a)$ arbitrarily when \eqref{glow}, \eqref{rest}, or \eqref{rest-2} fail), then the claims \eqref{no}, \eqref{no-2} follow from \eqref{rest} and the definition \eqref{bad-quad-2} of $\VBQ$.
\end{proof}

The function $\xi'$ has better additive structure than $\xi$, in that it respects almost all additive quadruples in a Bohr set, rather than almost all additive quadruples in a rough set.  More precisely, we have the following.

\begin{proposition}\label{pordo}  Let the notation and hypotheses be as in Theorem \ref{fifth-step}. Suppose that $\a, \a', \h$ are selected independently and regularly from $a_1 + B(S_1, \rho_{2,j} )$, $a_1 + B(S_1, \rho_{2,j} )$, and $B(S_1, \rho_{2,j+3} )$ respectively. Then with probability $1 - O( \eta^{C_1/200})$ we have 
\begin{equation}\label{stop}
 \| \xi'(\a) - \xi'(\a+\h) - \xi'(\a') + \xi'(\a'+\h) \|_{S} \leq \frac{4}{\rho_3}.
\end{equation}
\end{proposition}

\begin{proof}  Let $\a_{(2)}$ be drawn regularly from $a_{(2),1} + B(S_1, \rho_{2,j+2})$, independently of $\a,\a',\h$.  For each $a,a',h \in \Z/p\Z$, let $\I_{a,a',h}$ denote the random indicator variable 
$$ \I_{a,a',h} \coloneqq  1_{A_{(2)}}(\a_{(2)}) 1_{A_{(2)}}(\a_{(2)}+h) 1_{A_{(1)}}(a-\a_{(2)}) 1_{A_{(1)}}(a'-\a_{(2)}).$$ 
Suppose that we can show that with probability $1-O(\eta^{C_1/200})$, the triple $(\a,\a',\h)$ attains a value $(a,a',h)$ for which one has the estimates
\begin{align}
\E \I_{a,a',h} &\geq 0.9 \alpha_1^2 \alpha_2^2 \label{pen1}\\
\E \I_{a,a',h} 1_{\| \xi'(a) - \xi(a-\a_{(2)})-\xi(\a_{(2)}) \|_{S} > 1/\rho_3} &\leq 0.1 \alpha_1^2 \alpha_2^2 \label{pen2}\\
\E \I_{a,a',h}  1_{\| \xi'(a') - \xi(a'-\a_{(2)})-\xi(\a_{(2)}) \|_{S} > 1/\rho_3}  &\leq 0.1 \alpha_1^2 \alpha_2^2 \label{pen3}\\
\E \I_{a,a',h} 1_{\| \xi'(a+h) - \xi(a-\a_{(2)})-\xi(\a_{(2)}+h) \|_{S} > 1/\rho_3} &\leq 0.1 \alpha_1^2 \alpha_2^2 \label{pen4}\\
\E \I_{a,a',h} 1_{\| \xi'(a'+h) - \xi(a'-\a_{(2)})-\xi(\a_{(2)}+h) \|_{S} > 1/\rho_3} &\leq 0.1 \alpha_1^2 \alpha_2^2. \label{pen5}
\end{align}
Assuming these estimates, we conclude from the union bound that with probability $1 - O(\eta^{C_1/200})$, the random variable $(\a,\a',\h)$ attains a value $(a,a',h)$ for which there exists at least one element $a_{(2)}$ of $\Z/p\Z$ obeying the constraints
\begin{align*}
a_{(2)}, a_{(2)}+h &\in A_{(2)} \\  
a-a_{(2)}, a'-a_{(2)} &\in A_{(1)} \\
 \| \xi'(a) - \xi(a-a_{(2)})-\xi(a_{(2)}) \|_{S} &\leq \frac{1}{\rho_3} \\
 \| \xi'(a') - \xi(a'-a_{(2)})-\xi(a_{(2)}) \|_{S} &\leq \frac{1}{\rho_3}\\
 \| \xi'(a+h) - \xi(a-a_{(2)})-\xi(a_{(2)}+h) \|_{S} &\leq \frac{1}{\rho_3}\\
 \| \xi'(a'+h) - \xi(a'-a_{(2)})-\xi(a_{(2)}+h) \|_{S} &\leq \frac{1}{\rho_3} 
\end{align*}
and \eqref{stop} then follows from the triangle inequality.

It remains to establish \eqref{pen1}-\eqref{pen5}.
We first prove \eqref{pen2}.  By Markov's inequality, it suffices to show that
$$ \E \I_{\a,\a',\h} 1_{\| \xi'(\a) - \xi(\a-\a_{(2)})-\xi(\a_{(2)}) \|_{S} > 1/\rho_3}  \ll \eta^{C_1/200} \alpha_1^2 \alpha_2^2.$$
We rewrite the left-hand side as
$$ \E g_1(\a_{(2)}) g_2(\a_{(2)}) 1_{A_{(2)}}(\a_{(2)}) 1_{A_{(1)}}(\a-\a_{(2)}) 1_{\| \xi'(\a) - \xi(\a-\a_{(2)})-\xi(\a_{(2)}) \|_{S} > 1/\rho_3}$$
where
$$ g_1( a_{(2)} ) \coloneqq  \E 1_{A_{(1)}}( \a' - a_{(2)} )$$
and
$$ g_2( a_{(2)} ) \coloneqq  \E 1_{A_{(2)}}( a_{(2)} + \h ).$$
But from \eqref{no-2} we have
$$ \E 1_{A_{(2)}}(\a_{(2)}) 1_{A_{(1)}}(\a-\a_{(2)}) 1_{\| \xi'(\a) - \xi(\a-\a_{(2)})-\xi(\a_{(2)}) \|_{S} > 1/\rho_3} \ll \eta^{C_1/200} \alpha_1 \alpha_2,$$
from Lemma \ref{ati} one has
$$ g_1( \a_{(2)} ) = \alpha_1 + O( \eta^{10 C_1} )$$
and from \eqref{po} one has
$$ g_2( \a_{(2)} ) = \alpha_2 + O( \eta^{10 C_1} )$$
with probability $1 - O(\eta^{10 C_1})$ (say), with the trivial bound $g(\a_{(2)}) = O(1)$ otherwise, and the claim \eqref{pen2} then follows from \eqref{alphoid}.

The proofs of \eqref{pen3}-\eqref{pen5} are similar to \eqref{pen2} and are omitted.   It thus remains to prove \eqref{pen1}.  
From \eqref{op} and Markov's inequality, we see that with probability $1-O(\eta^{C_1/200})$, the random variable $\h$ attains a value $h$ for which
$$ \E 1_{A_{(2)}}( \a_{(2)} ) 1_{A_{(2)}}( \a_{(2)}+h) \geq 0.99 \alpha_2^2.$$
For any $h$ obeying this inequality, define $E(h) \subset \Z/p\Z$ to be the set
$$ E(h) \coloneqq  A_{(2)} \cap (A_{(2)}-h),$$ 
so that 
$$ \P( \a_{(2)} \in E(h) ) \geq 0.99 \alpha_2^2.$$
By \eqref{po} and the Chebyshev inequality, we conclude that with probability $1-O(\eta^{C_1/200})$, the random variable $(\a,\h)$ attains a value $(a,h)$ for which one has
$$ \P( \a_{(2)} \in E(h); a - \a_{(2)} \in A_{(1)} ) \geq 0.98 \alpha_1 \alpha_2^2.$$
For any $(a,h)$ of the above form, define $E'(a,h) \subset \Z/p\Z$ to be the set
$$ E'(a,h) \coloneqq  \E(h) \cap (a - A_{(1)}),$$
then
$$ \P( \a_{(2)} \in E'(a,h)  ) \geq 0.98 \alpha_1 \alpha_2^2.$$
By one last application of \eqref{po} and the Chebyshev inequality, we see that with probability $1-O(\eta^{C_1/200})$, the random variable $(\a',\a,\h)$ attains a value $(a',a,h)$ for which one has
$$ \P( \a_{(2)} \in E'(a,h); a' - \a_{(2)} \in A_{(1)} ) \geq 0.97 \alpha_1^2 \alpha_2^2$$
which gives \eqref{pen1} as required.
\end{proof}

\subsection{Sixth step: a frequency function $\xi''$ that is approximately linear $100\%$ of the time on a Bohr set}

We now use a standard ``majority vote'' argument to upgrade the ``$99\%$ linear'' structure of $\xi'$ to a ``100\% linear'' structure of a closely related function $\xi''$ (cf. \cite{blum}).  More precisely, one has

\begin{theorem}\label{sixth-step} Let the notation and hypotheses be as in Theorem \ref{locu3}.  Let $j, S_1$ be as in Theorem \ref{fourth-step}, and let $a_1$, $\xi'$ be as in Theorem \ref{fifth-step}.  Then there is a function $\xi'': B(S_1,\rho_3) \to \Z/p\Z$ such that
\begin{equation}\label{linr}
\| \xi''(n+m) - \xi''(n) - \xi''(m) \|_{S} \leq \frac{24}{\rho_3}
\end{equation}
for all $n,m \in B(S_1,\rho_3/2)$, and such that for any $n \in B(S_1,\rho_3)$, if $\a$ is drawn regularly from $a_1 + B(S_1, \rho_{2,j})$, one has
\begin{equation}\label{senile}
\| \xi'(\a) - \xi'(\a-n) - \xi''(n) \|_{S} \leq \frac{8}{\rho_3}
\end{equation}
with probability $1-O(\eta^{C_1/200})$.
\end{theorem}

\begin{proof}  Let $\a, \h$ be drawn independently and regularly from $a_* + B(S_1, \rho_{2,j} )$ and $B(S_1, \rho_{2,j+3})$ respectively.  From Proposition \ref{pordo} and the pigeonhole principle, we may find $a'_0 \in\Z/p\Z$ such that
\begin{equation}\label{slash}
 \P\left( \| \xi'(\a) - \xi'(\a+\h) - \xi'(a'_0) + \xi'(a'_0+\h) \|_S \leq \frac{4}{\rho_3} \right) \geq 1 - O( \eta^{C_1/200} ).
\end{equation}
Fix this $a'_0$.  Now let $n$ by an arbitrary element of $B(S_1,\rho_3)$.  Then using Lemma \ref{ati} to compare $\a$ with $\a-n$ and $\h$ with $\h+n$, we obtain
\begin{align*} \P\bigg( \| \xi'(\a-n) - \xi'(\a+\h) - \xi'(a'_0) + \xi'(a'_0+\h+n) \|_S & \leq \frac{4}{\rho_3} \bigg)\\ &  \geq 1 - O( \eta^{C_1/200} ).\end{align*}
Combining this with \eqref{slash} and the triangle inequality, we see that
\begin{align*}
 \P\bigg( \| \xi'(\a) - \xi'(\a-n) + \xi'(a'_0+\h) - \xi'(a'_0 + \h + n) \|_S & \leq \frac{8}{\rho_3} \bigg) \\ & \geq 1 - O( \eta^{C_1/200} ).
\end{align*}
Thus, by the pigeonhole principle, we may find $h_n \in \Z/p\Z$ such that
\begin{align*}
 \P\bigg( \| \xi'(\a) - \xi'(\a-n) + \xi'(a'_0+h_n) - \xi'(a'_0 + h_n + n) \|_S & \leq \frac{8}{\rho_3} \bigg) \\ & \geq 1 - O( \eta^{C_1/200} ).
\end{align*}
If we thus define
$$ \xi''(n) \coloneqq  \xi'(a'_0+h_n+n) - \xi'(a'_0+n)$$
then we have obtained \eqref{senile}.

Now suppose that $n, m \in B( S_1, \rho_3/2 )$.  From \eqref{senile}, we see that with probability at least $1 - O(\eta^{C_1/200})$ we have
$$ \| \xi'(\a) - \xi'(\a-n) - \xi''(n) \|_S \leq \frac{8}{\rho_3},$$
$$ \| \xi'(\a) - \xi'(\a-m) - \xi''(m) \|_S \leq \frac{8}{\rho_3},$$
and
$$ \| \xi'(\a) - \xi'(\a-n-m) - \xi''(n+m) \|_S \leq \frac{8}{\rho_3}.$$
Using Lemma \ref{ati} to compare $\a$ with $\a-n$ in the second inequality, we also conclude
$$ \| \xi'(\a-n) - \xi'(\a-n-m) - \xi''(m) \|_S \leq \frac{8}{\rho_3},$$
with probability $1-O(\eta^{C_1/200})$.  Thus there is a positive probability that the first, third, and fourth estimates hold simultaneously, and the claim \eqref{linr} follows from the triangle inequality.
\end{proof}

The function $\xi''$ is still closely related to $\xi$, and in particular a variant of the correlation estimate \eqref{sam} is obeyed by $\xi''$.

\begin{proposition}\label{dfcor}  Let the notation and hypotheses be as in the preceding theorem.  Then there exist $a_0 \in B(S, 3\rho_2)$ and $\xi_0 \in \Z/p\Z$ such that
\begin{align*}
\sum_{n_0,n} \P(\n_0=n_0, \n=n) | \E f(n_0+\h+a_0-n) \overline{f}(n_0+\h) & e_p( (\xi''(n) -  \xi_0) \h ) |^2 \\ & \gg \eta^{C_1+O(1)},
\end{align*}
where $\n, \n_0,\h$ are drawn independently and regularly from the Bohr sets $B(S_1, \rho_3/4)$, $B(S,\rho_0)$, $B(S_1, \rho_4)$ respectively.
\end{proposition}

With this proposition and the previous theorem, we may now safely forget about the original function $\xi$, and work now with $\xi''$; the parameters $a_1, j$ will also no longer be relevant.

\begin{proof}  Let $\n$, $\a$, $\a_{(2)}$ be drawn independently and regularly from $B(S_1, \rho_3/4)$, $a_1 + B(S_1, \rho_{2,j})$, and $B(S_1, \rho_{2,j+2} )$ respectively.  From \eqref{senile} we have
$$ \| \xi'(\a) - \xi'(\a-\n) - \xi''(\n) \|_S \ll \frac{1}{\rho_3}$$
with probability $1-O( \eta^{C_1/200} )$.  Similarly, from \eqref{no}, \eqref{no-2}, \eqref{alphoid} we see that with probability $1-O(\eta^{C_1/200})$, the random variable $\a$ attains a value $a$ for which
\begin{align*}
\P\bigg( a - \a_{(2)} \in A_{(1)}; \a_{(2)} \in A_{(2)}; \| \xi'(a) - \xi(a-\a_{(2)})- & \xi(\a_{(2)}) \|_S \leq \frac{1}{\rho_3} \bigg) \\ & \gg \alpha_1 \alpha_2.
\end{align*}
Using Lemma \ref{ati} to compare $\a$ and $\a-\n$, we also see that with with probability $1-O(\eta^{C_1/200})$, the random variable $(\a,\n)$ attains a value $(a,n)$ for which
\begin{align*}
\P\bigg( a-n - & \a_{(2)} \in A_{(1)}; \a_{(2)} \in A_{(2)}; \\   & \| \xi'(a-n) - \xi(a-n-\a_{(2)})-\xi(\a_{(2)}) \|_S \leq \frac{1}{\rho_3} \bigg) \gg \alpha_1 \alpha_2.
\end{align*}
From the union bound and Fubini's theorem, we conclude that with probability $\gg \alpha_1 \alpha_2$, we simultaneously have the statements
\begin{align*}
\a-\n - \a_{(2)} &\in A_{(1)} \\
\a_{(2)} &\in A_{(2)}\\
 \| \xi'(\a) - \xi'(\a-\n) - \xi''(\n) \|_S &\ll \frac{1}{\rho_3} \\
\| \xi'(\a-\n) - \xi(\a-\n-\a_{(2)})-\xi(\a_{(2)}) \|_S &\leq \frac{1}{\rho_3}
\end{align*}
and hence by the triangle inequality
$$ \| \xi'(\a) - \xi(\a-\n-\a_{(2)}) - \xi(\a_{(2)}) - \xi''(\n) \|_S \ll \frac{1}{\rho_3}.$$
By the pigeonhole principle, we may thus find $a, a_{(2)} \in \Z/p\Z$ such that the statements
\[
a-\n - a_{(2)} \in A_{(1)} \]
\[ a_{(2)} \in A_{(2)} \]
\[  \| \xi'(a) - \xi(a-\n-a_{(2)}) - \xi(a_{(2)}) - \xi''(\n) \|_S \ll \frac{1}{\rho_3}  \]
simultaneously hold with probability $\gg \alpha_1 \alpha_2$, and thus with probability $\gg \eta^{C_1+O(1)}$ thanks to \eqref{alphoid}.  Writing $a_0 \coloneqq  a - a_{(2)}$ and $\xi_0 \coloneqq  \xi(a_{(2)}) - \xi'(a)$, and recalling from Theorem \ref{third-step} that $A_{(1)} \in S$, we thus have
$$ 
\P\left( a_0 - \n \in S; \| \xi''(\n) + \xi( a_0 - \n ) - \xi_0 \|_{S} \ll 1/\rho_3 \right) \gg \eta^{C_1+O(1)}.$$

In particular, since $\n \in B(S_1, \rho_3/4)$ and $S \subset B(S,2\rho_2)$, we have $a_0 \in B(S, 3\rho_2)$.

Let $\n_0, \n_1$ be drawn independently and regularly from $B(S,\rho_0), B(S,\rho_1)$ respectively, independently of all previous random variables.  From the above estimate and \eqref{sam}, we see that with probability $\gg \eta^{C_1+O(1)}$, the random variable $\n$ attains a value $n$ for which the statements
\begin{equation}
a_0 - n \in S \label{so0} \end{equation}
\begin{equation}
\| \xi''(n) + \xi( a_0 - n ) - \xi_0 \|_{S_1} \ll 1/\rho_3 \label{so1}\end{equation}
\begin{equation} \sum_{n_0} \P(\n_0 = n_0) 
\left| \E f(n_0+\n_1+a_0-n) \overline{f}(n_0+\n_1) e_p( -\xi(a_0-n) \n_1 )\right|^2 \geq \eta/8 \label{so2}
\end{equation}
simultaneously hold.  

Let $n$ obey the above estimates \eqref{so0}, \eqref{so1}, \eqref{so2}. If we now draw $\h$ regularly from $B(S_1, \rho_4)$, then by using Lemma \ref{ati} to compare $\n_1$ with $\n_1+\h$ in \eqref{so2}, we obtain
\begin{align*}
\sum_{n_0} \P(\n_0=n_0) \bigg| \E f(n_0+\n_1+\h+a_0-n) & \overline{f}(n_0+\n_1+\h) \times \\  & \times e_p( -\xi(a_0-n) (\n_1+\h) ) \bigg|^2  \gg \eta 
\end{align*}
and thus by the triangle inequality in $L^2$
\begin{align*}
\sum_{n_0,n_1} \P(\n_0=n_0,\n_1=n_1) \bigg|\E f(n_0+& n_1+\h+a_0-n)   \overline{f}(n_0+n_1+\h) \times \\ & \times e_p( -\xi(a_0-n) (n_1+\h) ) \bigg|^2 \gg \eta.
\end{align*}
We may delete the deterministic phase $e_p(-\xi(a_0-n) n_1)$ to obtain
\begin{align*}
\sum_{n_0,n_1} \P(\n_0=n_0,\n_1=n_1) \bigg|\E f(n_0+n_1& +\h+a_0-n)  \overline{f}(n_0+n_1+\h) \times \\ & \times e_p( -\xi(a_0-n) \h ) \bigg|^2 \gg \eta.
\end{align*}
Since $\h$ takes values in $B(S_1,\rho_4)$, we see from \eqref{so1} that
$$ e_p( -\xi(a_0-\n) \h ) = e_p( (\xi''(\n) - \xi_0) \h ) + O( \eta^{100} )$$
(say), and so
\begin{align*}
 \sum_{n_0,n_1} \P(\n_0=n_0,\n_1=n_1) \bigg|\E f(n_0+n_1&+\h+a_0-n) \overline{f}(n_0+n_1+\h) \times \\ & \times e_p( (\xi''(n) - \xi_0) \h )\bigg|^2 \gg \eta.
\end{align*}
Using Lemma \ref{ati} to compare $\n_0$ with $\n_0+\n_1$, we conclude that
\begin{align*}
 \sum_{n_0,n_1} \P(\n_0=n_0,\n_1=n_1) \bigg|\E f(n_0+\h+a_0-n) \overline{f}(n_0+\h) e_p( & (\xi''(n) - \xi_0) \h )\bigg|^2 \\ & \gg \eta.
\end{align*}
Multiplying by $\P(\n=n)$ and summing in $n$, we obtain the claim.
\end{proof}

\subsection{Seventh step: derivatives of $f$ correlate with a locally bilinear form}

We now pass to the ``cohomological'' phase of the argument, in which we remove the error $\xi''(n+m) - \xi''(n) - \xi''(m)$ in the linearity of $\xi''$ that appears in \eqref{linr}.  This improved linearity of the form $(n,h) \mapsto \xi(n) h$ in the $n$ aspect will come at the expense of the $h$ aspect, which will now merely be locally linear instead of globally linear.  However, this is a worthwhile tradeoff for our purposes (and in any event local linearity is more natural in this context than global linearity).

More precisely, the purpose of this subsection is to establish the following result towards the proof of Theorem \ref{locu3}.

\begin{theorem}\label{seventh-step}  Let the notation and hypotheses be as in Theorem \ref{locu3}.  Then there exists a set $S_1$ with $S \subset S_1 \subset \Z/p\Z$ and $|S_1| \leq |S| + O(\eta^{-O(C_1)})$, a locally bilinear map
$$ \Xi: B(S_1,\rho_4) \times B(S_1,\rho_4) \to \R/\Z,$$
a shift $a_1 \in B(S,4\rho_2)$, and a frequency $\xi_1 \in \Z/p\Z$ such that
\begin{equation}\label{no7}
\begin{split}
& \sum_{n_0,n_1} \P(\n_0=n_0, \n_1=n_1) \times \\
&\quad \left| \E f(n_0+\m_1+a_1-n_1) \overline{f}(n_0+\m_1) e\left( \Xi(n_1,\m_1) - \frac{\xi_1 \m_1}{p} \right) \right|^2 \gg \eta^{C_1+O(1)}
\end{split}
\end{equation}
if $\n_0, \m_1, \n_1$ are drawn independently and regularly from $B(S,\rho_0)$, $B(S_1,\rho_5)$, and $B(S_1, \rho_6)$ respectively.
\end{theorem}

Once the proof of this theorem is completed, the auxiliary data $\xi, \xi', \xi'', j$, $\Omega, \VBQ$ used in the previous parts of the section are no longer needed and may be discarded.

We now prove Theorem \ref{seventh-step}.  Let $j_*, S_1$ be as in Theorem \ref{fourth-step}, let $a_*$, $\xi'$ be as in Theorem \ref{fifth-step}, let $\xi'': B(S_1,\rho_3) \to \Z/p\Z$ be as in Theorem \ref{sixth-step}, and let $a_0, \xi_0$ be as in Proposition \ref{dfcor}.  We will use a ``cohomological'' argument to construct the required bilinear map $\Xi$.  Namely, we define the \emph{cocycle} $\mu: B(S_1,\rho_3/2) \times B(S_1,\rho_3/2) \to \Z/p\Z$ to be the quantity
\begin{equation}\label{mun-def}
\mu(n,m) \coloneqq \xi''(n+m) - \xi''(n) - \xi''(m).
\end{equation}
Clearly \eqref{linr} is symmetric, and we have the \emph{cocycle equation}
\begin{equation}\label{cocycle}
 \mu(n_1, n_2 + n_3) + \mu(n_2,n_3) = \mu(n_1, n_2) + \mu(n_1 + n_2, n_3) 
\end{equation}
as well as the auxiliary equations
$$ \mu(n_1,n_2) = \mu(n_2, n_1); \quad \mu(n_1,0) = 0$$
whenever $n_1,n_2, n_3 \in B(S_1,\rho_3/4)$.  From \eqref{linr} we also have the estimate
\begin{equation}\label{linra}
 \|\mu(n,m) \|_S \leq \frac{24}{\rho_3}
\end{equation}
for all $n, m \in B(S_1,\rho_3/4)$.  

To construct the bilinear map $\Xi$, we will show that a certain projection of $\mu$ is a ``coboundary'' is a certain sense.  Let $\phi: \Z^S \to \Z/p\Z$ be the homomorphism
$$ \phi( (n_s)_{s \in S} ) \coloneqq \sum_{s \in S} n_s s.$$
From \eqref{linra}, we see that for each $n,m \in B(S_1, \rho_3/4)$ we have a representation of the form
\begin{equation}\label{mun}
 \mu(n,m) = \phi( \tilde \mu(n,m) )
\end{equation}
for some lift $\tilde \mu(n,m) \in \Z^S$ of size
\begin{equation}\label{mun2}
 |\tilde \mu(n,m)| \leq 24 / \rho_3.
\end{equation}
This lift $\tilde \mu(n,m)$ is only defined up to an element of the kernel $\mathrm{ker}(\phi) \coloneqq \{ p \in \Z^S: \phi(p) = 0 \}$ of $\phi$; to eliminate this ambiguity we will apply a projection.  Since $S$ contains a non-zero element, $\phi: \Z^S \to \Z/p\Z$ is a surjective homomorphism, and in particular, $\mathrm{ker}(\phi)$ is a sublattice of $\Z^S$ of index $p$.  Applying Lemma \ref{bohr-basis}, we may find generators $v_1, \dots, v_{|S|}$ of $\mathrm{ker}(\phi)$ and real numbers $N_1,\dots,N_{|S|} > 0$ with
\begin{equation}\label{ips}
 \prod_{i=1}^{|S|} N_i = O(K)^{O(K)} p 
\end{equation}
such that
\begin{align}\nonumber
 B_{\R^S} (0, O(K)^{-3K/2} t  ) \cap \mathrm{ker}(\phi)  \subset  \{ n_1 v_1 + \dots & + n_{|S|} v_{|S|}: |n_i| \leq tN_i \} \\ &  \subset B_{\R^S}(0,t) \cap \mathrm{ker}(\phi) \label{kco}
\end{align}
for all $t > 0$.  

By relabeling, we may take the $N_i$ to be non-increasing.  Let $d$, $0 \leq d \leq |S|$ be such that
\begin{equation}\label{ngrad}
 N_1 \geq \dots \geq N_d > \frac{\rho_3}{\exp(K^{C_1})} \geq N_{d+1} \geq \dots \geq N_{|S|}.
\end{equation}
From \eqref{ips}, \eqref{cstar2} we see that $d$ cannot equal $|S|$.  Let $V$ be the $d$-dimensional subspace of $\R^S$ spanned by $v_1,\dots,v_d$, let $V^\perp$ be the orthogonal complement of $V$ in $\R^S$, and let $\pi: \R^S \to V^\perp$ be the orthogonal projection.

We claim that $\pi(\tilde \mu(n,m))$ is now uniquely determined by $\mu(n,m)$ for $n,m \in B(S_1,\rho_3/4)$.  Indeed, if $\tilde \mu(n,m)$ and $\tilde \mu'(n,m)$ both obeyed \eqref{mun}, \eqref{mun2}, then their difference (call it $w$) would be of magnitude $O( 1/\rho_3)$ and lies in the kernel of $\phi$.  By \eqref{kco} with $t = \exp( - K^{C_1}) \rho_3$, we conclude that $w$ lies in $V$, and hence $\pi(\tilde \mu(n,m))$ and $\pi(\tilde \mu'(n,m))$ agree.

A variant of the above argument shows that $\pi \circ \tilde \mu$ also continues to obey the cocycle equation.

\begin{lemma}[Projected lift is a cocycle]\label{pcycle}  One has
$$ \pi(\tilde \mu(n_1, n_2 + n_3)) + \pi(\tilde \mu(n_2,n_3)) = \pi(\tilde \mu(n_1, n_2)) + \pi(\tilde \mu(n_1 + n_2, n_3))$$
and additionally
$$ \pi(\tilde \mu(n_1,n_2)) = \pi(\tilde \mu(n_2,n_1)); \quad \pi(\tilde \mu(n_1,0)) = 0$$
for all $n_1,n_2,n_3 \in B(S_1,\rho_3/4)$.
\end{lemma}

\begin{proof}  By \eqref{mun2}, the quantity $w \coloneqq \tilde \mu(n_1, n_2 + n_3) + \tilde \mu(n_2,n_3) - \tilde \mu(n_1, n_2) - \tilde \mu(n_1 + n_2, n_3)$ has magnitude $O( 1/\rho_3)$; by \eqref{mun}, \eqref{cocycle}, $w$ lies in the kernel of $\phi$.  Repeating the previous arguments, we conclude that $w \in V$.  Applying the homomorphism $\pi$, we obtain the first claim.  The second claim is proven similarly.
\end{proof}

We can in fact make $\pi \circ \tilde \mu$ a coboundary, after shrinking the domain somewhat.

\begin{proposition}[Projected lift is a coboundary]\label{projlift}  There exists a map $F: B(S_1, 2\exp(-K^{C_1^2}) \rho_3) \to V^\perp$ with
\begin{equation}\label{fsm}
 F(n) \ll \frac{K^{O(C_1)}}{\rho_3}
\end{equation}
for all $n \in B(S_1, 2\exp(-K^{C_1^2}) \rho_3)$, such that
$$ \pi(\tilde \mu(n_1,n_2)) = F(n_1+n_2) - F(n_1) - F(n_2) $$
for all $n_1,n_2 \in B(S_1, \exp(-K^{C_1^2})  \rho_3)$.
\end{proposition}

\begin{proof}  As a first attempt at constructing $F$, we introduce the average
$$ F_1(n) \coloneqq \E \pi(\tilde \mu(n, \n_3))$$
for $n \in B(S_1,\rho_3/4)$, where $\n_3$ is drawn regularly from $B(S_1, \rho_3/4)$.  From \eqref{mun2} we have
$$ |F_1(n)| \leq \frac{24}{\rho_3}$$
for all $n \in B(S_1,\rho_3/4)$.  Also, since $|S_1| \ll K^{O(C_1)}$, if we replace $n_3$ by $\n_3$ in Lemma \ref{pcycle} and take expectations using Lemma \ref{ati}, we conclude that
$$ F_1(n_1) + F_1(n_2) = \pi(\tilde \mu(n_1,n_2)) + F_1(n_1+n_2) + O\left( \frac{K^{O(C_1)} \| n_2 \|_{S_1^\perp}}{\rho^2_3} \right)$$
for all $n_1,n_2 \in B(S_1,\rho_3/8)$.

If we now introduce the modified cocycle 
$$ \sigma_1(n_1,n_2) \coloneqq \pi(\tilde \mu(n_1,n_2)) + F_1(n_1+n_2) - F_1(n_1) - F_1(n_2) $$
for $n_1,n_2 \in B(S_1,\rho_3/8)$, then we have the cocycle equation
\begin{equation}\label{sig1-cocycle}
 \sigma_1(n_1,n_2+n_3) + \sigma_1(n_2,n_3) = \sigma_1(n_1,n_2) + \sigma_1(n_1+n_2,n_3), 
\end{equation}
the auxiliary equations
$$ \sigma_1(n_1,n_2) = \sigma_1(n_2,n_1); \quad \sigma_1(n_1,0) = 0$$
and the bound
\begin{equation}\label{sang}
 \sigma_1(n_1,n_2) \ll \frac{K^{O(C_1)} \| n_2 \|_{S_1^\perp}}{\rho^2_3} 
\end{equation}
for $n_1,n_2 \in B(S_1,\rho_3/16)$.

We now make $\sigma_1$ a coboundary by using a basis for $B(S_1,\rho_3/16)$.  Set $d \coloneqq |S_1| \leq K^{O(C_1)}$.  By Corollary \ref{bohr-basis-cor}, we can find $a_1,\dots,a_d$ of $\Z/p\Z$ and real numbers $N_1,\dots,N_d > 0$ such that
\begin{equation}\label{ain-1}
 \|a_i\|_{S_1^\perp} \leq N_i^{-1}
\end{equation}
for all $i=1,\dots,d$, and such that for any $a \in \Z/p\Z$, there exists a representation
\begin{equation}\label{nag-1}
 a = m_1 a_1 + \dots + m_d a_d
\end{equation}
with $m_1,\dots,m_d$ integers of size
\begin{equation}\label{nis-1}
 m_i \ll \exp( O(K^{O(C_1)}) ) N_i \| a\|_{S_1^\perp}
\end{equation}
for $i=1,\dots,d$, with at most one such representation obeying the bounds $|m_i| < N_i/2$ for $i=1,\dots,d$.

By relabeling we may assume that $N_i \geq 32 d' / \rho_3$ for $i=1,\dots,d'$ and $N_i <32 d' / \rho_3$ for $i=d'+1,\dots,d$ for some $0 \leq d' \leq d$.
By \eqref{ain-1} we have $a_i \in B(S_1, \rho_3 / 32d')$ for all $i=1,\dots,d'$.  In particular, from \eqref{sig1-cocycle} we see that for any $n \in B(S_1, \rho_3/32)$ and $1 \leq i,j \leq d'$, we have
$$ \sigma_1(n_1, a_i + a_j) + \sigma_1(a_i,a_j) = \sigma_1(n_1, a_i) + \sigma_1(n_1+a_i, a_j)$$
and hence by swapping $i$ and $j$ and subtracting
$$ \sigma_1(n_1 + a_j, a_i) - \sigma_1(n_1, a_i) = \sigma_1(n_1 + a_i, a_j) - \sigma_1( n_1, a_j ).$$
Let $P \subset \Z^{d'}$ denote the collection of tuples $(m_1,\dots,m_{d'}) \in \Z^{d'}$ with $|m_i| \leq \frac{\rho_3}{2N_i}$ for $i=1,\dots,d'$, and for each $m \in P$ and $i=1,\dots,d$, define the quantity
$$ f_i(m) \coloneqq \sigma_1( \phi(m), a_i)$$
where $\phi: \Z^{d'} \to \Z/p\Z$ is the homomorphism
$$ \phi(m_1,\dots,m_{d'}) \coloneqq \sum_{k=1}^{d'} m_k a_k.$$
Then from \eqref{ain-1} we have $\phi(P) \subset B(S_1, \rho_3/32)$.
The above identity then says that the ``$1$-form'' $(f_1,\dots,f_{d'})$ is ``closed'' or ``curl-free'' in the sense that
\begin{equation}\label{fimj}
 f_i(m + e_j) - f_i(m) = f_j(m + e_i) - f_j(m)
\end{equation}
whenever $i,j = 1,\dots,d'$ and $m, m+e_i, m+e_j \in P$, where $e_1,\dots,e_{d'}$ is the standard basis for $P$.  This implies that there exists a function $H: P \to V^\perp$ such that $F(0)=0$ and $f_i(m) = H(m+e_i) - H(m)$ whenever $i=1,\dots,d$ and $m,m+e_i \in P$.  Indeed, one can define $H$ to be an ``antiderivative'' of the $(f_1,\dots,f_{d'})$ by setting
$$ H(m) \coloneqq \sum_{l=0}^{L-1} f_{i_l}(m_l)$$
whenever $0=m_0,\dots,m_L = m$ is a path in $P$ with $m_{l+1} = m_l + e_{i_l}$ for $l=0,\dots,L-1$; a ``homotopy'' argument using \eqref{fimj} shows that the right-hand side does not depend on the choice of path.  From \eqref{sang}, \eqref{ain-1} we have
$$ f_i(m) \ll \frac{K^{O(C_1)}}{N_i \rho_3^2} $$
for $m \in P$ and $i=1,\dots,d'$, which on ``integrating'' (and recalling that $d' \leq d \ll K^{O(C_1)}$) implies that
$$ H(m) \ll \frac{K^{O(C_1)}}{\rho_3} $$
for all $m \in P$.

Since $\sigma_1( 0, e_i) = 0$, we have $f_i(0)=0$ and hence $H(e_i) = 0$ for all $i=1,\dots,d'$.  Thus we have
$$ \sigma_1( \phi(m), \phi(e_i) ) = H(m+e_i) - H(m) - H(e_i)$$
whenever $m, m+e_i \in P$.  An induction (on the magnitude of a vector $m'$) using \eqref{sig1-cocycle} then shows that
$$ \sigma_1( \phi(m), \phi(m') ) = H(m+m') - H(m) - H(m')$$
whenever $m, m', m+m' \in P$.  Now, if $n \in B( S_1, 2\exp(-K^{C_1^2}) \rho)$, then by \eqref{nag-1}, \eqref{nis-1} we see that $n = \phi(m)$ for some $m \in P$.  If we then define $F_2: B( S_1, 2\exp(-K^{C_1^2}) \rho) \to V^\perp$ by setting $F_2(n) \coloneqq H(m)$, we conclude that
$$ F_2(n) \ll \frac{K^{O(C_1)}}{\rho_3}$$
and
$$ \sigma_1( n, n' ) = F_2(n+n') - F_2(n) - F_2(n')$$
for all $n,n' \in B( S_1, \exp(-K^{C_1^2}) \rho)$. Setting $F \coloneqq F_2 - F_1$, we obtain the claim.
\end{proof}

Let $F$ be as in Proposition \ref{projlift}.  We use $F$ to construct the locally bilinear form 
$\Xi: B(S_1,\rho_4) \times B(S_1,\rho_4) \to \R/\Z$ as follows.  We first define the locally linear map $\iota: B(S_1, \rho_4) \to \R^{S}$ by the formula
$$ \iota(m) \coloneqq \left( \{ \frac{ms}{p} \} \right)_{s \in S},$$
where $x \mapsto \{x\}$ is the signed fractional map from $\R/\Z$ to $(-1/2,1/2]$; note that $\iota$ takes values in the box $[-\rho_4, \rho_4]^S$.  We then define
\begin{equation}\label{moody}
 \Xi( n, m ) \coloneqq \frac{\xi''(n) m}{p} - F(n) \cdot \iota( m ) 
\end{equation}
for $n,m \in B(S_1, \rho_4)$, where $\cdot$ denotes the dot product on $\R^S$.  It is clear that $\Xi$ is locally linear in $m$; we also claim that it is locally linear in $n$, thus
\begin{equation}\label{xil}
 \Xi(n_1+n_2,m) - \Xi(n_1,m) - \Xi(n_2,m) = 0
\end{equation}
whenever $n_1,n_2,n_1+n_2 \in B(S_1,\rho_4)$.  By \eqref{mun-def} and Proposition \ref{projlift}, the left-hand side of \eqref{xil} may be written as
$$ \frac{\mu(n_1,n_2) m}{p} - \pi(\tilde \mu(n_1,n_2)) \cdot \iota(m)\ \mathrm{mod}\ 1.$$
From \eqref{mun} we have
$$ \frac{\mu(n_1,n_2) m}{p} = \tilde \mu(n_1,n_2) \cdot \iota(m)\ \mathrm{mod}\ 1$$
so to prove \eqref{xil}, it suffices to show that $\iota(m)$ lies in $V^\perp$.  This is equivalent to showing that $\iota(m) \cdot v_i = 0$ for $i=1,\dots,d$.  Since $v_i \in \mathrm{ker}(\phi)$, we have
$$ \iota(m) \cdot v_i = 0 \ \mathrm{mod}\ 1.$$
On the other hand, we have $\iota(m) = O( K^{1/2} \rho_4 )$, and from \eqref{kco} with $t=N_i^{-1}$ followed by \eqref{ngrad}, we have 
$$|v_i| \leq N_i^{-1} < \frac{\exp(K^{C_1})}{\rho_3}$$
and hence $|\iota(m) \cdot v_i| < 1$.  The claim follows.

Now we verify \eqref{no7}.  Let $a_0, \xi_0$ be as in Proposition \ref{dfcor}.  Let $\n, \n_0,\h, \n_1$, $\m_1$ be drawn independently and regularly from the Bohr sets $B(S_1, \rho_3/4)$, $B(S,\rho_0)$, $B(S_1, \rho_4)$, $B(S_1, \rho_6)$, $B(S_1,\rho_5)$ respectively.  From Proposition \ref{dfcor} we have
\begin{align*}
 \sum_{n_0,n} \P(\n_0=n_0, \n = n) | \E f(n_0+\h+a_0-n) \overline{f}(n_0+\h) e_p( (\xi''( & n) - \xi_0) \h ) |^2 \\ & \gg \eta^{C_1+O(1)}.
\end{align*}
Using Lemma \ref{ati} to replace $\n$ by $\n+\n_1$, and to replace $\h$ by $\h+\m_1$, we have
\begin{align*}
 \sum_{n_0,n,n_1}&  \P(\n_0=n_0, \n=n, \n_1=n_1)  \bigg| \E f(n_0+\h+\m_1+a_0-n-n_1) \times \\  & \times \overline{f}(n_0+\h+\m_1)  e_p( (\xi''(n+n_1) - \xi_0) (\h+\m_1))\bigg|^2 \gg \eta^{C_1+O(1)}
\end{align*}
and thus by the triangle inequality we have
\begin{align*}
\sum_{n_0,n,n_1,h}& \P(\n_0=n_0, \n=n, \n_1=n_1, \h=h)  \bigg| \E f(n_0+h+\m_1+a_0-n-n_1) \times \\ & \times \overline{f}(n_0+h+\m_1) e_p( (\xi''(n+n_1) - \xi_0) (h+\m_1) ) \bigg|^2 \gg \eta^{C_1+O(1)}.
\end{align*}
The phase $e( (\xi''(n+n_1)-\xi_0) h )$ is deterministic and may thus be omitted:
\begin{align*} \sum_{n_0,n,n_1,h} & \P(\n_0=n_0, \n=n, \n_1=n_1, \h=h)  \bigg| \E f(n_0+h+\m_1+a_0-n-n_1) \times \\ & \times \overline{f}(n_0+h+\m_1) e_p( (\xi''(n+n_1) - \xi_0) \m_1) \bigg|^2 \gg \eta^{C_1+O(1)}.
\end{align*}
As the expectation only depends on the sum $n_0+h$ rather than the individual variables $n_0,h$, we thus have
\begin{align*} \sum_{n_0,n,n_1} & \P(\n_0+\h=n_0, \n=n, \n_1=n_1) \bigg| \E f(n_0+\m_1+a_0-n-n_1) \times \\ & \times \overline{f}(n_0+\m_1) e_p( (\xi''(n+n_1) - \xi_0) \m_1) \bigg|^2 \gg \eta^{C_1+O(1)}.
\end{align*}
By Lemma \ref{ati} we may replace $\n_0+\h$ here by $\n_0$.
From \eqref{linr} we have
$$ \| \xi''(n+n_1) - \xi''(n) - \xi''(n_1)) \m_1 \|_{\R/\Z} \ll \eta^{100C_1}$$
and so
\begin{align*}
\sum_{n_0,n,n_1} & \P(\n_0=n_0, \n=n, \n_1=n_1)  \bigg| \E f(n_0+a_0+\m_1-n-n_1) \times \\ & \times \overline{f}(n_0+\m_1) e_p( (\xi''(n)+ \xi''(n_1) - \xi_0) \m_1 ) 
\bigg|^2 \gg \eta^{C_1+O(1)}.
\end{align*}
By the pigeonhole principle, there thus exists $n \in B(S_*,\rho_3/4)$ such that
\begin{align*}
 \sum_{n_0,n_1} \P(\n_0=n_0 \n_1= & n_1)   | \E  f(n_0+a_0+\m_1-n-n_1) \overline{f}(n_0+\m_1) \times  \\ & \times e_p( (\xi''(n)+ \xi''(n_1) - \xi_0) \m_1 ) |^2 \gg \eta^{C_1+O(1)},
\end{align*}
which, if we write $a_1 \coloneqq  a_0-n$ and $\xi_1 \coloneqq  \xi_0 - \xi''(n)$, simplifies to
\begin{align*}
 \sum_{n_0,n_1} \P(\n_0=n_0 \n_1=n_1)    | \E f(& n_0+\m_1+a_1- n_1)  \overline{f}(n_0+\m_1) \times \\ & \times e_p( (\xi''(n_1) - \xi_1) \m_1 ) |^2 \gg \eta^{C_1+O(1)}.
\end{align*}
Since $a_0 \in B(S,3\rho_2)$ and $n \in B(S_*,\rho_3/4)$, we have $a_1 \in B(S,4\rho_2)$.

Now, from \eqref{moody} one has
$$ e_p( \xi''(n_1) \m_1 ) = e( \Xi(n_1,\m_1) ) e( - F(n_1) \cdot \iota( \m_1 ) );$$
but since $\m_1 \in B(S_*,\rho_5)$, we have $\iota(\m_1) = O( K \rho_5)$, and hence by \eqref{fsm} we have 
$$ \| F(n_1) \cdot \iota( \m_1 )  \|_{\R/\Z} \ll \eta^{100C_1},$$
and so
\begin{align*}
 \sum_{n_0,n_1} \P(\n_0=n_0; \n_1=n_1)     \bigg| \E&  f(n_0+\m_1+a_1-n_1) \overline{f}(n_0+\m_1) \times \\ & \times e( \Xi(n_1,\m_1) - \xi_1 \m_1 ) \bigg|^2 \gg \eta^{C_1+O(1)},
\end{align*}
which gives \eqref{no7}.   The proof of Theorem \ref{seventh-step} is now complete.

\subsection{Eighth step: making the frequency function symmetric}

The next step is the ``symmetry step'' from \cite{gt-inverseu3,sam}, which uses the Cauchy-Schwarz inequality to ensure that $\Xi$ is essentially symmetric.

\begin{theorem}\label{eighth-step} Let the notation and hypotheses be as in Theorem \ref{seventh-step}.
For $n,m \in B(S_1,\rho_4)$, define
$$ \{n,m\} \coloneqq  \Xi(n,m) -\Xi(m,n).$$
Then there exists a natural number $k$ with $1 \leq k \ll \exp(K^{O(C_1)})$ such that
$$ \| k\{n,m\} \|_{\R/\Z} \leq \frac{ \| n \|_{S_1^\perp}}{\rho_8} \frac{ \| m \|_{S_1}}{\rho_8}$$
for all $n,m \in B(S_1, \rho_9)$.
\end{theorem}

\begin{proof}
Let $\n_0,\m_1,\n_1$ be as in Theorem \ref{seventh-step}.
From \eqref{no7} and the pigeonhole principle, we may find $n_0 \in \Z/p\Z$ such that
\begin{align*} \sum_{n_1} \P(\n_1=n_1)  \bigg| \E f(n_0+\m_1+a_1-&n_1) \overline{f}(n_0+\m_1)\times \\ & \times  e( \Xi(n_1,\m_1) - \xi_1 \m_1 ) \bigg|^2 \gg \eta^{C_1+O(1)}\end{align*}
which by the boundedness of the expectation implies
\begin{align*} \sum_{n_1} \P(\n_1=n_1) \bigg| \E f(n_0+\m_1+a_1-n_1) & \overline{f}(n_0+\m_1) \times \\ & \times e( \Xi(n_1,\m_1) - \xi_1 \m_1 ) \bigg| \gg \eta^{C_1+O(1)}\end{align*}
and thus we may find a $1$-bounded function $b_1: \Z/p\Z \to \C$ such that
$$ |\E b_1(\n_1) f(n_0+\m_1+a_1-\n_1) \overline{f}(n_0+\m_1) e( \Xi(\n_1,\m_1) - \xi_1 \m_1 )| \gg \eta^{C_1+O(1)}.$$
Writing $b_2(n) \coloneqq  f(n_0+a_1+n)$ and $b_3(n) \coloneqq  \overline{f}(n_0+\m_1) e(-\xi_1 \m_1)$, we may simplify this as
$$ |\E b_1(\n_1) b_2( \m_1-\n_1) b_3(\m_1) e( \Xi(\n_1,\m_1) )| \gg \eta^{C_1+O(1)}.$$
Using the Cauchy-Schwarz inequality (Lemma \ref{cauchy-schwarz}) to eliminate the $b_3(\m_1)$ factor, we conclude that
$$ |\E b_1(\n_1) \overline{b_1}(\n'_1) b_2( \m_1-\n_1) \overline{b_2}( \m_1-\n'_1 ) e( \Xi(\n_1,\m_1) - \Xi(\n'_1,\m_1)| \gg \eta^{2C_1+O(1)}$$
where $\n'_1$ is an independent copy of $\n_1$.  Writing $\k \coloneqq  \n_1 + \n'_1 - \m_1$, and noting from the local bilinearity of $\Xi$ that
\begin{align*}
\Xi(\n_1,\m_1) - \Xi(\n'_1,\m_1) &= \Xi(\n_1-\n'_1,\m_1) \\
&= \Xi(\n_1-\n'_1, \n_1+\n'_1-\k) \\
&= \Xi(\n_1,\n_1) - \Xi(\n'_1,\n'_1) + \{ \n_1, \n'_1 \} \\ & \qquad \qquad- \Xi(\n_1,\k) + \Xi(\n'_1,\k)
\end{align*}
we conclude that
$$ |\E b_3( \n_1, \k ) b_4( \n'_1, \k) e( \{ \n_1,\n'_1\} )| \gg \eta^{2C_1+O(1)},$$
where $b_3, b_4: \Z/p\Z \times \Z/p\Z \to \C$ are the $1$-bounded functions
$$ b_3(n_1,k) \coloneqq  b_1(n_1) \overline{b_2}( k - n_1 ) e( \Xi(n_1,n_1) - \Xi(n_1,k) )$$
and
$$ b_4(n'_1,k) \coloneqq  \overline{b_1}(n'_1) b_2( k - n'_1 ) e( -\Xi(n'_1,n'_1) + \Xi(n'_1,k) ).$$
For fixed $\n_1, \n'_1$, we see from Lemma \ref{ati} that $\k$ differs from $\m_1$ in total variation by $O( \eta^{100C_1} )$, and hence
$$ |\E b_3( \n_1, \m_1 ) b_4( \n'_1, \m_1) e( \{ \n_1,\n'_1 \} )| \gg \eta^{2C_1+O(1)}.$$
By the pigeonhole principle, we may thus find $m_1 \in \Z/p\Z$ such that
$$ |\E b_3( \n_1, m_1 ) b_4( \n'_1, m_1) e( \{\n_1,\n'_1\} )| \gg \eta^{2C_1+O(1)}.$$
Using Cauchy-Schwarz (Lemma \ref{cauchy-schwarz}) to eliminate $b_4(\n'_1,m_1)$, and using the local bilinearity of $\{,\}$, we conclude that
$$ |\E b_3( \n_1, m_1 ) \overline{b_3}( \l_1, m_1 ) e( \{\n_1-\l_1,\n'_1\})| \gg \eta^{4C_1+O(1)}$$
where $\l_1$ is an independent copy of $\n_1$; using a further application of Cauchy-Schwarz (Lemma \ref{cauchy-schwarz}) to eliminate $b_3( \n_1, m_1 ) \overline{b_3}( \l_1, m_1 )$, we conclude that
$$ |\E e( \{\n_1-\l_1,\n'_1-\l'_1\})| \gg \eta^{8C_1+O(1)}$$
where $\l'_1$ is an independent copy of $\n'_1$ (thus $\n_1, \n'_1, \l_1, \l'_1$ are jointly independent and drawn regularly from $B(S_1,\rho_6)$).  In particular, by the pigeonhole principle one can find $l_1, l'_1 \in B(S_1,\rho_6)$ such that
$$ |\E e( \{\n_1-l_1,\n'_1-l'_1\})| \gg \eta^{8C_1+O(1)}.$$
By local bilinearity, one can rewrite $\{\n_1-l_1,\n'_1-l'_1\}$ as $\{\n_1,\n'_1\}$ plus locally linear functions of $\n_1$ and $\n'_1$.
The claim now follows from Proposition \ref{large-quadratic}.
\end{proof}

\subsection{Ninth step: integrating the frequency function}

We may now finally prove Theorem \ref{locu3}.  Let the notation and hypotheses be as in that theorem, let $S_1$ and $\Xi$ be as in Theorem \ref{seventh-step}, and let $k$ be as in Theorem \ref{eighth-step}.  Thus if we let $\n_0, \n_1, \m_1$ be drawn independently and regularly from $B(S,\rho_0)$, $B(S_1, \rho_6)$, $B(S_1,\rho_5)$ respectively, we have
\begin{align}\nonumber
 \sum_{n_0,n_1} \P(\n_0=n_0,\n_1=n_1) \big| \E f(n_0+&m_1+a_1-n_1)  \overline{f}(n_0+\m_1) \times \\ & \times e( \Xi(n_1,\m_1) - \xi_1 \m_1 )\big|^2 \gg \eta^{C_1+O(1)}.\label{no-again}
\end{align}
Now let $\n_2, \m_2$ be drawn independently and regularly from the Bohr sets $B(S_1,\rho_9), B(S_1,\rho_{10})$ respectively, independently of all previous random variables.  By Lemma \ref{ati}, we may replace $\n_1, \m_1$ by $\n_1 + 2k\n_2$ and $\m_1 + 2k\m_2$ in \eqref{no-again}, leading to
\begin{align*}
 &\sum_{n_0,n_1,n_2} \P(\n_0 = n_0,\dots, \n_2 = n_2) \big| \E f(n_0+\m_1+2k\m_2+a_1-n_1-2kn_2)\times \\ & \times \overline{f}(n_0+\m_1+2k\m_2) e( \Xi(n_1+2kn_2,\m_1+2k\m_2) - \xi_1 (\m_1+2k\m_2) )\big|^2 \\ & \qquad\qquad\qquad\qquad\qquad\qquad\qquad\qquad\qquad\qquad\qquad\qquad \gg \eta^{C_1+O(1)}.
\end{align*}
Thus we may find $n_1 \in B(S_1,\rho_6)$, $m_1 \in B(S_1,\rho_5)$ such that
\begin{align*}
& \sum_{n_0,n_2} \P(\n_0=n_0,\n_2=n_2) \big| \E f(n_0+m_1+2k\m_2+a_1-n_1-2kn_2) \times \\ & \overline{f}(n_0+m_1+2k\m_2) e( \Xi(n_1+2kn_2,m_1+2k\m_2) - \xi_1 (m_1+2k\m_2) ) \big|^2 \\ & \qquad\qquad\qquad\qquad\qquad\qquad\qquad\qquad\qquad\qquad\qquad\qquad \gg \eta^{C_1+O(1)},
\end{align*}
which we can simplify slightly as
\begin{align*}
& \sum_{n_0,n_2} \P(\n_0=n_0, \n_2=n_2) \big| \E f(n_0+2k\m_2+a_2-2kn_2) \times \\  & \times \overline{f}(n_0+m_1+2k\m_2) e( \Xi(n_1+2kn_2,m_1+2k\m_2) - 2k \xi_1 \m_2 ) \big|^2 \\ & \qquad\qquad\qquad\qquad\qquad\qquad\qquad\qquad\qquad\qquad\qquad\qquad\gg \eta^{C_1+O(1)}
\end{align*}
where $a_2 \coloneqq  a_1 + m_1 - n_1$; since $a_1 \in B(S,4\rho_2)$, $m_1 \in B(S_1, \rho_5)$, $n_1 \in B(S_1,\rho_6)$, we have $a_2 \in B(S,5\rho_2)$.  By the local bilinearity of $\Xi$, we have
\begin{align*}
\Xi(& n_1+2kn_2, m_1+2k\m_2)\\  &= \Xi(n_1,m_1) + 2k \Xi(n_2,m_1) + 2k \Xi(n_1,\m_2) + 4k^2 \Xi(n_2,\m_2) \\
&= \Xi(n_1,m_1) + 2k \Xi(n_2,m_1) + 2k \Xi(n_1,\m_2) + 2k^2 \Xi(n_2+\m_2,n_2+\m_2) \\
&\quad - 2k^2 \Xi(n_2,n_2) - 2k^2 \Xi(\m_2,\m_2) + 2k^2 \{ n_2,\m_2 \} 
\end{align*}
and so we have
\begin{align*}
\sum_{n_0,n_2} \P(\n_0=n_0, \n_2=n_2) | \E F( n_0, n_2 - \m_2 ) G( n_0, \m_2 ) & e( 2k^2 \{ n_2, \m_2 \} ) |^2 \\ &  \gg \eta^{C_1+O(1)}\end{align*}
where
\begin{equation}\label{fnm}
 F(n, m) \coloneqq  f(n + a_2 - 2km) e( - k^2 \Xi( m, m ) )
\end{equation}
and
$$ G(n, m) \coloneqq  \overline{f}(n + m_1 + 2km) e( 2k \Xi(n_1, m) - 2k^2 \Xi(m,m) - 2k \xi_1 m ).$$
By Theorem \ref{eighth-step}, one has $\| k \{ \n_2, \m_2 \} \|_{\R/\Z} \ll \eta^{100 C_1}$, and thus
$$
\sum_{n_0,n_2} \P(\n_0=n_0, \n_2=n_2) | \E F( n_0, n_2 - \m_2 ) G( n_0, \m_2 ) |^2 \gg \eta^{C_1+O(1)}.$$
By boundedness of the expectation, this implies that
$$
\sum_{n_0,n_2} \P(\n_0=n_0, \n_2=n_2) | \E( F( n_0, n_2 - \m_2 ) G( n_0, \m_2 ) | \gg \eta^{C_1+O(1)}$$
and thus
$$
|\E F( \n_0, \n_2 - \m_2 ) G( \n_0, \m_2 ) H( \n_0, \n_2 )| \gg \eta^{C_1+O(1)}$$
for some $1$-bounded function $H: \Z/p\Z \times \Z/p\Z \to \C$.  By Cauchy-Schwarz (Lemma \ref{cauchy-schwarz}), we thus have
$$
|\E F( \n_0, \n_2 - \m_2 ) G( \n_0, \m_2 ) \overline{F}( \n_0, \n_2 - \m'_2 ) \overline{G}( \n_0, \m_2 )| \gg \eta^{2C_1+O(1)}$$
where $\m'_2$ is an independent copy of $\m_2$; by a second application of Cauchy-Schwarz (Lemma \ref{cauchy-schwarz}), we then have
$$
|\E F( \n_0, \n_2 - \m_2 ) \overline{F}( \n_0, \n_2 - \m'_2 ) \overline{F}( \n_0, \n'_2 - \m_2 ) F( \n_0, \n'_2 - \m'_2 ) | \gg \eta^{4C_1+O(1)}$$
where $\n'_2$ is an independent copy of $\n_2$.  Since the distributions of $\m_2, \m'_2$ are symmetric, we thus have
$$
|\E F( \n_0, \n_2 + \m_2 ) \overline{F}( \n_0, \n_2 + \m'_2 ) \overline{F}( \n_0, \n'_2 + \m_2 ) F( \n_0, \n'_2 + \m'_2 ) | \gg \eta^{4C_1+O(1)}.$$
In particular, with probability $\gg \eta^{4C_1+O(1)}$, the random variable $\n_0$ attains a value $n_0$ for which
\begin{equation}\label{event}
|\E F( n_0, \n_2 + \m_2 ) \overline{F}( n_0, \n_2 + \m'_2 ) \overline{F}( n_0, \n'_2 + \m_2 ) F( n_0, \n'_2 + \m'_2 ) | \gg \eta^{4C_1+O(1)}.
\end{equation}
If $n_0$ is such that \eqref{event} holds, then we may apply Theorem \ref{locu2} and conclude that there exists a frequency $\beta(n_0) \in \Z/p\Z$ such that
$$ |\sum_{n_2} \P(\n_2=n_2) \E( F( n_0, n_2 + \m_2 ) e( - \beta(n_0) \m_2 )| \gg \eta^{2C_1+O(1)}$$
and thus (defining $\beta(n_0)$ arbitrarily if \eqref{event} does not hold),
$$ \sum_{n_0,n_2} \P(\n_0=n_0, \n_2=n_2) |\E F( n_0, n_2 + \m_2 ) e( - \beta(n_0) \m_2 )| \gg \eta^{6C_1+O(1)}$$
and hence there exists $n_2 \in B(S_1,\rho_9)$ with
$$ \sum_{n_0} \P(\n_0=n_0) |\E( F( n_0, n_2 + \m_2 ) e( - \beta(n_0) \m_2 )| \gg \eta^{6C_1+O(1)}.$$
Applying \eqref{fnm}, we conclude that
\begin{align*} \sum_{n_0} \P(\n_0=n_0) \big|\E( f(n_0 + a_3 - 2k\m_2 ) e( -k^2 \Xi(\m_2,\m_2) - & \beta(n_0) \m_2 )\big| \\ &  \gg \eta^{6C_1+O(1)}\end{align*}
where $a_3 \coloneqq  a_2 - 2kn_2$; since $a_2 \in B(S,5\rho_2)$, $n_2 \in B(S_1,\rho_9)$, and $k = O( \exp(K^{O(C_1)}) )$, we have $a_3 \in B(S,6\rho_2)$.  In particular, by Lemma \ref{ati}, $\n_0$ and $\n_0+a_3$ differ in total variation by $O( \eta^{100C_1+O(1)} )$, and thus
$$ \sum_{n_0} \P(\n_0=n_0) |\E f(n_0 - 2k\m_2 ) e( -k^2 \Xi(\m_2,\m_2) - \beta(n_0) \m_2 )| \gg \eta^{6C_1+O(1)}.$$
Theorem \ref{locu3} then follows after a change of variables, noting that the map $\m_2 \mapsto \Xi(\m_2,\m_2)$ is locally quadratic on $B(S_1,\rho_9)$.

\end{document}